\documentclass[a4paper]{scrartcl}

\usepackage[utf8]{inputenc}
\usepackage[T1]{fontenc}
\usepackage[english]{babel}
\usepackage{amsmath,amssymb,amsfonts,siunitx,commath}
\usepackage{amsthm}
\usepackage{tikz,url}
\usepackage{geometry}
\usepackage{mathtools}
\mathtoolsset{showonlyrefs}
\usepackage{subcaption}
\usepackage{algorithm}
\usepackage{algpseudocode}
\usepackage{hyperref}
\usepackage{enumerate} 

\DeclareSIUnit\year{a}
\DeclareMathOperator*{\argmax}{arg\,max}
\DeclareMathOperator*{\argmin}{arg\,min}
\DeclareMathOperator*{\esssup}{ess\,sup}
\begin{document}


\begin{titlepage}
	\newgeometry{margin = 2.5cm}
	\thispagestyle{empty}
	\begin{tikzpicture}[remember picture,overlay]%
		\node [anchor=north west, inner xsep=0pt, inner ysep=2.1cm] at (current page.north west) 
		{%
			\includegraphics[height=16mm]{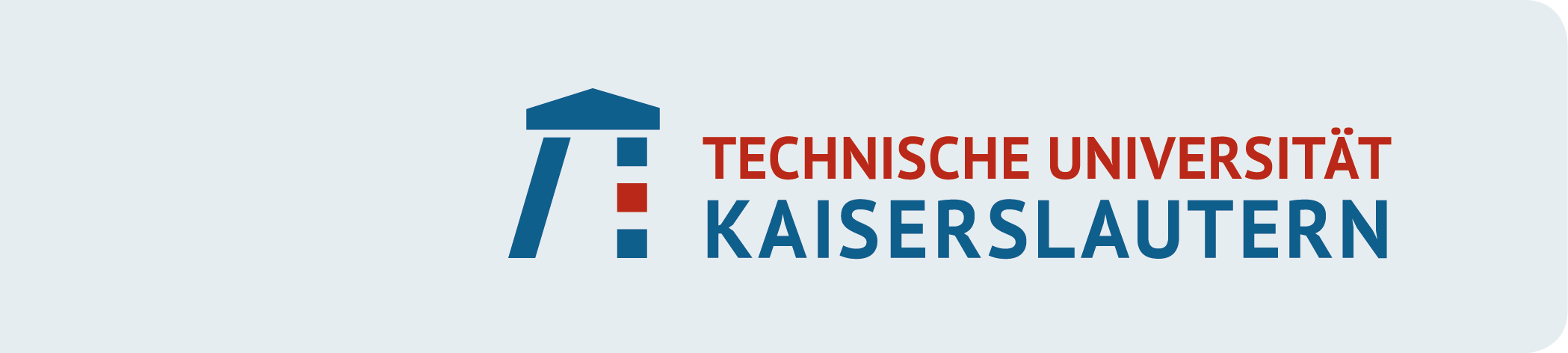}
		};%
	\end{tikzpicture}
	\vspace{1.5cm}
	\begin{center}
		{\Huge \textbf{Infinity-Laplacians on  Scalar- and Vector-Valued Functions}\\
		\textbf{and}\\
		\bigskip
		\textbf{Optimal Lipschitz Extensions on Graphs}}
		
		\vspace{2cm}
		
		{\LARGE \textbf{Bachelor Thesis}}\\
		{\Large {--verbesserte Version--}}
		
		\vspace{2cm}
		
		{\Large University of Kaiserslautern
		
		Department of Mathematics}
		
		\vspace{2cm}
		
		{\Large Johannes Hertrich}
		
		\vspace{2.5cm}
		
		{\Large supervised by}
		
		\vspace{0.25cm}
		
		{\Large Prof.\! Dr.\! Gabriele Steidl\\
		Sebastian Neumayer}
		
	\end{center}
	\vfill
	{\flushright \Large Kaiserslautern, Submitted on 04.\,October 2018}
	
	\thispagestyle{empty}
	\restoregeometry
\end{titlepage}


\newcommand{\R}{\mathbb{R}}
\newcommand{\C}{\mathbb{C}}
\newcommand{\Z}{\mathbb{Z}}
\newcommand{\Nat}{\mathbb{N}}
\newcommand{\Knoten}{\mathcal{H}}
\newcommand{\infLaplace}{\mathcal{L}_{w ,\infty}}
\newcommand\pLaplace[1]{\mathcal{L}_{w ,#1}}
\newcommand{\Patch}{P}
\newcommand{\Norm}[2]{\norm{#1}_{#2}}
\newcommand{\Neighbourhood}{\mathcal{N}}
\newcommand\inner[2]{\left\langle #1, #2 \right\rangle}
\newcommand\diverg[1]{\mathrm{div}\left( #1 \right)}
\newcommand{\todo}[1]{\textcolor{red}{\textbf{TODO:} #1}}
\newcommand\diam[1]{\mathrm{diam}(#1)}
\newcommand\prox{\mathrm{prox}}
\newcommand\vecmax{\mathrm{vecmax}}

\theoremstyle{plain}
\newtheorem{lemma}{Lemma}[section]
\newtheorem{theorem}[lemma]{Theorem}

\theoremstyle{definition}
\newtheorem{definition}[lemma]{Definition}
\newtheorem{example}[lemma]{Example}
\newtheorem*{notation}{Notation}
\newtheorem{remark}[lemma]{Remark}

\null
\thispagestyle{empty}
\newpage
\tableofcontents
\thispagestyle{empty}

\newpage

\section{Introduction}

Extending functions from boundary values plays an important role in various applications. In this thesis, we are interested in image inpainting. We consider discrete and continuous formulations of the problem based on $p$-Laplacians, in particular for $p=\infty$ and tight Lipschitz extensions. The thesis gives an overview of the existing theory and provides some novel results on the approximation of tight Lipschitz extensions for vector-valued functions.\\

Our work is structured in two parts. In Section \ref{sec:plap_inflap} and Section \ref{sec:vec-valLip} we consider the problem to extend Lipschitz continuous boundary values $g\colon \partial\Omega \to \R^m$ to a function $f\colon \overline{\Omega} \to \R^m$ with $f=g$ on $\partial \Omega$ for an open region $\Omega\subset \R^d$ such that $f$ has the same Lipschitz constant as $g$. Kirszbrauns theorem (Theorem \ref{thm:kirszbraun}, see \cite{lipLectures}) says that such an extension exists for all Lipschitz continuous boundary values $g$. But there are easy examples that such an extension is in general not unique (Lemma \ref{lem:forts}, see \cite{aronsson, lindqvist}). In Section \ref{sec:plap_inflap} we limit ourself to the case of scalar function, i.e. to the case $m=1$. We introduce some stronger optimality criteria to be the ``best'' extension of Lipschitz continuous boundary values than preserving the Lipschitz constant. This criteria is called absolute minimality and was introduced first by Aronsson in \cite{aronsson}. Following the line of \cite{lindqvist} we characterize absolute minimality by the so called $\infty$-Laplace equation. We show existence and cite uniqueness of its solutions. We call the solutions $\infty$-harmonic functions.
In Section \ref{sec:vec-valLip} we consider a canonical generalization of absolute minimality to the case $m>1$. We show some parallels to the case $m=1$ and that uniqueness fails. Therefore, we define one more stronger optimality criteria from \cite{sheffield2012vector}, called tightness. But we can neither show existence nor uniqueness of functions fulfilling this criteria. On the other hand we do not know any negative examples regarding existence or uniqueness.\\

In the second part consisting of Section \ref{sec:opt_Lip_ext_graphs} and Section \ref{sec:numexamples} we consider a discrete formulation of tight extensions on graphs. On a connected weighted graph $G=(V,E,\omega)$ and a subset $U\subset V$ we extend a functions $g\colon U\to \R^m$ to a function $f\colon V\to \R^m$ such that $f=g$ on $U$. In Section \ref{sec:opt_Lip_ext_graphs} we show existence and uniqueness of tight extensions on graphs based on \cite{sheffield2012vector}. Similar as in the continuous case we can characterize in the case $m=1$ tightness by a discrete $\infty$-Laplace equation. This leads to an approximation algorithm from \cite{ETT} of the tight extension of $g$. We implemented this algorithm in \cite{Fachpraktikum} in collaboration with Benedict Grevelhörster within the Fachpraktikum. In the case $m>1$ we show based on \cite{sheffield2012vector} that the tight extension is the limit of minimizers of some energy functionals $I_p$ as $p\to \infty$. We introduce a new algorithm to minimize these energy functionals $I_p$ to approximate the tight extension.
We implemented this minimization algorithm and in Section \ref{sec:numexamples} we apply these methods to the problem of image inpainting using the methods from \cite{Fachpraktikum}. 


\section{Preliminaries}

\subsection{Notations}

In the following we will write for $x\in \R^d$ that $\abs{x}\coloneqq \Norm{x}{2}$.\\
Let $\Omega\subset \R^d$ be an open and bounded domain. For a Norm $\Norm{\cdot}{}$ we use the notation $L^p_{\Norm{\cdot}{}}(\Omega,\R^m)\coloneqq \{ f\colon \Omega\to \R^m: \Norm{f}{}\in L^p(\Omega)\}$. For $f\in L^p_{\Norm{\cdot}{}}(\Omega,\R^m)$ we write $\Norm{f}{L^p_{\Norm{\cdot}{}}(\Omega,\R^m)}\coloneqq \Norm{\Norm{f}{}}{L^p(\Omega)}$. As in the scalar case $L^p_{\Norm{\cdot}{}}(\Omega,\R^m)$ is a reflexive Banach space. For $\Norm{\cdot}{}=\abs{\cdot}$ we write $L^p_{\abs{\cdot}}(\Omega,\R^m)=L^p(\Omega,\R^m)$\\
For arbitrary $D\subseteq \R^d$ we denote the set of Hölder continuous functions on $D$ to the exponent $\alpha \in [0,1]$ by 
\begin{equation}
C^{0,\alpha}(D)\coloneqq \{f \in C(D): \abs{f(x)-f(y)}\leq C\abs{x-y}^\alpha \text{ for some }C\in\R\text{ and for all } x,y\in D\}.
\end{equation}
In particular, for $\alpha =1$ the set $C^{0,1}(D)$ is the set of the Lipschitz continuous functions on $D$. We denote the set of the Lipschitz continuous functions with Lipschitz constant $L>0$ by
\begin{equation}
C^{0,1}_L(D)\coloneqq \{f\in C(D): \abs{f(x)-f(y)}\leq L\abs{x-y} \text{ for all } x,y\in D \}\subset C^{0,1}(D).
\end{equation}


\subsection{Sobolev spaces}

Since we will use Soblolev spaces later we give a short introduction. We use the introduction from \cite[Chapter 1.27]{altLinFuAna}. For convenience we write it down again.\\

To define Sobolev spaces we first consider completions of metric spaces.\\

For a metric space $(X,d)$ we define set of Cauchy sequences in $Y$ by
\begin{equation}
Y\coloneqq \{ (x_j)_{j\in \Nat}\subseteq Y: (x_j)_j \text{ is a Cauchy sequence}\}.
\end{equation}
We consider the equivalence relation $\sim$ on $Y$ defined by
\begin{equation}
(x_j)_j\sim (y_j)_j \text{ if and only if } d(x_j,y_j) \to 0 \text{ as } j\to \infty
\end{equation}
for $(x_j)_j,(y_j)_j\in Y$. Now we define $\tilde{X}=Y/\sim$ and 
\begin{equation}
\tilde{d}((x_j)_j, (y_j)_j)=\lim\limits_{j\to \infty} d(x_j,y_j).
\end{equation}
It is easy to check that $\tilde{d}$ is well defined on $Y$ and that $\tilde{d}((w_j)_j,(x_j)_j)=\tilde{d}((y_j)_j,(z_j)_j)$ if $(w_j)_j\sim (y_j)_j$ and $(x_j)_j\sim (z_j)_j$. Hence $\tilde{d}$ is well defined on $\tilde{X}$.\\
Further $(\tilde{X},\tilde{d})$ is a complete metric space (see \cite[Chapter 0.24]{altLinFuAna}). We call $(\tilde{X},\tilde{d})$ the \emph{completion} of $(X,d)$.\\

Let $\Omega\subseteq \R^d$ be an open domain, $m\geq 0$ and $1\leq p\leq \infty$. We consider the normed vector space
\begin{equation}
X\coloneqq \{ f\in C^\infty(\Omega) : \Norm{f}{X}<\infty \} \text{ with } \Norm{f}{X}\coloneqq \sum_{\abs{s}\leq m} \Norm{\partial^s f}{L^p(\Omega)}.
\end{equation}
We consider the completion $\tilde{X}$ of $X$. Since for $(f_j)_j\in \tilde{X}$ the sequences $(\partial^s f_j)_j$ are Cauchy sequences in $L^p(\Omega)$ there exist some $f^{(s)}\in L^p(\Omega)$ with $\partial^s f_j \to f^{(s)}$ in $L^p(\Omega)$. Due to partial integration we have for all $h\in C^\infty_c(\Omega)$ that
\begin{equation}
\int_\Omega (\partial^s h) f_j dx=(-1)^{\abs{s}} \int_\Omega h \partial^s f_j dx.
\end{equation}
Using Hölder's inequality this leads to
\begin{equation}\label{eq:sob_char}
\int_\Omega (\partial^s h) f^{(0)} dx=(-1)^{\abs{s}} \int_\Omega h f^{(s)} dx.
\end{equation}
Now we define the \emph{Sobolev space} of order $m$ with the exponent $p$ by
\begin{equation}
W^{m,p}(\Omega)\coloneqq\{ f\in L^p(\Omega): \text{ For } \abs{s}\leq m \text{ there exists } f^{(s)} \text{ such that } f^{(0)}=f \text{ and } \eqref{eq:sob_char}\}.
\end{equation}
The $f^{(s)}$ are called \emph{weak derivatives} of $f$. It can be shown that these $f^{(s)}$ are unique. We directly get that $W^{m,p}(\Omega)$ is a vector space. We equip $W^{m,p}(\Omega)$ with the norm
\begin{equation}
\Norm{f}{W^{m,p}(\Omega)}\coloneqq \sum_{\abs{s}\leq m} \Norm{f^{(s)}}{L^p(\Omega)}.
\end{equation}
It can be shown that $W^{m,p}(\Omega)$ characterizes $\tilde{X}$ through the bijective isometry 
\begin{equation}
J\colon \tilde{X} \to W^{m,p}(\Omega) \text{ defined by }J((f_j)_j)=\lim\limits_{j \to \infty} f_j,
\end{equation}
where the limit is the limit in $L^p(\Omega)$ (see \cite[Chapter 1.27]{altLinFuAna}). This shows the completeness of $W^{m,p}(\Omega)$. Hence $W^{m,p}(\Omega)$ is a Banach space. Further it can be shown that for $1<p<\infty$ the Banach space $W^{m,p}(\Omega)$ is reflexive (see \cite[Example 6.11 (3)]{altLinFuAna}). \\

We define the closed subspace of Sobolev functions vanishing at the boundary $W^{m,p}_0(\Omega) \subset W^{m,p}(\Omega)$ by
\begin{equation}
\begin{aligned}
W^{m,p}_0(\Omega)\coloneqq \{ f\in W^{m,p}(\Omega):&\text{ There exist } (f_j)_j \subset C^\infty_c(\Omega) \\
&\text{ with } \Norm{f_j-f}{W^{m,p}(\Omega)}\to 0 \text{ as } j\to \infty \}.
\end{aligned}
\end{equation}

\subsection{Preliminary results}

We give some definitions and cite some results, which we use later.

\begin{lemma}[Morrey's inequality]\label{lem:morrey}
Let $p>d$ and let $\Omega\subset \R^d$ be open, connected and bounded. Then for $v\in W^{1,p}_0(\Omega)$ it holds
\begin{equation}
\abs{v(x)-v(y)} \leq \frac{2pd}{p-d}\abs{x-y}^{1-\frac{d}{p}}\Norm{\nabla v}{L^p(\Omega)}.
\end{equation}
Further $v$ can be redefined in a set of measure zero and extended to the boundary, such that $v\in C^{1-\frac{d}{p}}(\overline{\Omega})$ and $v\mid_{\partial \Omega}=0$.
\end{lemma}
\begin{proof}
see \cite[Lemma 3]{lindqvist}.
\end{proof}

\begin{lemma}[Poincaré inequality]\label{lem:poincare}
Let $1\leq p <\infty$ and let $\Omega\subset \R^d$ be open and bounded. Then it holds for all $u\in W_0^{1,p}(\Omega)$
\begin{equation}
\Norm{u}{L^p(\Omega)}\leq C \Norm{\nabla u}{L^p(\Omega)}
\end{equation}
for some $0<C\in \R$ depending on $\Omega$ and $p$.
\end{lemma}
\begin{proof}
see \cite[U8.10]{altLinFuAna}.
\end{proof}

\begin{theorem}\label{thm:rademacher}
Let $\Omega\subset \R^d$ open. Then $C^{0,1}(\overline{\Omega})\subset W^{1,\infty}(\Omega)$ and if $f\in C^{0,1}_L(\Omega)$ then it holds $\Norm{\nabla f}{L^\infty(\Omega)}\leq L$.
\end{theorem}
\begin{proof}
See \cite[Satz 5.23]{dobrowolski}.
\end{proof}

\begin{lemma}\label{lem:Lpschachtel}
Let $(\Omega,\mathfrak{A},\mu)$ be a probability space, $1\leq p \leq q \leq \infty$. Then it holds ${L^q(\Omega)\subset L^p(\Omega)}$ and we have for all $f\in L^q(\Omega)$ that $\Norm{f}{L^p(\Omega)}\leq \Norm{f}{L^q(\Omega)}$.
\end{lemma}
\begin{proof}
Case 1: $q=\infty$. Let $f\in L^\infty(\Omega)$. Then it holds 
\begin{equation}
\int_\Omega \abs{f}^p dx\leq \int_\Omega \Norm{f}{L^\infty(\Omega)}^p dx\leq \Norm{f}{L^\infty(\Omega)}^p<\infty.
\end{equation}
Thus we get $f\in L^p(\Omega)$ and $\Norm{f}{L^p(\Omega)}\leq \Norm{f}{L^\infty(\Omega)}$.\\
Case 2: $q<\infty$. Let $f\in L^q(\Omega)$. Then we get using Hölder's inequality:
\begin{equation}
\int_\Omega \abs{f}^p dx\leq \Norm{\abs{f}^p}{L^\frac{q}{p}}\Norm{1}{L^\frac{q}{q-p}}= \Norm{f}{L^q(\Omega)}^p<\infty.
\end{equation}
Thus we get $f\in L^p(\Omega)$ and $\Norm{f}{L^p(\Omega)}\leq \Norm{f}{L^q(\Omega)}$.
\end{proof}

\begin{definition}
Let $f\colon \R^d\to \R\cup \{+\infty\}$ be proper, convex and lower semi-continuous. We define the \emph{subdifferential} by 
\begin{equation}
\partial f(x_0)=\{ p\in \R^d: f(x)-f(x_0)\geq \inner{p}{x-x_0} \text{ for all } x\in \R^d\}.
\end{equation}
We call the elements of $\partial f(x_0)$ \emph{subgradients}. By \emph{Fermat's rule} we have that $x\in \R^d$ is a global minimizer of $f$ if and only if $0\in\partial f(x)$.\\
For $\lambda>0$ the \emph{proximal map} is defined by
\begin{equation}
\prox_{\lambda f} (x) = \argmin_{y\in \R^d} \left\{\frac{1}{2 \lambda} \Norm{x-y}{2}^2 +f(y) \right\}.
\end{equation}
For more details see \cite{BSS_prox_op}.
\end{definition}

\begin{theorem}[Chain rule for subdifferentials]\label{thm:chain_subdiff}
Let H be a Hilbert space and $f\colon H\to \R$ be continuous and convex. Let $\phi\colon \R \to \R\cup\{+\infty\}$ be lower semicontinuous, convex and increasing on the image $\mathrm{Im}(f)$. Suppose that $\mathrm{ri}\{\mathrm{Im}(f)+\R_{>0}\}\cap \mathrm{dom}\phi \neq \emptyset$. Let $x\in H$ such that $f(x)\in\mathrm{dom}\phi$. Then it holds
\begin{equation}
\partial (\phi \circ f)(x)=\{\alpha y:\alpha\in \partial \phi(f(x)),y\in\partial f(x)\}
\end{equation}
\end{theorem}
\begin{proof}
See \cite[Corollary 16.7.2]{bauschke-combettes}.
\end{proof}

\begin{theorem}[Moreau decomposition]\label{thm:moreau}
Let $f\colon \R^d\to \R\cup\{+\infty\}$ be proper, convex and lower semicontinuous and $\lambda>0$. Then we have
\begin{equation}
\prox_{\lambda f}(x)+\prox_{\lambda f^*(\lambda^{-1} \cdot)}(x)=x
\end{equation}
\end{theorem}
\begin{proof}
See \cite{BSS_prox_op}.
\end{proof}

\begin{theorem}\label{thm:subdiff_norm}
Let $\Norm{\cdot}{}$ be an arbitrary norm on $\R^d$ and let $\Norm{\cdot}{*}$ be its dual norm. Then the subdifferential of $\Norm{\cdot}{}$ is given by
\begin{equation}
\partial \Norm{\cdot}{}(x)=
\begin{cases}
\{p:\Norm{p}{*}=1, \inner{p}{x}=\Norm{x}{}\},&$if $x\neq 0,\\
B_{\Norm{\cdot}{*}}(0,1),&$if $x=0.
\end{cases}
\end{equation}
\end{theorem}
\begin{proof}
See \cite{theory_ext}.
\end{proof}


\section{The variational $p$-Laplacian and $\infty$-Laplacian}\label{sec:plap_inflap}
This section is based on \cite{lindqvist}.

In the sequel let $\Omega \subset \R^d$ be open, connected and bounded, $g\in C(\overline{\Omega})\cap W^{1,p}(\Omega)$ and 
\begin{equation}
{U_p=\{u\in C(\overline{\Omega})\cap W^{1,p}(\Omega): u\mid_{\partial \Omega}=g\mid_{\partial \Omega}\}}={g+\{u\in C(\overline{\Omega})\cap W^{1,p}(\Omega): u\mid_{\partial \Omega}=0\}}
\end{equation}
for $1< p\leq \infty$. Obviously, $U_p$ is convex.

\subsection{The variational $p$-Laplacian}
Let $1<p<\infty$. In the following we minimize the energy functional
\begin{equation}
J_p(u)=\frac{1}{p}\int_\Omega \abs{\nabla u(x)}^p dx
\end{equation}
for $u\in U_p$. The proofs of the following statements (Lemma \ref{lem:p-harmun}, Theorem \ref{thm:varsol_charac}, Theorem \ref{thm:p-harmex}) follow the lines of \cite[Theorem 4]{lindqvist}, but we split the proof and add some details.
\begin{lemma}\label{lem:p-harmun}
The functional $J_p$ is strictly convex on $U_p$.
In particular, minimizers of $J_p$ are unique in $U_p$.
\end{lemma}
\begin{proof}
From the linearity of $\nabla$ and the convexity of $\int_\Omega \abs{\cdot}^p dx$ we get for $u,v\in U_p$ and $\lambda \in (0,1)$:
\begin{equation}\label{eq:Jpconv}
\begin{aligned}
J_p(\lambda u + (1-\lambda v))=\frac{1}{p}\int_\Omega \abs{\lambda \nabla u(x)+ (1-\lambda) \nabla v(x)}^p dx\\\leq\frac{1}{p}\int_\Omega \lambda \abs{\nabla u(x)}^p+ (1-\lambda)\abs{ \nabla v(x)}^p dx=\lambda J_p(u) +(1-\lambda) J_p(v)
\end{aligned}
\end{equation}
If we have $\nabla u \neq \nabla v$ in a set of positive measure we get by the strict convexity of $\abs{\cdot}^p$ that 
\begin{equation}
\abs{\lambda \nabla u+ (1-\lambda) \nabla v}^p<\lambda\abs{\nabla u}^p+ (1-\lambda)\abs{ \nabla v}^p
\end{equation}
in a set of positive measure. Hence we get a strict inequality in \eqref{eq:Jpconv}. Therefore, if we have no strict inequality in \eqref{eq:Jpconv}, we get that $u$ and $v$ differ only by a constant almost everywhere. This means $u=v$ by the definition of $U_p$. Thus the convexity of $J_p$ is strict. 
\end{proof}

We call the minimizers of $J_p$ \emph{$p$-harmonic functions}. Our goal is to find some criteria for being a $p$-harmonic function.

\begin{remark}
For $1<p<\infty$ and $u\in U_p$ it holds that
\begin{equation}
J_p(u)=\frac{1}{p}\Norm{\nabla u}{L^p(\Omega)}^p.
\end{equation}
Hence minimizing $J_p$ is equivalent to minimizing $\Norm{\nabla u}{p}$. A naive way to take the limit of $J_p$ as $p\to \infty$ would be minimizing the functional
\begin{equation}
J_\infty(u)=\Norm{\nabla u}{L^\infty(\Omega)}=\esssup_{x\in \Omega} \nabla u(x).
\end{equation}
We show in Lemma \ref{lem:forts} that the minimizers of $J_\infty$ are not unique. In this section we work out a stronger formulation of the limit of the minimizers of $J_p$ as $p\to \infty$.
\end{remark}

\begin{definition}
For $2\leq p<\infty$ the functional $\Delta_p\colon C^2(\Omega) \to C(\Omega)$ defined by
\begin{equation}
\Delta_p u= \diverg{\frac{\nabla u}{\abs{\nabla u}^{2-p}}}
\end{equation}
is called the \emph{variational $p$-Laplacian}.
\end{definition}
In particular it holds, that $\Delta_p=\Delta$ for $p=2$. Now we have the following theorem.
\begin{theorem}\label{thm:varsol_charac}
A function $u\in U_p$ is a minimizer of $J_p$ if and only if 
\begin{equation}
\int_\Omega \inner{\nabla h(x)}{\frac{\nabla u(x)}{\abs{\nabla u(x)}^{2-p}}}dx=0
\end{equation}
for all $h\in C_c^\infty(\Omega)$.
In particular, $u\in U_p \cap C^2(\Omega)$ minimizes $J_p$ if and only if $\Delta_p u=0$ for $2\leq p<\infty$.
\end{theorem}
\begin{proof}
Since $J_p$ is convex $u\in U_p$ is a minimizer of $J_p$ if and only if for all $h\in C_c^\infty(\Omega)$ the following holds true:
\begin{equation}\label{eq:ableiten}
\begin{aligned}
0 = D J_p (u)[h]&=\frac{d}{dt} J_p(u+th) = \lim\limits_{t \to 0} \frac{J_p(u+th)-J_p(u)}{t}\\
&=\frac{1}{p} \lim\limits_{t \to 0} \int_\Omega \frac{\abs{\nabla u(x)+t\nabla h(x)}^p-\abs{\nabla u(x)}^p}{t}dx.
\end{aligned}
\end{equation}
Now the integrand for $0< t \leq 1$ is bounded from above by
\begin{equation}
\begin{aligned}
\frac{\abs{\nabla u(x)+t\nabla h(x)}^p-\abs{\nabla u(x)}^p}{t}&\leq \frac{\left(\abs{\nabla u(x)}+\abs{t}\abs{\nabla h(x)}\right)^p-\abs{\nabla u(x)}^p}{\abs{t}}\\
&=\sum_{i=1}^p {p \choose i} \abs{t}^{i-1}\abs{\nabla h(x)}^i \abs{\nabla u(x)}^{p-i}\\
&\leq\sum_{i=1}^p {p \choose i}\abs{\nabla h(x)}^i \abs{\nabla u(x)}^{p-i}
\end{aligned}
\end{equation}
which is integrable since $\abs{\nabla h},\abs{\nabla u} \in L^p(\Omega)$ by the assumptions on $h$ and $u$. Similarly the integrand is bounded from below by an integrable functional. Thus we get from \eqref{eq:ableiten} using Lebesgues limit theorem and the chain rule:
\begin{equation}
\begin{aligned}
0&=\frac{1}{p} \int_\Omega \lim\limits_{t \to 0} \frac{\abs{\nabla u(x)+t\nabla h(x)}^p-\abs{\nabla u(x)}^p}{t}dx\\
&=\frac{1}{p} \int_\Omega \frac{d}{dt} \abs{\nabla u(x)+t\nabla h(x)}^p \mid_{t=0} dx\\
&=\frac{1}{p} \int_\Omega \frac{d}{dt} \left(\abs{\nabla u(x)+t\nabla h(x)}^2\right)^{\frac{p}{2}} \mid_{t=0} dx\\
&=\frac{1}{p} \int_\Omega \frac{p}{2}\left(\abs{\nabla u(x)}^2\right)^{\frac{p}{2}-1}\cdot 2 \inner{\nabla h(x)}{\nabla u(x)}dx\\
&=\int_\Omega \inner{\nabla h(x)}{\frac{\nabla u(x)}{\abs{\nabla u(x)}^{2-p}}}dx.
\end{aligned}
\end{equation}
This shows the first part of the theorem. If we have $2\leq p <\infty$ and $u\in U_p \cap C^2(\Omega)$ we get through integration by parts:
\begin{equation}
\int_\Omega \inner{\nabla h(x)}{\frac{\nabla u(x)}{\abs{\nabla u(x)}^{2-p}}}dx=-\int_\Omega h(x)\diverg{\frac{\nabla u(x)}{\abs{\nabla u(x)}^{2-p}}}dx=-\int_\Omega h(x)\Delta_p u(x)dx.
\end{equation} 
Since this holds true for every $h\in C_c^\infty(\Omega)$ we get that this is equivalent to $\Delta_p u=0$.
\end{proof}

\begin{theorem}\label{thm:p-harmex}
Let $p>d$ and $g\in C(\overline{\Omega})\cap W^{1,p}(\Omega)$. Then there exists a minimizer of $J_p$ in $U_p$.
\end{theorem}
\begin{proof}
Define $I_0= \inf_{v\in U_p} J_p (v)$. Let $(v_j)_{j\in \Nat} \subset U_p$ with $J_p v_j\leq I_0+\frac{1}{j}$. Then we have $\frac{1}{p}\Norm{\nabla v_j}{L^p(\Omega)}^p=J_p v_j\leq I_0+1$. Hence $(\Norm{\nabla v_j}{L^p(\Omega)})_j$ is bounded. Due to the Poincaré inequality (Lemma \ref{lem:poincare}) we get that 
\begin{equation}
\begin{aligned}
\Norm{v_j}{L^p(\Omega)}&\leq\Norm{v_j-g}{L^p(\Omega)}+\Norm{g}{L^p(\Omega)}\\
&\leq C\Norm{\nabla (v_j-g)}{L^p(\Omega)}+\Norm{g}{L^p(\Omega)}\\
&\leq C\Norm{\nabla v_j}{L^p(\Omega)}+C\Norm{\nabla g}{L^p(\Omega)}+\Norm{g}{L^p(\Omega)} \\
&\leq C\sup_{j\in\Nat} \Norm{\nabla v_j}{L^p(\Omega)} +C\Norm{\nabla g}{L^p(\Omega)}+\Norm{g}{L^p(\Omega)}
\end{aligned}
\end{equation}
for some $C>0$ for all $j$.
Thus $(\Norm{v_j}{W^{1,p}(\Omega)})_j$ is bounded. Since $W^{1,p}(\Omega)$ is reflexive for $1<p<\infty$, we have some $u\in W^{1,p}(\Omega)$ and a subsequence $(v_{j_k})_k$ of $(v_j)_j$ with $v_{j_k} \rightharpoonup u$  as $k\to \infty$. Because $W_0^{1,p}(\Omega)$ is closed under weak convergence and it holds $v_{j_k} - g \in W_0^{1,p}(\Omega)$ we have that $u-g\in W_0^{1,p}(\Omega)$. Since $p>d$ we get by Lemma \ref{lem:morrey} that $u-g\in C^0(\overline{\Omega})$ and $(u-g)\mid_{\partial \Omega}=0$. Thus we have $u\in W^{1,p}(\Omega)\cap C^0(\overline{\Omega})$ and $u\mid_{\partial \Omega}=g\mid_{\partial \Omega}$.\\
Further we get by the weak lower semicontinuity of the norm
\begin{equation}
J_p u=\frac{1}{p} \Norm{\nabla u}{L^p(\Omega)}^p\leq \liminf_{k\to \infty} \frac{1}{p} \Norm{\nabla v_{j_k}}{L^p(\Omega)}^p=\liminf_{k\to \infty} J_p v_{j_k}=I_0.
\end{equation}
Hence $u$ is a minimizer of  $J_p$.
\end{proof}

\subsection{The variational $\infty$-Laplacian}

We start with the definition of the $\infty$-Laplacian.
\begin{definition}
The functional $\Delta_\infty\colon C^2(\Omega) \to C(\Omega)$ defined by
\begin{equation}
\Delta_\infty u=\sum_{i=1}^d \sum_{j=1}^d \partial_i u \partial_j u \partial_i\partial_j u 
\end{equation}
is called the \emph{variational $\infty$-Laplacian}.
\end{definition}
By the chain rule we get directly 
\begin{equation}
\Delta_\infty u=\frac{1}{2} \inner{\nabla \left( \abs{\nabla u}^2 \right)}{\nabla u}.
\end{equation}
Now we can express the $p$-Laplacians by the $2$-Laplacian and the $\infty$-Laplacian. The statement is mentioned in \cite{lindqvist}. Since we did not found a proof, we give one below.
\begin{lemma}\label{lem:p-lap_darstellung}
For $u\in C^2(\Omega)$ and $2\leq p<\infty$ it holds that 
\begin{equation}
\Delta_p u=\abs{\nabla u}^{p-2} \Delta u +(p-2) \abs{\nabla u}^{p-4}\Delta_\infty u.
\end{equation}
\end{lemma}
\begin{proof}
Using the product and chain rule we get:
\begin{equation}
\begin{aligned}
\Delta_p u&= \sum_{i=1}^d \partial_i \left( \partial_i u \abs{\nabla u}^{p-2} \right) \\
&= \sum_{i=1}^d \partial_i^2 u \abs{\nabla u}^{p-2} + \sum_{i=1}^d \partial_i u \partial_i \left( \abs{\nabla u}^2 \right)^{\frac{p-2}{2}}\\
&=\abs{\nabla u}^{p-2} \Delta u + \sum_{i=1}^d \partial_i u \frac{p-2}{2} \left(\abs{\nabla u}^{p-4}\sum_{j=1}^d \partial_i (\partial_j u)^2 \right)\\
&=\abs{\nabla u}^{p-2} \Delta u + \sum_{i=1}^d \partial_i u \frac{p-2}{2} \left(\abs{\nabla u}^{p-4}\sum_{j=1}^d 2 \partial_j u \partial_i\partial_j u \right)\\
&=\abs{\nabla u}^{p-2} \Delta u +(p-2) \abs{\nabla u}^{p-4}\sum_{i=1}^d \sum_{j=1}^d \partial_i u  \partial_j u  \partial_i \partial_j u\\
&=\abs{\nabla u}^{p-2} \Delta u +(p-2) \abs{\nabla u}^{p-4}\Delta_\infty u.
\end{aligned}
\end{equation}
\end{proof}
In particular we get for $x\in \Omega$ with $\nabla u(x)\neq 0$ that
\begin{equation}
\Delta_p u=0 \Leftrightarrow (\abs{\nabla u}^{p-2} \Delta u +(p-2) \abs{\nabla u}^{p-4}\Delta_\infty u)=0 \Leftrightarrow \Delta_\infty u=  -\frac{\abs{\nabla u}^2 \Delta u}{p-2}.
\end{equation}
Hence we get the $\infty$-Laplace equation $\Delta_\infty u=0$ as the limit of the $p$-Laplace equations $\Delta_p u=0$ for $p \to \infty$. We formalize this by the following theorem:

\begin{theorem}
Let $u\in C(\overline{\Omega})\cap C^2(\Omega)$ with pointwise limit $\lim\limits_{p \to \infty} \Delta_p u(x)=0$ for all $x\in \Omega$. Then it holds $\Delta_\infty u=0$.
\end{theorem}
\begin{proof}
Let $x\in \Omega$. In the case $\nabla u(x)=0$ we get 
\begin{equation}
\Delta_\infty u(x)=\frac{1}{2} \inner{\nabla \left( \abs{\nabla u}^2 \right) (x)}{\nabla u (x)}=\frac{1}{2} \inner{\nabla \left( \abs{\nabla u}^2 \right)(x)}{0}=0.
\end{equation}
In the following, let $\nabla u (x) \neq 0$.\\
Then we get
\begin{equation}
\begin{aligned}
\Delta_p u(x)=\abs{\nabla u(x)}^{p-2} \Delta u(x) +(p-2) \abs{\nabla u (x)}^{p-4}\Delta_\infty u (x) \\
\Leftrightarrow \Delta_\infty u(x)=\frac{\Delta_p u(x)}{(p-2) \abs{\nabla u (x)}^{p-4}}-\frac{\abs{\nabla u (x)}^2 \Delta u (x)}{p-2} \to 0 \text{ as } p\to \infty.
\end{aligned}
\end{equation}
\end{proof}

It can be shown that for some fixed boundary conditions the equation $\Delta_\infty u=0$ has no solution $u\in C(\overline{\Omega})\cap C^2(\Omega)$. To give an example for such boundary values we first cite a theorem proven by Aronsson.

\begin{theorem}\label{thm:critical_point_constant}
Let $\Omega\subset \R^2$ be an open bounded domain and $u\in C^2(\Omega)$ such that $\Delta_\infty u=0$ in $\Omega$. Then, either $\nabla u (x)\neq 0$ for all $x\in \Omega$ or $u$ reduces to a constant. 
\end{theorem}
\begin{proof}
See \cite[Theorem 7]{lindqvist}. 
\end{proof} 

Now we can give an example from \cite{lindqvist} for a domain $\Omega$ and Lipschitz boundary values $g\colon \partial \Omega \to \R$ such there exists no $u\in C(\overline{\Omega})\cap C^2(\Omega)$ with $u=g$ on $\partial \Omega$ and $\Delta_\infty u=0$ on $\Omega$.

\begin{example}
Let $\Omega= \{(x,y)\in \R^2: x^2+y^2<1\}$ be the unit disc and let $g\colon \partial \Omega \to \R$ be defined by $g(x,y)=x y$. Obviously $g$ is Lipschitz continuous.

Assume that there exists $u\in C(\overline{\Omega})\cap C^2(\Omega)$ with $u=g$ on $\partial \Omega$ and $\Delta_\infty u=0$ on $\Omega$. We will show in Remark \ref{rem:uniqueness} that such an $u$ is unique. Since also $v \in C(\overline{\Omega})\cap C^2(\Omega)$ defined by $v(x,y)=u(-x,-y)$ fulfills $u=g$ on $\partial \Omega$ and $\Delta_\infty u=0$ on $\Omega$ we can conclude $u=v$. This yields $u(x,y)=u(-x,-y)$. Therefore, we have $\nabla u(0,0)=0$. By Theorem \ref{thm:critical_point_constant} we have that $u$ is constant on $\Omega$. This contradicts $u=g$ on $\partial \Omega$. Thus such a $u$ does not exists.
\end{example}

To ensure existence we are looking for some formulation of $\Delta_p u=0$ for $p \to \infty$ with less assumptions in the next section.

\subsection{Lipschitz extensions}

As introduced in \cite{aronsson} there is a close connection between the $\infty$-Laplacian and the problem of extending Lipschitz continuous functions.\\

Let $g\colon \partial \Omega \to \R$ be Lipschitz continuous and let $L$ be the smallest possible Lipschitz constant. 
In the following, we extend $g$ to $\overline{\Omega}$ preserving the Lipschitz continuity. First we show the existence of such an extension, later we will discuss the connection to the $\infty$-Laplacian.
\begin{lemma}\label{lem:forts}
Let $g\in C^{0,1}_L(\partial \Omega)$ and let $h_1(x)=\max_{z \in \partial \Omega} \{g(z)-L \abs{x - z}\}$ and $h_2(x)=\min_{z \in \partial \Omega} \{g(z) + L \abs{x - z}\}$. Then $h_i\in C^{0,1}_L(\overline{\Omega})$, $i=1,2$ are Lipschitz continuous extensions of $g$. Further, every Lipschitz continuous extension $h\in C^{0,1}_L(\overline{\Omega})$ of $g$ fulfills
\begin{equation}
h_1(x)\leq h(x) \leq h_2(x) \text{ for all } x\in \overline{\Omega}\label{eq:lipbound}
\end{equation}
\end{lemma}
The statement is mentioned in \cite{aronsson}. We check it by a short computation.
\begin{proof}
Since $\partial \Omega$ is compact the minimum and maximum in the definition of $h_i$, $i=1,2$ exists. For $x\in \partial \Omega$ we have $h_i(x)=g(x)$. For $x,y\in \overline{\Omega}$ with $h(x)\geq h(y)$ we have 
\begin{equation}
\begin{aligned}\label{eq:lipext}
\abs{h_1(x)-h_1(y)}&=h_1(x)-h_1(y)\\
&=\max_{z\in \partial \Omega} \{g(z)-L\abs{x-z}\}-\max_{w\in \partial \Omega} \{g(w)-L\abs{y-w}\}\\
&=\max_{z\in \partial \Omega}\{g(z)-L\abs{x-z}\}+\min_{w\in \partial \Omega} \{L\abs{y-w}-g(w)\}.
\end{aligned}
\end{equation}
Let $z_0 \in \argmax_{z\in \partial \Omega } \{g(z)-L\abs{x-z}\}$. Then \eqref{eq:lipext} becomes
\begin{equation}
\begin{aligned}
g(z_0)-L\abs{x-z_0}+\min_{w\in \partial \Omega} \{L\abs{y-w}-g(w)\}\leq g(z_0)-L\abs{x-z_0}+L\abs{y-z_0}-g(z_0)\\
=-L\abs{z_0-x}+L\abs{y-z_0}\leq L\abs{y-z_0-z_0-x}=L\abs{x-y}.
\end{aligned}
\end{equation} 
Hence $h_1$ is a Lipschitz continuous extension of $g$. Similarly we get the result for $h_2$. Property \eqref{eq:lipbound} follows directly by the Lipschitz continuity of $h$.
\end{proof}

In general we have $h_1 \neq h_2$. Thus the extension of $g$ to a Lipschitz continuous function on $\overline{\Omega}$ is not unique and we may search for a kind of the best extension. To formalize this we introduce for $f\in C^{0,1}(D)$ the notation 
\begin{equation}
\mu(f,D)=\sup_{x\neq y\in D} \frac{\abs{f(x)-f(y)}}{\abs{x-y}} \in [0,\infty).
\end{equation}
In other words, $\mu(f,D)$ is the smallest number, such that $f$ is Lipschitz continuous with constant $\mu(f,D)$. For  an open and bounded domain $\Omega\subset \R^d$ and Lipschitz continuous $f\colon \Omega \to \R$ we get a connection between $\mu(f,\Omega)$ and $\Norm{\nabla f}{L^\infty(\Omega)}$.

\begin{lemma}\label{lem:infNorm-lipschitz}
Let $g\in C^{0,1}_L(\partial \Omega)$ and let $f\colon \overline{\Omega} \to \R$ be Lipschitz continuous with $\Norm{\nabla f}{L^\infty(\Omega)}\leq L$ and $f\mid_{\partial \Omega}=g$. Then $\mu(f,\Omega)\leq L$ i.e. $L$ is a Lipschitz constant of $f$.
\end{lemma}
The statement is used in \cite{lindqvist}. For completeness we add a proof.
\begin{proof}
Let $x,y \in \overline{\Omega}$. Let $\gamma\colon [0,1] \to \R^d$ with $\gamma(t)=ty+(1-t)x$. Since $\gamma$ and $f$ are Lipschitz continuous also $f\circ \gamma$ is Lipschitz continuous. Hence we have by \cite[Theorem 3.3]{lipLectures} that 
\begin{equation}
(f\circ \gamma)(x)-(f\circ\gamma)(y)=\int_x^y (f\circ\gamma)'(t) dt.
\end{equation}

Case 1: $\gamma((0,1))\subset \Omega$. We have $f(y)=f(x)+\int_{[0,1]} \inner{\nabla f(\gamma (t))}{y-x}dt$. This yields $\abs{f(y)-f(x)}\leq \int_{[0,1]}\abs{\inner{\nabla f(\gamma (t))}{y-x}}dt\leq \int_{[0,1]}\Norm{\nabla f}{L^\infty(\Omega)} \abs{x-y} dt\leq L\abs{x-y}$.\\

Case 2: $\gamma((0,1))\not\subset \Omega$. Since $\partial \Omega$ is closed we have that $t_1=\min\{t\in [0,1]: \gamma(t)\in \partial \Omega \}$ and $t_2=\max\{t\in [0,1]: \gamma(t)\in \partial \Omega \}$ exist and $t_2\geq t_1$. Because $f\mid_{\partial \Omega}=g$ and $g$ is Lipschitz, we get $\abs{f(\gamma(t_1))-f(\gamma(t_2))}\leq L\abs{\gamma(t_1)-\gamma(t_2)}$. Due to case 1 we have $\abs{f(x)-f(\gamma(t_1))}\leq L\abs{x-\gamma(t_1)}$ and $\abs{f(\gamma(t_2))-f(y)}\leq L\abs{\gamma(t_2)-y}$. Thus
\begin{equation}
\begin{aligned}
\abs{f(x)-f(y)}&\leq \abs{f(x)-f(\gamma(t_1))} + \abs{f(\gamma(t_1))-f(\gamma(t_2))}+\abs{f(\gamma(t_2))-f(y)} \\
&\leq L \left(\abs{x-\gamma(t_1)}+\abs{\gamma(t_1)-\gamma(t_2)}+\abs{\gamma(t_2)-y}\right)=L\abs{x-y}.
\end{aligned}
\end{equation}
\end{proof} 

\begin{definition}\label{def:abs_min}
A Lipschitz continuous extension $f\colon \overline{\Omega} \to \R$ of $g\in C^{0,1}(\partial \Omega)$ is called \emph{absolute minimal} if we have for all open $D\subset \R^d$ that $\mu(f,\overline{D})=\mu(f,\partial D)<\infty$.
\end{definition}
Informally this means that $f$ is absolute minimal if and only if there is no subregion, where we can improve the Lipschitz constant.\\

The following theorem proven by Aronsson states a one-to-one relation of absolute minimal functions to solutions of $\Delta_\infty u=0$. We only cite the result, because as mentioned above such an $u\in C(\overline{\Omega})\cap C^2(\Omega)$ does not exist for all boundary values $g$.

\begin{theorem}
Let $u\in C^2(\Omega)$. Then $u$ is absolute minimal if and only if $\Delta_\infty u=0$ on $\Omega$.
\end{theorem}
\begin{proof} 
See \cite[Theorem 8]{aronsson}.
\end{proof}

We will show that a sequence $(u_p)_{p\in\Nat} \subset U_p$ of $p$-harmonic functions has a subsequence converging to an absolute minimal function.

\begin{theorem}\label{thm:varsolex}
Let $g\in C^{0,1}(\partial \Omega)$ with $\mu(g,\partial \Omega)=L$ and let $(u_p)_{p\in \Nat}$ a sequence of $p$-harmonic $u_p\in U_p$. Then there exists a subsequence $(u_{p_j})_j$ converging uniformly to some $u_\infty \in U_\infty$ such that  $\nabla u_{p_j}\rightharpoonup \nabla u_\infty$ weakly as $j \to \infty$ in $L^s(\Omega)$ for all $s>1$.\\
Further, we have that $\mu(u_\infty,\overline{\Omega})=L$.
\end{theorem}
\begin{definition}
A function $u_\infty\in U_\infty$ constructed as in Theorem \ref{thm:varsolex} is called a \emph{variational solution} of the $\infty$-Laplace equation $\Delta_\infty u=0$.
\end{definition}
The proof of Theorem \ref{thm:varsolex} comes from \cite[Section 3]{lindqvist}. We added some details.
\begin{proof}[Proof of Theorem \ref{thm:varsolex}]
Due to Lemma \ref{lem:forts} we can extend $g$ to some Lipschitz continuous function on $\overline{\Omega}$. We denote this extension again with $g$. By Theorem \ref{thm:rademacher} the Lipschitz continuity ensures $g\in C(\overline{\Omega})\cap W^{1,p}(\Omega)$. First we show equicontinuity and equiboundedness of $\{u_p\}$ for $p>d$. Since $u_p\in U_p$ we have that $v_p=u_p-g \in W_0^{1,p}(\Omega)$. Thus Lemma \ref{lem:morrey} yields for $p>d$
\begin{equation}
\abs{v_p(x)-v_p(y)}\leq \frac{2pd}{p-d} \abs{x-y}^{1-\frac{d}{p}}\Norm{\nabla v_p}{L^p(\Omega)}
\end{equation} 
Hence,
\begin{equation}
\begin{aligned}
\abs{u_p(x)-u_p(y)}&\leq\abs{g(x)-g(y)}+\abs{v_p(x)-v_p(y)}\\
&\leq L\abs{x-y}+\frac{2pd}{p-d} \abs{x-y}^{1-\frac{d}{p}}\Norm{\nabla v_p}{L^p(\Omega)}\\
&\leq L\abs{x-y}+\frac{2pd}{p-d} \abs{x-y}^{1-\frac{d}{p}}(\Norm{\nabla u_p}{L^p(\Omega)}+\Norm{\nabla g}{L^p(\Omega)})\\
&\leq L\abs{x-y}+\frac{2pd}{p-d} \abs{x-y}^{1-\frac{d}{p}}(2\Norm{\nabla g}{L^p(\Omega)})\\
&= L\abs{x-y}+\frac{4pd}{p-d} \abs{x-y}^{1-\frac{d}{p}}\Norm{\nabla g}{L^p(\Omega)}.
\end{aligned}
\end{equation}
Since $\left(\frac{4pd}{p-d}\right)_p$ converges to $4d$ for $p\to \infty$ the sequence is bounded and there exists some constant $C>0$ depending on $\Omega$ such that 
\begin{equation}\label{eq:absch_varsolex}
\abs{u_p(x)-u_p(y)}\leq L\abs{x-y}+C \abs{x-y}^{1-\frac{d}{p}}
\end{equation}
for all $p>d$. Therefore, $\{u_p:p\in \Nat\}\subset C(\overline{\Omega})$ is equicontinuous. Because $u_p\mid_{\partial \Omega}=g\mid_{\partial \Omega}$ we get for $x\in\Omega$ some fixed $x_0\in \partial \Omega$ 
\begin{equation}
\begin{aligned}
\abs{u_p(x)}&\leq\abs{u_p(x)-u_p(x_0)}+\abs{u_p(x_0)}\\
&\leq L\abs{x-x_0}+C \abs{x-x_0}^{1-\frac{d}{p}}+\abs{g(x_0)}\\
&\leq L\diam{\Omega}+C\diam{\Omega}^{1-\frac{d}{p}}+\abs{g(x_0)}.
\end{aligned}
\end{equation}
This yields for $\diam{\Omega}\leq 1$ that $\abs{u_p(x)}\leq L+C+\abs{g(x_0)}$. For $\diam{\Omega}>1$ we get 
\begin{equation}
\abs{u_p(x)}\leq L\diam{\Omega}+C\diam{\Omega}^{1-\frac{d}{p}}+\abs{g(x_0)}\leq L\diam{\Omega}+C\diam{\Omega}+\abs{g(x_0)}.
\end{equation}
Thus $\{u_p:p\in \Nat\}\subset C(\overline{\Omega})$ is equibounded. Now the theorem of Arzela-Ascoli yields that there is a subsequence $(u_{p_j})_j$, which converges uniformly to some $u_\infty \in C(\overline{\Omega})$ with $u_\infty \mid_{\partial \Omega}=g\mid_{\partial \Omega}$.\\
Furthermore, we have 
\begin{equation}
\begin{aligned}
\abs{u_\infty (x) -u_\infty (y)}&=\lim\limits_{j\to \infty} \abs{u_{p_j} (x) - u_{p_j} (y)}\\
&\leq \lim\limits_{j\to \infty}  L\abs{x-y}+C \abs{x-y}^{1-\frac{d}{p_j}}\\
&=L\abs{x-y}+C\abs{x-y}.
\end{aligned}
\end{equation}
Therefore, $u_\infty$ is Lipschitz continuous. In particular, it is almost everywhere differentiable by Theorem \ref{thm:rademacher}. Since $\frac{1}{\lambda(\Omega)}dx$ is a probability measure we get for $p_j>s$ using Lemma \ref{lem:Lpschachtel}
\begin{equation}
\begin{aligned}
\left(\int_\Omega \abs{\nabla u_{p_j}}^s\frac{1}{\lambda (\Omega)} dx\right)^\frac{1}{s}&\leq\left(\int_\Omega \abs{\nabla u_{p_j}}^{p_j}\frac{1}{\lambda (\Omega)} dx\right)^\frac{1}{p_j}\\
&\leq \left(\int_\Omega \abs{\nabla g}^{p_j}\frac{1}{\lambda (\Omega)} dx\right)^\frac{1}{p_j}\\
&\leq \Norm{\abs{\nabla g}}{L^\infty (\Omega)}=L,
\end{aligned}
\end{equation}
where $\lambda$ denotes the Lebesgue measure. Hence $\Norm{u_{p_j}}{L^s(\Omega)}$ and $\Norm{\abs{\nabla u_{p_j}}}{L^s(\Omega)}$ are bounded. Thus also the Sobolev norm $\Norm{u_{p_j}}{W^{1,s}(\Omega)}$ is bounded for all $j$. Since the Sobolev spaces $W^{1,p}(\Omega)$ for $1<p<\infty$ are reflexive we know that $\overline{B_r(0)}$ is weak sequentially compact for all $r>0$. Hence each subsequence of $(u_{p_j})_j$ has a subsequence $(u_{p_{j_k}})_k$ which converges weakly to some $v\in W^{1,s}(\Omega)$. The weak convergence in $W^{1,s}(\Omega)$ yields weak convergence of $(u_{p_{j_k}})_k$ to $v$ and of $(\nabla u_{p_{j_k}})_k$ to $\nabla v$ in $L^s(\Omega)$. Since weak limits are unique and the uniform convergence of $(u_{p_{j_k}})_k$ to $u_\infty$ implies weak convergence we have $v=u_\infty$ and $\nabla u_{p_{j_k}}\rightharpoonup \nabla u_\infty$ weakly in $L^s(\Omega)$.\\
Note that we have such a subsequence $(u_{p_{j_k}})_k$ for all $s>1$, which depends on $s$. Since for $p>s$ weak convergence in $L^p(\Omega)$ implies weak convergence in $L^s(\Omega)$ we can create by diagonalization a subsequence  $(u_{p_{j_k}})_k$ of $(u_{p_j})_j$ with $\nabla u_{p_{j_k}}\rightharpoonup \nabla u_\infty$ weakly in $L^s(\Omega)$ for all $1<s\leq\infty$. Thus we get by the lower semi-continuity of the norm with regard to weak convergence
\begin{equation}
\int_\Omega \abs{\nabla u_\infty}^s dx \leq \limsup\limits_{k\to \infty} \int_\Omega \abs{\nabla u_{j_k}}^s dx \leq L^s
\end{equation}
i.e. $\Norm{\nabla u_\infty}{L^s(\Omega)}\leq L$. Hence for $s \to \infty$ we get
\begin{equation}
\Norm{\nabla u_\infty}{L^\infty(\Omega)} \leq L.
\end{equation}
Thus $u_\infty$ is Lipschitz continuous with constant $L$ by Lemma \ref{lem:infNorm-lipschitz} and we have 
\begin{equation}
L=\mu(g,\partial \Omega)=\mu(u_\infty, \partial \Omega) \leq \mu(u_\infty,\overline{\Omega}) \leq L.
\end{equation}
\end{proof}

\begin{theorem}\label{thm:varsolabsmin}
Let $g\in C^{0,1}(\partial \Omega)$ with $\mu(g,\partial \Omega)=L$ and $u_\infty \in U_\infty$ be a variational solution of $\Delta_\infty u =0$. Then $u_\infty$ is absolute minimal.
\end{theorem}
\begin{proof}
See \cite[Theorem 5]{lindqvist}. Let $D$ be an open set with $\overline{D}\subset \Omega$ and $(u_{p_j})$ a subsequence of $p$-harmonic functions $u_p$ converging uniformly to $u_\infty$. Denote by $v_p$ the unique solutions of
\begin{equation}
\int_D \inner{\nabla h(x)}{\frac{\nabla v(x)}{\abs{\nabla v(x)}^{2-p}}}dx=0 \text{ for all } h\in C_c^\infty (D) \text{ s.t. } v\mid_{\partial D}=u_\infty\mid_{\partial D}.
\end{equation}
Now $(v_{p_j})_j$ has by Theorem \ref{thm:varsolex} a subsequence $(v_{p_{j_k}})_k$ converging uniformly to some $v_\infty \in U_\infty$ with $\mu (v_\infty , \partial D)=\mu (v_\infty ,\overline{D})$. We claim that it holds $v_\infty=u_\infty$ on $\overline{D}$ i.e. that $v_{p_{j_k}} \to u_\infty$ uniformly for $k\to \infty$.

First we show by contradiction that $v_p-u_p$ attains its maximum in $\overline{D}$ on $\partial D$:
Assume $\alpha = \max_{\overline{D}}(v_p - u_p)>\max_{\partial D} (v_p-u_p) = \beta$ and define the open set $G=\{ x\in D: v_p(x)-u_p(x)>\frac{\alpha+\beta}{2}\}$. By assumption $G$ is not empty and due to the continuity of $v_p-u_p$ we have for all $x\in \partial G$ that $v_p(x)-u_p(x)=\frac{\alpha+\beta}{2}$.
Since $u_p$ minimizes $J_p$ and
\begin{equation}
J_p u_p=\frac{1}{p} \left(\int_{\Omega \backslash G} \abs{\nabla u_p}^p dx + \int_{G} \abs{\nabla u_p}^p dx\right)
\end{equation}
we get that $u_p$ is a solution of
\begin{equation}
\int_G \inner{\nabla h(x)}{\frac{\nabla u(x)}{\abs{\nabla u(x)}^{2-p}}}dx=0 \text{ for all } h\in C_c^\infty (G) \text{ s.t. } u\mid_{\partial G}=u_p\mid_{\partial G}
\end{equation}
Analogously $v_p$ is also $p$-harmonic on $G$. Since $\left(u_p+\frac{\alpha+\beta}{2}\right)\mid_{\partial G}=v_p\mid_{\partial G}$ we get by the uniqueness of $p$-harmonic functions of $G$ (Lemma \ref{lem:p-harmun}) that $\left(u_p+\frac{\alpha+\beta}{2}\right)\mid_G=v_p\mid_G$. By the definition of $G$ this yields that $G$ is empty and contradicts the assumption. Thus we get by uniform convergence of the $u_p$
\begin{equation}
\max_{\overline{D}}(v_{p_{j_k}}-u_{p_{j_k}})=\max_{\partial D}(v_{p_{j_k}}-u_{p_{j_k}}) =\max_{\partial D}(u_\infty-u_{p_{j_k}})\to 0
\end{equation}
Analogously we get $\min_{\overline{D}}(v_{p_{j_k}}-u_{p_{j_k}})\to 0$. Therefore, $v_{p_{j_k}}-u_{p_{j_k}}$ converges uniformly to $0$. This yields
\begin{equation}
v_\infty-u_\infty=(v_\infty-v_{p_{j_k}})+(v_{p_{j_k}}-u_{p_{j_k}})+(u_{p_{j_k}}-u_\infty)
\end{equation}
Since all summands converge uniformly to $0$ we have that $v_\infty=u_\infty$ and $\mu (u_\infty , \partial D)=\mu (u_\infty ,\overline{D})$. Thus $u_\infty$ is absolute minimal.
\end{proof}

In particular Theorem \ref{thm:varsolex} ensures the existence of a variational solution of the $\infty$-Laplace equation, since Theorem \ref{thm:p-harmex} yields the existence of $p$-harmonic functions for $p>d$. Theorem \ref{thm:varsolabsmin} yields that this solution is absolute minimal.

\begin{theorem}\label{thm:classical_solution_is_variational_solution}
Every $u\in C(\overline{\Omega})\cap C^2(\Omega)$ with $\Delta_\infty u=0$ is a variational solution of the $\infty$-Laplace equation.
\end{theorem}
\begin{proof}
See Remark \ref{rem:uniqueness} below.
\end{proof}

In this setting it is hard to prove the uniqueness of the $u_\infty$. Therefore, we will introduce the concept of viscosity solutions in the next section.

\subsection{Viscosity solutions}
\begin{definition}
Let $d<p\leq \infty$. 
\begin{enumerate}[(i)]\label{def:viscos}
\item We call $v\in C(\Omega)$ a \emph{viscosity supersolution} of $\Delta_p v=0$ in $\Omega$ if for all $\phi \in C^2(\Omega)$ with $\phi(x_0)=v(x_0)$ in some point $x_0\in \Omega$ and $\phi(x) < v(x)$ for $x\in \Omega \backslash \{x_0\}$ it holds $\Delta_p \phi (x_0)\leq 0$.
\item Analogously we call $u\in C(\Omega)$ a \emph{viscosity subsolution} of $\Delta_p u=0$ in $\Omega$ if if for all $\psi \in C^2(\Omega)$ with $\psi(x_0)=u(x_0)$ in some point $x_0\in \Omega$ and $\psi(x) > u(x)$ for $x\in \Omega \backslash \{x_0\}$ it holds $\Delta_p \psi (x_0)\geq 0$.
\item We call $h\in C(\Omega)$ a \emph{viscosity solution} of $\Delta_p h=0$ in $\Omega$, if $h$ is a viscosity supersolution and a viscosity subsolution.
\end{enumerate}
\end{definition}

In the $C^2$-case we directly get that the definitions of classical solutions and viscosity solutions of $\Delta_p u=0$ coincide.
\begin{theorem}
For $u\in C^2(\Omega)$ and $4\leq p\leq\infty$ holds: u is a viscosity solution of $\Delta_p u=0$ if and only if $\Delta_p u=0$ in $\Omega$.
\end{theorem}
The proof comes from \cite[Proposition 9]{lindqvist}. We added the infinitesimal calculations.
\begin{proof}
``$\Rightarrow$'': Let $u\in C^2(\Omega)$ with $\Delta_p u=0$ in viscosity sense and $x_0\in \Omega$ arbitrary fixed. It is easy to check that for $h(x)=(x-x_0)^4$, holds $h(x)\geq 0$ with equality if and only if $x=x_0$ and $\Delta_p h(x_0) =0$. \\
Hence $\phi (x)=u(x)-h(x)$ touches $u$ from below exactly in $x_0$. Since $u$ is a viscosity solution we get $\Delta_p u(x_0)=\Delta_p \phi (x_0) \leq 0$. On the other hand we have that $\psi (x) = u(x) + h(x)$ touches $u$ from above in $x_0$. Thus we get $\Delta_p u(x_0)=\Delta_p \psi(x_0) \geq 0$ and can conclude that $\Delta_p u(x_0)=0$.\\
``$\Leftarrow$'' Let $u \in C^2(\Omega)$ fulfill $\Delta_p u=0$ pointwise and $\phi\in C^2(\Omega)$ with $\phi(x_0)=u(x_0)$ for some $x_0\in \Omega$ and $\phi(x)<u(x)$ for $x\in \Omega\backslash \{x_0\}$.\\
Case 1: $p=\infty$. Since $(\phi-u)$ has a local maximum at $x_0$ we have that $\nabla (\phi-u)(x_0)=0$ and that $H(\phi-u)(x_0)$ is negative semidefinite, where $H(\phi-u)$ denotes the Hessian matrix of $\phi-u$. This yields 
\begin{equation}
\begin{aligned}
\Delta_\infty \phi (x_0)&= \sum_{i,j=1}^d \partial_i \phi (x_0) \partial_j \phi (x_0) \partial_i\partial_j \phi (x_0)\\
&=\nabla \phi (x_0) H\phi (x_0) \nabla \phi(x_0)\\
&=(\nabla \phi (x_0)+ \nabla (u-\phi)(x_0)) H(\phi-u+u)(x_0) (\nabla \phi (x_0)+ \nabla (u-\phi)(x_0))\\
&=\nabla u(x_0) (H(\phi-u)(x_0)+ H u (x_0)) \nabla u(x_0)\\
&=\nabla u(x_0) H(\phi-u)(x_0) \nabla u(x_0)+ \nabla u(x_0) H u (x_0) \nabla u(x_0)\\
&=\nabla u(x_0) H(\phi-u)(x_0) \nabla u(x_0)+\Delta_\infty u (x_0)\\
&\leq \Delta_\infty u (x_0) =0.
\end{aligned}
\end{equation}
Thus $u$ is a viscosity supersolution of $\Delta_\infty u=0$.\\
Case 2: $p<\infty$. Then we have due to Lemma \ref{lem:p-lap_darstellung}
\begin{equation}\label{eq:vissuppkleinerinf}
\Delta_p \phi (x_0)=\abs{\nabla \phi (x_0)}^{p-2} \Delta \phi (x_0) + (p-2) \abs{\nabla \phi(x_0)}^{p-4} \Delta_\infty \phi(x_0).
\end{equation}
Since $(\phi-u)$ has a local maximum at $x_0$ we have that $H(\phi-u)(x_0)$ is negative semidefinite. In particular, $H(\phi-u)(x_0)$ has non positive diagonal elements. Hence $\Delta (\phi-u) (x_0)\leq 0$. Further we have that $\nabla (\phi-u)(x_0)=0$ and  therefore $\nabla \phi(x_0)=\nabla u (x_0)$. As seen above we have $\Delta_\infty \phi \leq \Delta_\infty u$ Thus \eqref{eq:vissuppkleinerinf} becomes
\begin{equation}
\begin{aligned}
\Delta_p \phi (x_0)&=\abs{\nabla u (x_0)}^{p-2} (\Delta u (x_0) + \Delta (\phi-u) (x_0)) + (p-2) \abs{\nabla u(x_0)}^{p-4} \Delta_\infty \phi(x_0)\\
&\leq \abs{\nabla u (x_0)}^{p-2} \Delta u (x_0) + (p-2) \abs{\nabla u(x_0)}^{p-4} \Delta_\infty u(x_0)\\
&= \Delta_p u(x_0)=0.
\end{aligned}
\end{equation}
Hence $u$ is a viscosity supersolution of $\Delta_p u=0$.\\
Analogously we get that $u$ is a viscosity subsolution of $\Delta_p u=0$. Therefore, $u$ is a viscosity solution.
\end{proof}

Our goal is to show that every variational solution of $\Delta_\infty u=0$ is also a viscosity solution.

\begin{lemma}\label{lem:p-harm_viscos}
Let $1<p<\infty$ and $v\in C(\Omega)\cap W^{1,p}(\Omega)$ with
\begin{equation}
\int_\Omega \inner{\abs{\nabla v}^{p-2} \nabla v}{\nabla h}dx\geq 0
\end{equation}
for all $0\leq h \in C_c^\infty(\Omega)$. Then $v$ is a viscosity supersolution of $\Delta_p v=0$ in $\Omega$.\\
In particular, if $v$ is $p$-harmonic, then it is a viscosity solution of $\Delta_p v=0$.
\end{lemma}
The proof is a modified version of the proof in \cite[Lemma 10]{lindqvist}. 
\begin{proof}
Assume $v$ is no viscosity supersolution of $\Delta_p v=0$ in $\Omega$. Then by definition there exists $x_0\in \Omega$ and $\phi\in C^2(\Omega)$ with $\phi(x_0)=v(x_0)$ and $\phi(x)<v(x)$ for all $x\in \Omega \backslash \{x_0\}$ such that $\Delta_p \phi (x_0)>0$.\\
Because $\Delta_p \phi$ is continuous we have that $\Delta_p \phi >0$ in $B_{2\rho}(x_0)$ for some $\rho >0$. Define
\begin{equation}
\psi(x)=\phi (x)+ \frac{1}{2} \min_{\partial B_{\rho}(x_0)} \{ v-\phi\}.
\end{equation}
By definition we get $\psi < v$ on $\partial B_\rho(x_0)$ and $\psi(x_0)>v(x_0)$. Consider $D_\rho=\{\psi>v\}\cap B_\rho(x_0)$. We get that $\psi=v$ on $\partial D_\rho$ and $x_0\in D_\rho$.
Since $\Delta_p \psi=\Delta_p \phi >0$ on $B_{2\rho}(x_0)$ we get for all  $0\leq h \in C_c^\infty(\Omega)$ that
\begin{equation}
-\int_{D_\rho}h(x)\Delta_p \psi (x) dx=\int_{D_\rho} \inner{\abs{\nabla \psi}^{p-2} \nabla \psi}{\nabla h}dx\leq 0.
\end{equation}
Thus we get
\begin{equation}
\int_{D_\rho} \inner{\abs{\nabla \psi}^{p-2} \nabla \psi-\abs{\nabla v}^{p-2} \nabla v}{\nabla h}dx\leq 0.
\end{equation}
Because $\psi>v$ on $D_\rho$ we get
\begin{equation}\label{eq:comparison}
\int_{D_\rho} \inner{\abs{\nabla \psi}^{p-2} \nabla \psi-\abs{\nabla v}^{p-2} \nabla v}{\nabla \psi-\nabla v}dx\leq 0.
\end{equation}
Since the sign of the $i$-th component of $\nabla \psi(x)-\nabla v(x)$ is the same as the sign of the $i$-th component of $\abs{\nabla \psi(x)}^{p-2} \nabla \psi(x)-\abs{\nabla v(x)}^{p-2} \nabla v(x)$ we get that the integrand of \eqref{eq:comparison} is nonnegative. Thus \eqref{eq:comparison} yields that $\nabla \psi -\nabla v=0$ almost everywhere. Because $\psi$ and $v$ are continuous we get that there exists a $C\in \R$ such that $\psi=v+C$ on $D_\rho$.  Since $\psi=v$ on $\partial D_\rho$ we have $C=0$ and $\psi=v$ on $D_\rho$. This is a contradiction to $\psi(x_0)>v(x_0)$.
If $v$ is $p$-harmonic we get that $v$ and $-v$ are viscosity supersolutions. Hence $v$ is a viscosity solution of $\Delta_p v=0$.
\end{proof}

Lemma \ref{lem:p-harm_viscos} ensures together with Theorem \ref{thm:p-harmex} the existence of a viscosity solution of $\Delta_p u=0$ for given boundary values $g\colon \partial \Omega \to \R$.

\begin{lemma}\label{lem:min_conv}
Let $f_i \to f$ uniformly in $\overline{\Omega}$ for $f_i, f \in C(\overline{\Omega})$ and let $\phi \colon \Omega \to \R$ with $\phi(x)<f(x)$ for $x\neq x_0$ and $\phi(x_0)=f(x_0)$ for some $x_0\in \Omega$. Then there exist $(x_j)_j\subset \Omega$ with $x_j\to x_0$ such that
\begin{equation}
f_j(x_j)-\phi(x_j)=\min_\Omega \{f_j-\phi\}
\end{equation}
for some subsequence.
\end{lemma}
\begin{proof}See \cite[Lemma 11]{lindqvist}.
Since $f(x)>\phi(x)$ for $x\in \Omega\backslash \{x_0\}$ and $f$ is continuous we have for $r>0$ that $\inf_{x\in \Omega\backslash\{B_r(x_0)}(f(x)-\phi(x)>0$. Because $f_j-\phi=(f-\phi)+(f_j-f)$ and $f_j$ converges uniformly to $f$, we get for sufficiently large $j$ that
\begin{equation}
\inf_{\Omega\backslash B_r(x_0)}\{ f_j-\phi\} \geq \inf_{\Omega\backslash B_r(x_0)} \{f-\phi\}+\inf_{\Omega\backslash B_r(x_0)}\{ f-f_j\} \geq \frac{1}{2} \inf_{\Omega\backslash B_r(x_0)} \{f-\phi\}>0.
\end{equation}
Since $f_j(x_0)\to f(x_0)$ and $f(x_0)-\phi(x_0)=0$ there exists a $j_r$ such that it holds for all $j>j_r$ that
\begin{equation}
\inf_{\Omega\backslash B_r(x_0)} \{f_j-\phi\}>f_j(x_0)-\phi(x_0).
\end{equation}
Thus there exists some $x_j\in \overline{B_r(x_0)}$ such that
\begin{equation}\label{eq:min_f_j-phi}
\min_{\Omega} \{f_j-\phi\}=f_j(x_j)-\phi(x_j)
\end{equation}
for $j>j_r$. Now for $j\in \Nat$ choose $x_j\in \overline{B_r(x_0)}$ with \eqref{eq:min_f_j-phi} and 
\begin{equation}
r=\min \left\{ \frac{1}{k}:k\in \Nat\text{ with }j>j_\frac{1}{k}\right\}.
\end{equation}
Then we have $x_j\to x_0$ as $j\to \infty$.
\end{proof}

\begin{theorem}\label{thm:variational_solution_is_viscosity_solution}
Let $u_\infty$ be a variational solution of $\Delta_\infty u=0$ for Lipschitz continuous boundary values $g\colon \partial \Omega\to \R$. Then it is a viscosity solution of $\Delta_\infty u=0$.
\end{theorem}
This theorem is a special case of \cite[Theorem 12]{lindqvist}. We adapted the proof to this special case.
\begin{proof}
From Theorem \ref{thm:varsolex} we get $p_j$-harmonic functions with $u_{p_j}\mid_{\partial \Omega}=g$, where $u_{p_j}$ converges uniformly to $u_\infty$ for $(p_j)_j$ a subsequence of $(1,2,...)$. Due to Lemma \ref{lem:p-harm_viscos} we have that $u_{p_j}$ is a viscosity solution of $\Delta_{p_j} u=0$.\\
Let $x_0\in \Omega$ and $\phi\in C^2(\Omega)$ with $\phi(x)<u_\infty(x)$ for $x\neq x_0$ and $\phi(x_0)=u_\infty (x_0)$. Due to Lemma \ref{lem:min_conv} we have a subsequence $u_{p_{j_k}}$ of $u_{p_j}$ points $x_k \to x_0$ with $u_{p_{j_k}}(x_k)-\phi(x_k)=\min_\Omega \{u_{p_{j_k}}-\phi\}$. Hence we get for $\phi'=\phi+u_{p_{j_k}}(x_k)-\phi(x_k)-(x-x_k)^4$ that $\phi'(x_k)=u_{p_{j_k}}(x_k)$ and $\phi'(x)<u_{p_{j_k}}(x)$ for $x\neq x_k$. Because $u_{p_{j_k}}$ is a viscosity solution of $\Delta_{p_{j_k}} u=0$ we get
\begin{equation}
\Delta_{p_{j_k}} \phi=\Delta_{p_{j_k}} \phi' \leq 0.
\end{equation}
By Lemma \ref{lem:p-lap_darstellung} we can rewrite this as
\begin{equation}\label{eq:infLapPhi}
\abs{\nabla \phi(x_k)}^{p_{j_k}-2} \Delta \phi(x_k) +(p_{j_k}-2) \abs{\nabla \phi(x_k)}^{p_{j_k}-4}\Delta_\infty \phi(x_k)\leq 0.
\end{equation}
If $\abs{\nabla \phi(x_0)}=0$ then we have $\Delta_\infty \phi = \frac{1}{2} \inner{\nabla \left( \abs{\nabla \phi}^2 \right)(x_0)}{\nabla \phi(x_0)}=0$.\\ Otherwise we have by the continuity of $\nabla \phi$ that $\abs{\nabla \phi(x_k)}\to c$ for some $c>0$ and $k \to \infty$. Hence we can divide out $\abs{\nabla \phi(x_k)}$ in \eqref{eq:infLapPhi} for large $k$. We get
\begin{equation}
\frac{\Delta \phi (x_k)}{p_{j_k}-2}+\frac{\Delta_\infty \phi (x_k)}{\abs{\nabla \phi(x_k)}^2}\leq 0.
\end{equation}
Since $\abs{\nabla \phi(x_k)}\to c$ we get for $k\to \infty$ that $\Delta_\infty \phi \leq 0$.\\
Thus $u_\infty$ is a viscosity supersolution of $\Delta_\infty u=0$. Analogously we get that $u_\infty$ is a viscosity subsolution. Hence it is a viscosity solution. 
\end{proof}

\begin{theorem}\label{thm:uniqueness_of_viscosity}
Let $g\colon \partial \Omega \to \R$ Lipschitz continuous. Then there exists a unique $u_\infty\in C(\overline{\Omega})\cap W^{1,\infty}(\Omega)$ such that $u_\infty$ is a viscosity solution of $\Delta_\infty u=0$ on $\Omega$ and $u=g$ on $\partial \Omega$.
\end{theorem}
\begin{proof}
Theorem \ref{thm:varsolex} and Theorem \ref{thm:variational_solution_is_viscosity_solution} yield the existence of such an $u_\infty$. For the uniqueness see \cite[Theorem 27]{lindqvist}.
\end{proof}

\begin{remark}\label{rem:uniqueness}
In particular, it follows from the Theorems \ref{thm:variational_solution_is_viscosity_solution} and \ref{thm:uniqueness_of_viscosity} that variational solutions of $\Delta_p u=0$ are unique for given Lipschitz boundary values $g\colon \partial \Omega \to \R$. This proves Theorem \ref{thm:classical_solution_is_variational_solution}. Further this proves that solutions of $\Delta_\infty u=0$ for fixed boundary values are unique.
\end{remark}


\section{Vector-valued Lipschitz extensions}\label{sec:vec-valLip}
We denote by $\Norm{A}{2}=\sup_{x\in \R^d} \abs{Ax}$ for $A\in \R^{m\times d}$ the operator norm of $A$.\\

We define for $1\leq p \leq \infty$ the Sobolev space
\begin{equation}
W^{1,p}(\Omega,\R^m)\coloneqq \{ f=(f^{(1)},...,f^{(m)})\colon \Omega \to \R^m: f^{(i)}\in W^{1,p}(\Omega)\}
\end{equation}
and equip this linear space with the norm
\begin{equation}
\Norm{f}{W^{1,p}(\Omega,\R^m)}\coloneqq \Norm{f}{L^p(\Omega, \R^m)}+ \Norm{D f}{L^p_{\Norm{\cdot}{2}}(\Omega, \R^{m\times d})},
\end{equation}
where $D f = \left(\begin{matrix}\nabla {f^{(1)}}^T\\...\\\nabla {f^{(m)}}^T\end{matrix}\right)$ is the weak derivative of $f$. It can be shown that for $1<p<\infty$ the vector space $W^{1,p}(\Omega,\R^m)$ is a reflexive Banach space.


\subsection{Vector-valued $p$-harmonic functions}

Let $1<p<\infty$ and let $\Norm{\cdot}{}$ be an arbitrary norm on $\R^{m\times d}$. In the following we minimize the energy functional
\begin{equation}
E_p(u)=\frac{1}{p}\int_\Omega \Norm{D u(x)}{}^p dx
\end{equation}
for $u\in U_p$, where $Du(x)$ denotes the Jacobian of $u$ in $x$.\\

We can show uniqueness and existence of minimizers of $E_p$ for $p>d$, but since arbitrary norms on $\R^{m\times d}$ are not differentiable we cannot derive Euler-Lagrange equations in general. We adapt some statements from the scalar case.

\begin{lemma}
The functional $E_p$ is strictly convex on $U_p$.
In particular, minimizers of $E_p$ are unique in $U_p$
\end{lemma}
\begin{proof}
Analogously to Lemma \ref{lem:p-harmun}.
\end{proof}

\begin{theorem}
Let $p>d$ and $g\in W^{1,p}(\Omega,\R^m)\cap C(\overline{\Omega})$. Then there exists a minimizer of $E_p$ in $U_p$.
\end{theorem}
\begin{proof}The proof is similar to the proof of Theorem \ref{thm:p-harmex}:\\
Define $I_0=\inf_{v\in U_p} E_p v$. Let $(v_j)_{j\in\Nat}\subset U_p$ with $E_p v_j \leq I_0 + \frac{1}{j}$. Since all norms on $\R^{m\times d}$ are equivalent, there exists some $c>0$ such that we have 
\begin{equation}
\frac{1}{p}\Norm{\abs{\partial_i v_j}}{L^p(\Omega)}\leq \frac{1}{p}\Norm{\Norm{D v_j}{F}}{L^p(\Omega)}\leq\frac{c}{p}\Norm{\Norm{D v_j}{}}{L^p(\Omega)}=c E_p v_j\leq c (I_0 +1)
\end{equation}
for all $j\in \Nat$.\\
Further we have by the Poincaré inequality (Lemma \ref{lem:poincare}) that there exists some $C>0$ such that
\begin{equation}
\begin{aligned}
\Norm{\abs{v_j}}{L^p(\Omega)}&\leq \sum_{i=1}^d \Norm{v_j^{(i)}}{L^p(\Omega)}\\
&\leq \sum_{i=1}^d \Norm{v_j^{(i)}-g^{(i)}}{L^p(\Omega)}+\Norm{g^{(i)}}{L^p(\Omega)}\\
&\leq \sum_{i=1}^d C\Norm{\nabla(v_j^{(i)}-g^{(i)})}{L^p(\Omega)}+\Norm{g^{(i)}}{L^p(\Omega)}\\
&\leq \sum_{i=1}^d C\Norm{\nabla{v_j^{(i)}}}{L^p(\Omega)}+\Norm{\nabla g^{(i)}}{L^p(\Omega)}+\Norm{g^{(i)}}{L^p(\Omega)}\\
&\leq	 \sum_{i=1}^d C\Norm{\Norm{D v_j}{F}}{L^p(\Omega)}+\Norm{\nabla g^{(i)}}{L^p(\Omega)}+\Norm{g^{(i)}}{L^p(\Omega)}\\
&\leq d C c \Norm{\Norm{D v_j}{}}{L^p(\Omega)} + \sum_{i=1}^d \Norm{\nabla g^{(i)}}{L^p(\Omega)}+\Norm{g^{(i)}}{L^p(\Omega)}\\
&=d C c p E_p v_j + \sum_{i=1}^d \Norm{\nabla g^{(i)}}{L^p(\Omega)}+\Norm{g^{(i)}}{L^p(\Omega)}\\
&\leq d C c p (I_0+1) + \sum_{i=1}^d \Norm{\nabla g^{(i)}}{L^p(\Omega)}+\Norm{g^{(i)}}{L^p(\Omega)}.
\end{aligned}
\end{equation}
Hence $\left(\Norm{\partial_i v_j}{L^p(\Omega)}\right)_j$ and $\left(\Norm{\abs{v_j}}{L^p(\Omega)}\right)_j$ are bounded and therefore $\left(\Norm{v_j}{W^{1,p}(\Omega, \R^m)}\right)_j$ is bounded. Since $W^{1,p}(\Omega, \R^m)$ is reflexive for $1<p<\infty$, we get that there is some $v\in W^{1,p}(\Omega, \R^m)$ and a subsequence $(v_{j_k})_k$ with $v_{j_k} \rightharpoonup v$ in $W^{1,p}(\Omega, \R^m)$. Since $W_0^{1,p}(\Omega,\R^m)$ is closed under weak convergence and $v_{j_k}-g \in W_0^{1,p}(\Omega,\R^m)$ we have that $v\mid_{\partial \Omega}=g\mid_{\partial \Omega}$.\\
Since $p>d$ and $v^{(i)}-g^{(i)}\in W_0^{1,p}(\Omega,\R^m)$ for $i=1,...,m$ we get by Lemma \ref{lem:morrey}, that we can redefine $v^{(i)}$ in a set of measure zero such that $v^{(i)}-g^{(i)}\in C(\overline{\Omega})$. After doing this for each component we have $v-g\in C(\overline{\Omega})$. Since $g$ is continuous by assumption we have $v\in  W^{1,p}(\Omega, \R^m)\cap C(\overline{\Omega})$.\\
Since $D\colon W^{1,p}(\Omega,\R^m) \to L^p_{\Norm{\cdot}{}}(\Omega,\R^{m\times d})$ is linear and continuous we have that $D v_{j_k} \rightharpoonup D v$ in $L^p_{\Norm{\cdot}{}}(\Omega,\R^{m\times d})$ and because of the lower semicontinuity of the norm we get 
\begin{equation}
E_p v=\frac{1}{p} \Norm{\Norm{D v}{}}{L^p(\Omega)}^p \leq \liminf_{k\to \infty} \frac{1}{p} \Norm{\Norm{D v_{j_k}}{}}{L^p(\Omega)}^p = \liminf_{k\to \infty} E_p v_{j_k}=I_0.
\end{equation}
Therefore, $v$ is a minimizer of $E_p$.
\end{proof}

\subsection{Vector-valued limiting process $p\to \infty$}

\begin{theorem}[Kirszbraun]\label{thm:kirszbraun}
Let $A\subset \R^d$ and $f\colon A\to \R^m \in C^{0,1}_L(A)$. Then there exists $F\colon \R^d\to \R^m \in C^{0,1}_L(\R^d)$ with $F\mid_A=f$.
\end{theorem}
\begin{proof}
See \cite[Theorem 3.5]{lipLectures}.
\end{proof}

As in the one-dimensional case we use the notation
\begin{equation}
\mu(f,D)=\sup_{x\neq y\in D} \frac{\abs{f(x)-f(y)}}{\abs{x-y}}
\end{equation}
for $f\colon D\to \R^m \in C^{0,1}(D)$.

We call an Lipschitz continuous extension $f\colon \overline{\Omega} \to \R^m$ of a Lipschitz continuous function $g\colon \partial \Omega \to \R^m$ \emph{absolute minimal} if we have $\mu(f,D)=\mu(f,\partial D)$ for all open $D\subset \Omega$.\\

The following lemma shows a connection between Lipschitz continuity and the matrix norm $\Norm{A}{2}=\sup_{x\in \R^d} \abs{Ax}$ for $A\in \R^{m\times d}$.

\begin{lemma}\label{lem:lip_opnorm}
Let $f\colon \overline{\Omega}\to \R^m$ be a Lipschitz continuous extension of $g\colon \partial \Omega \to \R^m$ with $\mu(g,\partial \Omega)=L$ and $\Norm{\Norm{D f}{2}}{L^\infty(\Omega)}\leq L$. Then we have $\mu(f,\overline{\Omega})=L$. 
\end{lemma} 
\begin{proof}
Assume that $\mu(f,\overline{\Omega})>L$. Then there exists $x,y\in \Omega$ such that $\frac{\abs{f(x)-f(y)}}{\abs{x-y}}>L\geq\Norm{\Norm{Df}{2}}{L^\infty(\Omega)}$. With out loss of generality we can assume that $f(y)=0$. Consider the function $\abs{f}$. Since $\abs{\cdot}$ is Lipschitz continuous with constant $1$, we have that $\abs{f}$ is Lipschitz continuous and by Theorem \ref{thm:rademacher} that $\nabla \abs{\cdot}\leq 1$. Further Theorem \ref{thm:rademacher} implies that $\abs{f}\in W^{1,\infty}(\Omega)$. This yields that $\nabla \abs{f} = D f \nabla (\abs{\cdot})$. This yields by the definition of $\Norm{\cdot}{2}$ that $\abs{\nabla \abs{f}}=\Norm{D f}{2}\abs{\nabla (\abs{\cdot})}\leq\Norm{D f}{2}$. Hence we have 
\begin{equation}\label{eq:lip_opnorm_rechnerei}
\Norm{\nabla \abs{f}}{L^\infty(\Omega)}\leq \Norm{\Norm{Df}{2}}{L^\infty(\Omega)}\leq L<\frac{\abs{f(x)-f(y)}}{\abs{x-y}}=\frac{\abs{\abs{f(x)}-\abs{f(y)}}}{\abs{x-y}}.
\end{equation}
Now $\abs{f}$ is a Lipschitz continuous extension of $\abs{g}$ and since $\abs{\cdot}\in C^{0,1}_1(\Omega)$ we have that $\abs{g}\in C^{0,1}_L(\Omega)$. Hence Lemma \ref{lem:infNorm-lipschitz} and \eqref{eq:lip_opnorm_rechnerei} yield that $\mu(\abs{f},\Omega)=L$. This contradicts that $\frac{\abs{\abs{f(x)}-\abs{f(y)}}}{\abs{x-y}}>L$.
\end{proof}

We again adapt a theorem from the scalar case.

\begin{theorem}\label{thm:mehrdim_varsolex}
Let $g \colon \partial \Omega \to \R^d \in C^{0,1}(\partial \Omega)$ with $\mu(g,\partial \Omega)=L$ and let $(u_p)_{p\in \Nat}$ a sequence with $p$-harmonic $u_p\in U_p$. Then there exists a subsequence $(u_{p_j})_j$ converging uniformly to some $u_\infty \in U_\infty$ such that  $\partial_i u_{p_j}\rightharpoonup \partial_i u_\infty$ weakly as $j \to \infty$ in $L^s(\Omega)$ for all $s>1$ and $i=1,...,m$.
Further, we have that $\mu(u_\infty,\overline{\Omega})=L$, if we use the norm $\Norm{\cdot}{2}$ on $\R^{m\times d}$ as defined above.
\end{theorem}
\begin{definition}
We call the $u_\infty$ constructed in Theorem \ref{thm:mehrdim_varsolex} a \emph{variational $\infty$-harmonic} function. 
\end{definition}
\begin{proof}[Proof of Theorem \ref{thm:mehrdim_varsolex}]
The proof is similar to the proof of Theorem \ref{thm:varsolex}.\\
Due to Kirszbrauns theorem (Theorem \ref{thm:kirszbraun}) we can extend $g$ to some Lipschitz continuous function on $\overline{\Omega}$ with Lipschitz constant $L$. We denote this extension again by $g$. Rademachers theorem (Theorem \ref{thm:rademacher}) ensures that $g\in C^{0,1}(\overline{\Omega})\cap W^{1,p}(\Omega,\R^m)$. We show equicontinuity and equiboundedness of $\{u_p\}$ for $p>d$. Since $g$ is Lipschitz continuous with constant $L$,  we have that also the component functions $g^{(i)}$ are in $C^{0,1}_L(\overline{\Omega})$. Hence we get analogously to \eqref{eq:absch_varsolex} for all $p>d$ that
\begin{equation}
\abs{u_p^{(i)}(x)-u_p^{(i)}(y)}\leq L\abs{x-y}+C\abs{x-y}^{1-\frac{d}{p}}
\end{equation}
for some $C>0$ depending on $\Omega$. Therefore, we have
\begin{equation}
\abs{u_p(x)-u_p(y)}\leq\sum_{i=1}^m\abs{u_p^{(i)}(x)-u_p^{(i)}(y)}\leq m L\abs{x-y}+m C\abs{x-y}^{1-\frac{d}{p}}.
\end{equation}
Thus $\{u_p\}_p$ is equicontinuous. Since $u_p\mid_{\partial\Omega}=g\mid_{\partial\Omega}$ we get for some fixed $x_0\in \partial\Omega$ that for all $x\in \overline{\Omega}$ it holds
\begin{equation}
\begin{aligned}
\abs{u_p(x)}&\leq\abs{u_p(x)-u_p(x_0)}+\abs{u_p(x_0)}\\
&\leq m L\abs{x-x_0}+ m C \abs{x-x_0}^{1-\frac{d}{p}}+\abs{g(x_0)}\\
&\leq m L \diam{\Omega}+ m C \diam{\Omega}^{1-\frac{d}{p}}+\abs{g(x_0)}
\end{aligned}
\end{equation}
Hence we have $\abs{u_p(x)}\leq m L + m C +\abs{g(x_0)}$ for $\diam{\Omega}\leq 1$ . For $\diam{\Omega}>1$ we get thtat
\begin{equation}
\abs{u_p(x)}\leq m L \diam{\Omega}+ m C \diam{\Omega}^{1-\frac{d}{p}}+\abs{g(x_0)} \leq m L \diam{\Omega}+ m C \diam{\Omega}+\abs{g(x_0)}.
\end{equation}
Therefore, $\{u_p\}$ is equibounded and we can use the theorem of Arzela-Ascoli: There exists a subsequence $(u_{p_j})_j$ converging uniformly to some $u_\infty\in C(\overline{\Omega})$ with $u_\infty\mid_{\partial\Omega}=g\mid_{\partial\Omega}$.\\
Furthermore, we have
\begin{equation}
\begin{aligned}
\abs{u_\infty(x)-u_\infty(y)}&=\lim\limits_{j\to \infty}\abs{u_{p_j}(x)-u_{p_j}(y)}\\
&\leq \lim\limits_{j\to \infty} m L \abs{x-y}+ m C \abs{x-y}^{1-\frac{d}{p_j}}\\
&= m L \abs{x-y} + m C \abs{x-y}
\end{aligned}
\end{equation}
Thus $u_\infty$ is Lipschitz continuous and therefore almost everywhere differentiable by Theorem \ref{thm:rademacher}. Since $\frac{1}{\lambda (\Omega)}dx$ is a probability measure we get for $p_j>s$ using Lemma \ref{lem:Lpschachtel} and the equivalence of the norms on $\R^{m\times d}$
\begin{equation}\label{eq:mehrdim_varsolex_absch}
\begin{aligned}
\left( \int_\Omega \Norm{\partial_i u_{p_j}}{2}^s \frac{1}{\lambda(\Omega)}dx\right)^\frac{1}{s}&\leq\left( \int_\Omega \Norm{D u_{p_j}}{F}^s \frac{1}{\lambda(\Omega)}dx\right)^\frac{1}{s}\\
&\leq c \left( \int_\Omega \Norm{D u_{p_j}}{}^s \frac{1}{\lambda(\Omega)}dx\right)^\frac{1}{s}\\
&\leq c \left( \int_\Omega \Norm{D u_{p_j}}{}^{p_j} \frac{1}{\lambda(\Omega)}dx\right)^\frac{1}{p_j}\\
&\leq c \left( \int_\Omega \Norm{D g}{}^{p_j} \frac{1}{\lambda(\Omega)}dx\right)^\frac{1}{p_j}\\
&\leq c \Norm{\Norm{D g}{}}{L^\infty(\Omega)}\leq c C \Norm{\Norm{D g}{2}}{L^\infty(\Omega)}\leq c C L,
\end{aligned}
\end{equation}
for some $C, c>0$ and $\lambda$ the Lebesgue measure. If $\Norm{\cdot}{}=\Norm{\cdot}{2}$ we have $C=1$ in \eqref{eq:mehrdim_varsolex_absch}. Hence $\Norm{\partial_i u_{p_j}}{L^s(\Omega,\R^m)}$ is bounded. Thus also the Sobolev norm $\Norm{u_{p_j}}{W^{1,s}(\Omega,\R^m)}$ is bounded. Since the Sobolev spaces $W^{1,p}(\Omega,\R^m)$ for $1<p<\infty$ are reflexive we know that $\overline{B_r(0)}$ is weak sequentially compact for all $r>0$. Hence each subsequence of $(u_{p_j})_j$ has a subsequence $(u_{p_{j_k}})_k$, which converges weakly to some $v\in W^{1,s}(\Omega,\R^m)$. The weak convergence in $W^{1,s}(\Omega,\R^m)$ yields weak convergence of $(u_{p_{j_k}})_k$ to $v$ and of $(\partial_i u_{p_{j_k}})_k$ to $\partial_i v$ in $L^s(\Omega)$. Since weak limits are unique in $W^{1,s}(\Omega,\R^m)$ we have that $v=u_\infty$ and $\partial_i u_{p_{j_k}}\rightharpoonup \partial_i u_\infty$ weakly in $L^s(\Omega)$.
Similar to the one-dimensional case we find by diagonalization a subsequence $(u_{p_{j_k}})_k$ of $(u_{p_j})_j$ such that $\partial_i u_{p_{j_k}} \rightharpoonup \partial_i u_\infty$ for all $s>0$.\\
Now let $\Norm{\cdot}{}=\Norm{\cdot}{2}$. Since $D\colon W^{1,p}(\Omega,\R^m) \to L^p(\Omega,\R^{m\times d})$ is linear and continuous we have that $D u_{p_{j_k}} \rightharpoonup D u_\infty$ in $L^p(\Omega,\R^{m\times d})$ and by the weak semi-continuity of the norm we have using \eqref{eq:mehrdim_varsolex_absch} with $C=1$
\begin{equation}
\int_\Omega \Norm{D u_\infty}{2}^s dx \leq \limsup_{k\to \infty} \int_\Omega \Norm{D u_{p_{j_k}}}{2}^s dx \leq L^s.
\end{equation}
i.e. $\Norm{\Norm{D u_\infty}{2}}{L^s(\Omega)}\leq L$. Hence for $s\to\infty$ we get
\begin{equation}
\Norm{\Norm{D u_\infty}{2}}{L^\infty(\Omega)}\leq L.
\end{equation}
Thus we get by Lemma \ref{lem:lip_opnorm}
\begin{equation}
L=\mu(g,\partial \Omega)=\mu(u_\infty,\partial \Omega)=\mu(u_\infty,\overline{\Omega})\leq L.
\end{equation}
\end{proof}


\begin{definition}
For $x\in \Omega$ and Lipschitz continuous $u\colon \Omega\to \R^m$ we denote the \emph{local Lipschitz constant} of $u$ in $x$ by 
\begin{equation}
L u(x) = \inf_{r>0} \mu(u, \Omega \cap B_r(x)).
\end{equation}
\end{definition}

\begin{example}
This example from \cite{sheffield2012vector} shows, that absolute minimal Lipschitz extensions are in general not unique.
We identify $\R^2$ with $\C$ using the canonical isometric isomorphism $(x,y)\mapsto x+i y$.\\

For $t\in [0,1]$ define $u_t\colon B_1(0) \subset \C \to \C$ by
\begin{equation}
u_t(z)=
\begin{cases} 
t z^2 + (1-t) \frac{z^2}{\abs{z}}, &$if $z\neq 0,\\
0, &$if $z=0.
\end{cases}
\end{equation}
Since the function $z\mapsto z^2$ is differentiable with derivative $z\mapsto 2 z$ and the function $z\mapsto \frac{z^2}{\abs{z}}$ has Lipschitz constant $2$, we get that the local Lipschitz constant of $u_t$ is bounded by 
\begin{equation}\label{eq:local_Lipschitz_ex_abs_min_not_unique}
Lu_t(x)\leq2+2t(\abs{x}-1).
\end{equation}
Now let $D\subset B_1(0)$ be open and $m=\max_{z\in \overline{D}} \abs{z}>0$. Then the we get because of \eqref{eq:local_Lipschitz_ex_abs_min_not_unique} that 
\begin{equation}
\mu (u_t,\overline{D})\leq 2+2t(m-1).
\end{equation}
Further the function $\overline{D}\to \R$ defined by $z\mapsto \abs{z}$ reaches its maximum $m$ at some $z_0\in \partial D$. For all $r>0$ there exist some $x_r\neq y_r\in \partial D \cap B_r(z_0)$ with $\abs{x_r}=\abs{y_r}$. Hence $y_r=e^{i \varphi}x_r$ for some $\varphi\in(-\pi,\pi)$. Now we have
\begin{equation}
\begin{aligned}
\mu(u_t,\partial D)&\geq\frac{\abs{u_t(x_r)-u_t(y_r)}}{\abs{x_r-y_r}}\\
&=\frac{\abs{u_t(x_r)-u_t(e^{i \varphi}x_r)}}{\abs{x_r(1-e^{i \varphi})}}\\
&=\frac{\abs{t x_r^2 (1-e^{2 i \varphi}) +\frac{1-t}{\abs{x_r}} x_r^2 (1-e^{2 i \varphi})}}{\abs{x_r(1-e^{i \varphi})}}\\
&=\frac{\abs{(1-e^{2 i \varphi})} (1 + t (\abs{x_r}-1)) \abs{x_r}}{\abs{(1-e^{i \varphi})} \abs{x_r}}\\
&=\frac{(1-e^{2 i \varphi})}{(1-e^{i \varphi})} (1 + t (\abs{x_r}-1)).
\end{aligned}
\end{equation}
Since $\phi \to 0$ and $\abs{x_r}\to \abs{z_0}=m$ as $r\to 0$ and $\frac{(1-e^{2 i \varphi})}{(1-e^{i \varphi})}\to 2$ as $\phi \to 0$ we get for $r\to 0$ that
\begin{equation}
\mu(u_t,\partial D)\geq 2+2t(m-1) \geq \mu (u_t,\overline{D}).
\end{equation}
Thus $u_t$ is absolute minimal for all $t\in [0,1]$. Since we have for $t_1,t_2\in [0,1]$ $u_{t_1}(z)=z^2=u_{t_2}(z)$ on $\partial B_1(0)$, the absolute minimal extension of the boundary values $g\colon \partial B_1(0) \to \C$ defined by $z\mapsto z^2$ is not unique.
\end{example}

\subsection{Tightness}

Since absolute minimal Lipschitz extensions of $g\colon \partial \Omega \to \R^m \in C^{0,1}_L$ to $f\colon \overline{\Omega} \to \R^m$ are in general not unique, \cite{sheffield2012vector} introduces another definition to characterize good Lipschitz extensions.

\begin{definition}
Let $u,v\in C^{0,1}(\overline{\Omega})$ with $u\mid_{\partial \Omega}=v\mid_{\partial \Omega}$. We say $v$ is \emph{tighter} than $u$ if $u$ and $v$ satisfy
\begin{equation}
\sup \{L u(x):x\in\Omega\text{ and }L v(x)<L u(x) \}> \sup \{L v(x): x\in\Omega \text{ and }L u(x)<L v(x)\}.
\end{equation}
We say $u$ is \emph{tight} if there is no tighter $v\in C^{0,1}(\overline{\Omega})$ with $u\mid_{\partial \Omega}=v\mid_{\partial \Omega}$.
\end{definition}

We observe that tightness is stronger than absolute minimality. This statement is mentioned several times in \cite{sheffield2012vector}. Since we did not found a proof, we did it ourself.

\begin{theorem}
Let $u\colon \overline{\Omega} \to \R^m$ be tight. Then $u$ is absolute minimal on $\Omega$.
\end{theorem}
\begin{proof}
Let $u$ be not absolute minimal. We show that this implies that $u$ is not tight.\\
Since $u$ is not absolute minimal, there exists some open $D\subset \Omega$ such that $\mu(u,\partial D)<\mu(u,\overline{D})$. Due to Theorem \ref{thm:kirszbraun} there exists some $f\colon \overline{D}\to \R^m$ with $f\mid_{\partial D}=u\mid_{\partial D}$ and $\mu(f,\overline{D})=\mu(u,\partial D)$. Define $v\colon \overline{\Omega}\to \R^m$ by
\begin{equation}
v(x)=
\begin{cases}
f(x),&$if $ x\in D,\\
u(x),&$if $ x\not\in D.
\end{cases}
\end{equation}
We show that $v$ is tighter than $u$.\\
For $x \not\in \overline{D}$ there exists some $r>0$ such that $u=v$ on $B_r(x)$. Therefore, we have $Lu(x)=Lv(x)$.\\
Since $u=v$ on $\Omega\backslash D$ we have for $x \in \partial D$ that 
\begin{equation}
L v(x)\leq \max \{Lu(x), \mu(v,\overline{D})\}=\max \{Lu(x), \mu(u,\partial D)\}.
\end{equation}
Because $\mu(v,\overline{D})=\mu(u,\partial D)$ we have for $x\in D$ that $L v(x)\leq \mu(u,\partial D)$.
We can conclude that $\sup \{L v(x): x\in\Omega \text{ and }L u(x)<L v(x)\}\leq \mu(u,\partial D)$.\\
Since $\mu(u,\partial D)<\mu(u,\overline{D})$ there exists $x_0, y_0\in \overline{D}$ such that $\frac{\abs{u(x_0)-u(y_0)}}{\abs{x_0-y_0}}>\mu(u,\partial D)$. \\
We show that we can assume $\gamma((0,1))\subset D$. Let $\gamma\colon [0,1]\to\R^d$ defined by $\gamma (t)=t y_0 + (1-t) x_0$ and $a=\min\{ t \in [0,1]: \gamma(t)\in \partial D\}$ and $b=\max\{ t \in [0,1]: \gamma(t)\in \partial D\}$. If $\gamma((0,1))\not\subset D$ the minimum and maximum exist. Since $x_0, a, b, y_0$ are on a line we have
\begin{equation}
\begin{aligned}
\mu(u,\partial D)&<\frac{\abs{u(x_0) - u(y_0)}}{\abs{x_0-y_0}} \\
&\leq \frac{\abs{u(x_0) - u(\gamma(a))}+\abs{u(\gamma(a))-u(\gamma(b))}+\abs{u(\gamma(b)) - u(y_0)}}{\abs{x_0-y_0}}\\
&=\frac{\abs{u(x_0) - u(\gamma(a))}+\abs{u(\gamma(a))-u(\gamma(b))}+\abs{u(\gamma(b)) - u(y_0)}}{\abs{x_0-\gamma(a)}+\abs{\gamma(a)-\gamma(b)}+\abs{\gamma(b)-y_0}}.
\end{aligned}
\end{equation} 
Since $\gamma(a), \gamma(b)\in \partial D$ we have that $\frac{\abs{u(\gamma(a))-u(\gamma(b))}}{\abs{a-b}}\leq \mu(u,\partial D)$. This yields that 
\begin{equation}
\frac{f(x_0)-f(\gamma(a))}{\abs{x_0-\gamma(a)}}>\mu(u,\partial D) \text{ or } \frac{f(\gamma(b))-f(y_0)}{\abs{\gamma(b)-y_0}}>\mu(u,\partial D).
\end{equation}
Since $(x_0,\gamma(a))\subset D$ and $(\gamma(b),y_0)\subset D$ we can rename the points such that we have $x_0,y_0\in \overline{D}$ with $\gamma(0,1)\subset D$ and $\frac{\abs{u(x_0)-u(y_0)}}{\abs{x_0-y_0}}>\mu(u,\partial D)$.\\

Because $u$ and $\gamma$ are continuous we have that $\lim\limits_{\epsilon \to 0} \frac{\abs{u(\gamma(\epsilon)) - u(\gamma(1-\epsilon))}}{\gamma(\epsilon)-\gamma(1-\epsilon)}=\frac{\abs{u(x_0)-u(y_0)}}{\abs{x_0-y_0}}>\mu(u,\partial D)$. Hence there exists some $\frac{1}{2}>\epsilon>0$ such that $\frac{\abs{u(\gamma(\epsilon)) - u(\gamma(1-\epsilon))}}{\gamma(\epsilon)-\gamma(1-\epsilon)}>\mu(u,\partial D)$. Thus we can assume by renaming the points that $x_0,y_0\in D$.\\

Now we construct recursively $x_n,y_n\in D$ with $\frac{\abs{u(x_n)-u(y_n)}}{\abs{x_n-y_n}}\geq\frac{\abs{u(x_0)-u(y_0)}}{\abs{x_0-y_0}}$. Let $n\in \Nat$ and $x_n, y_n$ be already constructed. Then consider $z=\frac{x_n+y_n}{2}$. Since 
\begin{equation}
\frac{\abs{u(x_0)-u(y_0)}}{\abs{x_0-y_0}}\leq\frac{\abs{u(x_n)-u(y_n)}}{\abs{x_n-y_n}}\leq \frac{\abs{u(x_n)-u(z)}+\abs{u(z)-u(y_n)}}{\abs{x_n-z}+\abs{z-y_n}}
\end{equation}
we have that $\frac{\abs{u(x_n)-u(z)}}{\abs{x_n-z}}\geq \frac{\abs{u(x_0)-u(y_0)}}{\abs{x_0-y_0}}$ or $\frac{\abs{u(z)-u(y_n)}}{\abs{z-y_n}}\geq \frac{\abs{u(x_0)-u(y_0)}}{\abs{x_0-y_0}}$. Therefore, either $x_{n+1}\coloneqq x_n$ and $y_{n+1}\coloneqq z$ or $x_{n+1}\coloneqq z$ and $y_{n+1}\coloneqq y_n$ fulfill the conditions from above.\\

By construction we have that $\abs{x_n-y_n}=\frac{1}{2^n} \abs{x_0-y_0}$ and $\abs{x_{n+1}-x_n}\leq \frac{\abs{x_n-y_n}}{2}$. Hence $(x_n)_n$ is a Cauchy sequence and converges to some $x\in \R^d$. It holds $x\in \gamma([0,1])\subset D$ since $\gamma([0,1])$ is closed. Now for $r>0$ there exists some $N\in \Nat$ such that for all $n>N$ we have that $x_n,y_n\in B_r(x)$. Thus $\mu(u,\Omega\cap B_r(x))\geq \frac{\abs{u(x_0)-u(y_0)}}{\abs{x_0-y_0}}$ for all $r>0$. This yields $L u(x)\geq\frac{\abs{u(x_0)-u(y_0)}}{\abs{x_0-y_0}}>\mu(u,\partial D)$. Since $x\in D$ we know that $L v(x)\leq \mu(u,\partial D)$. This yields that
\begin{equation}
\begin{aligned}
&\sup \{L u(x): x\in\Omega \text{ and }L v(x)<L u(x)\}>\mu(u,\partial D)\\
\geq&\sup \{L v(x): x\in\Omega \text{ and }L u(x)<L v(x)\}
\end{aligned}
\end{equation}
Hence $v$ is tighter than $u$. This contradicts that $u$ is tight.
\end{proof}

We are not able to show existence or uniqueness of tight extensions of Lipschitz boundary values. But we get a similar characterization of tightness as in the real-valued case for absolute minimal extensions stated in \cite{sheffield2012vector}. To formulate this characterization we first one more definition from \cite{sheffield2012vector}.

\begin{definition}
Let $U\subset \R^d$ be open. We define a \emph{principal direction field} for a function $u\in C^1(U,\R^m)$ as a unit vector field $a\in C(U,\R^d)$ such that $a(x)$ spans the principal eigenspace of $Du(x)^T Du(x)$ (i.e. the eigenspace to the biggest eigenvalue of $Du(x)^T Du(x)$). 
\end{definition}

\begin{theorem}\label{thm:tightpde}
Let $U\subset \R^d$ be open and bounded. Suppose that $u\in C^3(U,\R^m)$ has a principal direction field $a\in C^2(U,\R^d)$. Then $u$ is tight if and only if
\begin{equation}
-\sum_{j=1}^d a^j\frac{\partial}{\partial x_j}\left(\sum_{i=1}^d a^i \frac{\partial u}{\partial x_i}\right)=0
\end{equation}
\end{theorem} 
\begin{proof}
See \cite[Theorem 1.5]{sheffield2012vector}.
\end{proof}

\begin{remark}
We can derive the Euler-Lagrange equations for the more general energy functional
\begin{equation}
v\mapsto \int_\Omega H(D v (x)) dx,
\end{equation}
where $H\in C^2(\R^{d\times m})$. The derived system of partial differential equations and its connection to tight functions is considered in \cite{katzourakis, katzourakis2}. If we choose $H=\Norm{\cdot}{F}^p$ a formulation of a vector-valued $\infty$-Laplace operator is derived. For more details see \cite{sheffield2012vector, katzourakis, katzourakis2}.
\end{remark}

In Section \ref{sec:discrete_tight} we consider a discrete formulation of tightness. In this case we can show existence and uniqueness of some boundary values.


\section{Optimal Lipschitz extensions on graphs}\label{sec:opt_Lip_ext_graphs}

In the following let $G=(V,E,\omega)$ be an undirected connected weighted graph with weighting function $\omega\colon E\to [0,1]$, where $\omega(x,y)=0$ if and only if $(x,y)\not\in E$. Let $\emptyset \neq U \subseteq V$. We denote the set of functions $u\colon V\to \R^m$ by $\Knoten(V)$. For $x,y\in V$ we write $x\sim y$ if and only if $(x,y)\in E$.\\

The following definitions without weighting functions are stated in \cite{sheffield2012vector}.

\begin{definition}\label{def:local_Lipschitz}
The \emph{local Lipschitz constant} of a function $u\colon V\to \R^m$ at $x\in V$ is defined by
\begin{equation}
S u(x)=\max_{y\sim x} \sqrt{\omega(x,y)}\abs{u(y)-u(x)}
\end{equation}
\end{definition}

\begin{definition}
We call a function $u\colon V \to \R^m$ \emph{discrete $\infty$-harmonic} if it holds
\begin{equation}
u(x)\in \argmin_{a\in \R^m} \left\{ \max_{y\sim x} \sqrt{\omega(x,y)}\abs{u(y)-a} \right\}
\end{equation}
for all $x\in V\backslash U$.
\end{definition}

\subsection{Graph-$\infty$-Laplace operator}

In this subsection we deal with the case $m=1$, i.e. with real-valued functions. The graph-$\infty$-Laplace operator is considered in connection with image processing and machine learning tasks in \cite{ETT, ABDERRAHIM2014153}.

As in \cite{ETT} we define a discrete formulation of the $\infty$-Laplacian to obtain these discrete $\infty$-harmonic functions.
Let $a^+ := \max(a,0)$ and $a^- := -\min(a,0)$.

\begin{definition}
The \emph{graph $\infty$-Laplace operator} for for $u\colon V\to \Knoten(V)$ is defined by
\begin{equation}
\Delta_{\omega,\infty}u (x)=\frac{1}{2} \left(\max_{y \sim x} \left(\sqrt{\omega(x,y)}(u(y)-u(x))^+\right) - \left(\max_{y \sim x} \sqrt{\omega(x,y)}(u(y)-u(x))^-\right)\right). 
\end{equation}
\end{definition}

The following theorem connects the two definitions.

\begin{theorem}\label{thm:disc_infhar_gleich_dirichlet}
Let $u\colon V\to \R$. Then $u$ is discrete $\infty$-harmonic if and only if ${\Delta_{\omega, \infty}u(x)=0}$ for all $x\in V\backslash U$.
\end{theorem}
Since we are not aware of a proof in literature, we provide it in the following.
\begin{proof}
Let $x\in V\backslash U$ be arbitrary fixed. We show that 
\begin{equation}\label{eq:discinfharminx}
u(x)\in \argmin_{a\in \R} \left\{ \max_{y\sim x} \sqrt{\omega(x,y)}\abs{u(y)-a} \right\}
\end{equation}
if and only if $\Delta_{\omega,\infty} u(x)=0$. Note that for $a\in \R$ it holds
\begin{equation}
\max_{y\sim x} \sqrt{\omega(x,y)}\abs{u(y)-a}=\max \left\{ \max_{y\sim x} \sqrt{\omega(x,y)}\left(u(y)-a\right)^+,\max_{y\sim x} \sqrt{\omega(x,y)}\left(u(y)-a\right)^-\right\}.
\end{equation}

Let $u(x)\in \argmin_{a\in \R} \left\{ \max_{y\sim x} \sqrt{\omega(x,y)}\abs{u(y)-a} \right\}$. Assume that
\begin{equation}
\epsilon \coloneqq \max_{y \sim x} \sqrt{\omega(x,y)}(u(y)-u(x))^+ -  \max_{y \sim x} \sqrt{\omega(x,y)}(u(y)-u(x))^->0.
\end{equation}
Then we have
\begin{equation}
\begin{aligned}
&\max_{y \sim x} \sqrt{\omega(x,y)}(u(y)-u(x))^+ -  \max_{y \sim x} \sqrt{\omega(x,y)}(u(y)-u(x))^->\\
&\max_{y \sim x} \sqrt{\omega(x,y)}(u(y)-u(x)-\frac{\epsilon}{2})^+ -  \max_{y \sim x} \sqrt{\omega(x,y)}(u(y)-u(x)-\frac{\epsilon}{2})^-\geq 0.
\end{aligned}
\end{equation}
Since $\max_{y \sim x} \sqrt{\omega(x,y)}(u(y)-u(x))^+>0$ we get
\begin{equation}
\begin{aligned}
\max_{y \sim x} \sqrt{\omega(x,y)}(u(y)-u(x)-\frac{\epsilon}{2})^-\leq &\max_{y \sim x} \sqrt{\omega(x,y)}(u(y)-u(x)-\frac{\epsilon}{2})^+\\
<&\max_{y \sim x} \sqrt{\omega(x,y)}(u(y)-u(x))^+.
\end{aligned}
\end{equation}
Thus we have
\begin{equation}
\max_{y\sim x} \sqrt{\omega(x,y)}\abs{u(y)-u(x)}>\max_{y\sim x} \sqrt{\omega(x,y)}\abs{u(y)-\left(u(x)+\frac{\epsilon}{2}\right)}.
\end{equation}
This is a contradiction to \eqref{eq:discinfharminx}. Therefore, we have ${\Delta_{\omega,\infty}u (x)\leq 0}$. Analogously we get that $\Delta_{\omega,\infty}u (x)\geq 0$.\\

Now let $\Delta_{\omega,\infty}u(x)=0$. Then we have that
\begin{equation}
 \max_{y \sim x} \sqrt{\omega(x,y)}(u(y)-u(x))^+=\max_{y \sim x} \sqrt{\omega(x,y)}(u(y)-u(x))^-=\max_{y \sim x}\sqrt{\omega(x,y)}\abs{u(y)-u(x)}
\end{equation}
Now we get for $a\geq0$ that
\begin{equation}
\begin{aligned}
\max_{y \sim x}\sqrt{\omega(x,y)}\abs{u(y)-(u(x)+a)}\geq&\max_{y \sim x} \sqrt{\omega(x,y)}(u(y)-u(x)-a)^-\\
\geq& \max_{y \sim x} \sqrt{\omega(x,y)}(u(y)-u(x))^-\\
=&\max_{y \sim x}\sqrt{\omega(x,y)}\abs{u(y)-u(x)}.
\end{aligned}
\end{equation}
Analogously we obtain for $a\leq0$ that
\begin{equation}
\max_{y \sim x}\sqrt{\omega(x,y)}\abs{u(y)-(u(x)+a)}\geq\max_{y \sim x}\sqrt{\omega(x,y)}\abs{u(y)-u(x)}.
\end{equation}
Hence it follows \eqref{eq:discinfharminx}.
\end{proof}

In the following we show existence and uniqueness of discrete $\infty$-harmonic extensions for some $g\colon U\to \R$.

\begin{lemma}[Minimum-Maximum Principle] \label{lem:MMP}
Let $g\colon U\to \R$, $u\colon V\to \R$ and $0<\tau\leq 2$. Then $u'\colon V\to \R$ defined by
\begin{equation}
u'(x)\coloneqq
\begin{cases}
g(x), &$for $x\in U,\\
u(x)+\tau \Delta_{\omega,\infty} u(x), &$for $x\in V\backslash U,
\end{cases}
\end{equation}
fulfills
\begin{equation}\label{eq:MMP}
\min_{x\in V} u(x)\leq \min_{x\in V} u'(x) \leq \max_{x\in V} u'(x) \leq \max_{x\in V} u(x).
\end{equation}
\end{lemma}
\begin{proof}
See \cite[Proposition 5.1]{ETT}. For $x_0\in U$ it holds
\begin{equation}
\min_{x\in V} u(x)\leq u(x_0)= u'(x_0)\leq \max_{x\in V} u(x).
\end{equation}
For $x_0\in V\backslash U$, $y\in V$ and $z\in \argmin_{x\in V} u(x)$ it holds
\begin{equation}
\begin{aligned}
u'(x_0)&=u(x_0)+\frac{\tau}{2} \left(\max_{y \sim x_0} \left(\sqrt{\omega(x_0,y)}(u(y)-u(x_0))^+\right) - \left(\max_{y \sim x_0} \sqrt{\omega(x_0,y)}(u(y)-u(x_0))^-\right)\right)\\
&\geq u(x_0)- \frac{\tau}{2}\left(\max_{y \sim x_0} \sqrt{\omega(x_0,y)}(u(y)-u(x_0))^-\right)\\
&\geq u(x_0)- \frac{\tau}{2} (u(z)-u(x_0))^-\\
&=(1-\frac{\tau}{2})u(x_0)+\frac{\tau}{2} u(z)\geq u(z) = \min_{x\in V} u(x).
\end{aligned}
\end{equation}
Analogously we get $u'(x_0)\leq \max_{x\in V} u(x)$. This yields \eqref{eq:MMP}.
\end{proof}

\begin{theorem}\label{thm:dirichlet_unique}
Let $g\colon U\to \R$. Then there exists a unique $v\colon V\to \R$ such that
\begin{equation}\label{eq:dirichlet}
\begin{cases}
\Delta_{\omega,\infty}v(x)=0 &$for $x\in V\backslash U\\
v(x)=g(x) &$for $x\in U
\end{cases}
\end{equation}
\end{theorem}
\begin{proof}
We use the proof from \cite[Theorem 1]{ABDERRAHIM2014153}:\\

\textbf{Uniqueness:} Let $v$ fulfill \eqref{eq:dirichlet}. First we show that there exists no $x\in V\backslash U$ with 
\begin{equation}\label{eq:dir_uniq0}
v(x)\geq v(y)\text{ for all }y\sim x\text{ and }v(x) > v(y)\text{ for one }y\sim x.
\end{equation}\\
Assume such a $x$ exists. Then we have
\begin{equation}
\max_{y\sim x}\left(\sqrt{\omega(x,y)}(v(y)-v(x))^+\right) =0<\max_{y\sim x}\left(\sqrt{\omega(x,y)}(v(y)-v(x))^-\right).
\end{equation}
This yields $\Delta_{\omega,\infty} v(x)<0$. This contradicts \eqref{eq:dirichlet}.\\
Now let $u$ and $v$ fulfill \eqref{eq:dirichlet} and $u\neq v$. Then we can assume without loss of generality that $M=\max_{x\in V\backslash U} (u(x)-v(x))>0$. We define the sets $H\coloneqq \{x\in V:(u(x)-v(x)))=M\}$ and $F=\{x\in H:u(x)=\max_{y\in H} u(y)\}$. By construction it holds that $H$ and $F$ are nonempty.\\
Let $x_0\in H$. Then $u(x_0)-v(x_0)=\max_{x\in V} u(x)-v(x)$. This yields for $y\sim x$ that $u(x_0)-v(x_0)\geq u(y)-v(y)$. Thus we get $v(y)-v(x_0)\geq u(y)-u(x_0)$. This implies
\begin{equation}\label{eq:dir_uniq1}
\max_{y\sim x_0} \left(\sqrt{\omega(x_0,y)}(u(y)-u(x_0))^-\right) \geq \max_{y\sim x_0} \left(\sqrt{\omega(x_0,y)}(v(y)-v(x_0))^-\right)
\end{equation}
and
\begin{equation}\label{eq:dir_uniq2}
\max_{y\sim x_0} \left(\sqrt{\omega(x_0,y)}(u(y)-u(x_0))^+\right) \leq \max_{y\sim x_0} \left(\sqrt{\omega(x_0,y)}(v(y)-v(x_0))^+\right).
\end{equation}
Since $\Delta_{\omega,\infty} u(x_0)=\Delta_{\omega,\infty} v(x_0)$ we have equality in \eqref{eq:dir_uniq1} and \eqref{eq:dir_uniq2}.\\

Let $x\in F$ and $x\sim y\in V\backslash F$. Such $x$ and $y$ exist, because $G$ is connected and $F\subsetneq V$. Then we have either $y\in H\backslash F$ and $u(y)<u(x)$ or $y\not\in H$. Then we have $u(y)-v(y)<u(x)-v(x)$. This yields $u(y)-u(x)<v(y)-v(x)$. Assume that $u(y)\geq u(x)$. Then we have 
\begin{equation}
0\leq \sqrt{\omega(x,y)}(u(y)-u(x))^+=\sqrt{\omega(x,y)}(u(y)-u(x))<\sqrt{\omega(x,y)}(v(y)-v(x))
\end{equation}
Since this holds for all $y\sim x$ with $y\not\in H$ and since for $z\sim x$ with $z\in H$ holds $u(y)\leq u(x)$ we have that
\begin{equation}
\max_{y\sim x}\left(\sqrt{\omega(x,y)}(u(y)-u(x))^+\right)<\max_{y\sim x}\left(\sqrt{\omega(x,y)}(v(y)-v(x))^+\right).
\end{equation}
This contradicts the equality in \eqref{eq:dir_uniq2}. Therefore, we have for $y\in V\backslash F$ with $y\sim x$ that $u(y)<u(x)$. That contradicts that $x$ cannot fulfill \eqref{eq:dir_uniq0}.\\

\textbf{Existence:} Define the compact set
\begin{equation}
A\coloneqq \{ u \colon V\to \R : u=g \text{ on } U \text{ and } \min_{y\in U} g(y)\leq u(x) \leq \max_{y\in U}g(y)\}.
\end{equation} 
Consider the mapping $\Phi \colon A\to A$ defined by 
\begin{equation}
\Phi (u)(x)=
\begin{cases}
g(x), &$for $x\in U,\\
u(x)+\Delta_{\omega,\infty} u(x), &$for $x\in V\backslash U.
\end{cases}
\end{equation}
Due to Lemma \ref{lem:MMP} we have indeed $\Phi (u)\in A$. Since $\Delta_{\omega,\infty}$ is continuous we have also that $\Phi$ is continuous. Hence Brouwers fixed point theorem implies that there exists some $u\in A$ with $\Phi(u)=u$. This yields \eqref{eq:dirichlet}.
\end{proof}

\begin{remark}
In particular, Theorem \ref{thm:dirichlet_unique} and Theorem \ref{thm:disc_infhar_gleich_dirichlet} imply that for $g\colon U\to \R$ there exists a unique discrete $\infty$-harmonic extension, i.e. there exists a unique discrete $\infty$-harmonic function $u\colon V\to \R$ with $u=g$ on $U$.
\end{remark}

To obtain the unique discrete $\infty$-harmonic extension, we consider for $u\colon V\times [0,T] \to \R$ and $g\colon U\to \R$ the following partial differential equation (see \cite{ETT}):
\begin{equation}\label{eq:diffequ}
\begin{cases}
\frac{\partial u(x,t)}{\partial t}=\Delta_{\omega,\infty} u(x) &$for $ x\in V\backslash U,\\
u(x,0)=u^0(x) &$for all $ x\in V,\\
u(x,t)=g(x) &$for $x\in U.
\end{cases}
\end{equation}
We discretize the derivative of $u$ using an explicit Euler scheme.
\begin{equation}
\frac{\partial u(x,t)}{\partial t}\approx\frac{u^{r+1}(x)-u^r(x)}{\Delta t}.
\end{equation}
With $\Delta t=\tau$, this leads to the following iteration scheme:
\begin{equation}\label{eq:iteration_scheme}
\begin{cases}
u^{r+1}(x)=u^{r}(x)+\tau \Delta_{\omega,\infty}u^r(x) &$for $x\in V\backslash U\\
u^{r+1}(x)=g(x) &$for $x\in U
\end{cases}
\end{equation}
Now we can formulate the following algorithm:
\begin{algorithm}[H]
\caption{Discrete $\infty$-harmonic extension}
\label{alg:disc_infty_harm}
\begin{algorithmic}
\State Given: $G=(V,E,\omega)$, $U\subset V$, $g:U \to \R$, $2>\tau>0$ and initial values $u^0\colon V\to \R$ with $u^0=g$ on $U$.
\For {$r=0,1,...$}
	\State $u^{r+1}(x)\coloneqq \begin{cases}g(x),&$for $x\in U,\\u^{r}(x)+\tau \Delta_{\omega,\infty} u^{r}(x),&$for $x\in V\backslash U.\end{cases}$
\EndFor
\end{algorithmic}
\end{algorithm}

\begin{theorem}\label{thm:alg_lim_disc_inft}
Let $g\colon U\to \R$ and $u^0\colon V \to \R$ with $g=u^0$ on $U$. Let $u^r$ be generated by Algorithm \ref{alg:disc_infty_harm}. If $u^r$ converges to some $u^*\colon V\to \R$, then $u^*$ fulfills \eqref{eq:dirichlet} i.e.
\begin{equation}
\begin{cases}
\Delta_{\omega,\infty}u^*(x)=0 &$for $x\in V\backslash U,\\
u^*(x)=g(x) &$for $x\in U.
\end{cases}
\end{equation}
In particular, $u^*$ is the unique discrete $\infty$-harmonic extension of $g$.
\end{theorem}
\begin{proof}
See \cite[Proposition 5.2]{ETT}. For $x\in U$ it holds 
\begin{equation}
u^*(x)=\lim\limits_{r\to \infty} u^r(x)=\lim\limits_{r\to \infty}g(x)=g(x).
\end{equation}
Due to the continuity of $\Delta_{\omega,\infty}$ we have for $x\in V\backslash U$ that
\begin{equation}
\begin{aligned}
u^*(x)=&\lim\limits_{r\to \infty} u^{r+1}(x)\\
=&\lim\limits_{r\to \infty} u^{r}(x)+\tau \Delta_{\omega,\infty}u^r(x)\\
=&u^*(x)+\tau\Delta_{\omega,\infty}\lim\limits_{r\to \infty} u^{r}(x)\\
=&u^*(x)+\tau\Delta_{\omega,\infty}u^*(x).
\end{aligned}
\end{equation}
Thus we have $\Delta_{\omega,\infty}u^*(x)=0$ for $x\in V\backslash U$. 
\end{proof}

To show that Algorithm \ref{alg:disc_infty_harm} converges we need some theory about averaged operators. For more details on averaged operators see \cite{BSS_prox_op}.

\begin{definition}
Consider an operator $T\colon \R^d\to \R^d$. We say that $T$ is \emph{nonexpansive}, if there exists some norm $\Norm{\cdot}{}$ on $\R^d$ such that
\begin{equation}
\Norm{Tx-Ty}{}\leq \Norm{x-y}{} \text{ for all }x,y\in \R^d.
\end{equation}
We say that $T$ is \emph{averaged} with constant $\alpha\in (0,1)$, if there exists a nonexpansive operator $R\colon \R^d\to\R^d$ such that
\begin{equation}
T=\alpha \mathrm{Id} + (1-\alpha)R.
\end{equation}
Further we call $T$ \emph{asymptotically regular}, if we have for all $x\in\R^d$ that
\begin{equation}
\lim\limits_{r\to\infty}T^{r+1} x- T^r x =0.
\end{equation}
\end{definition}

It is easy to check, that every averaged operator is nonexpansive. A proof for this fact is given in \cite[Lemma 4.3]{BSS_prox_op}. We cite some more statements about this definition.

\begin{theorem}[Asymptotic Regularity of Averaged Operators]\label{thm:asym_reg_aver}
Let $T\colon \R^d\to \R^d$ be an averaged operator with constant $\alpha\in (0,1)$. Assume that $Fix(T)\neq\emptyset$. Then, $T$ is asymptotically regular.
\end{theorem}
\begin{proof}
See \cite[Theorem 4.5]{BSS_prox_op}.
\end{proof}

\begin{theorem}[Opial's Convergence Theorem]\label{thm:opials_conv_thm}
Let $T\colon \R^d\to \R^d$ fulfill the following conditions: $Fix(T)\neq \emptyset$, $T$ is nonexpansive with respect to $\abs{\cdot}$ and asymptotically regular. Then, for every $x^0\in\R^d$ the sequence of Picard iterates $(x^r)_r$ generated by $x^{r+1}=T x^r$ converges to an element of $Fix(T)$.
\end{theorem}
\begin{proof}
See \cite[Theorem 1]{opial}.
\end{proof}

For general norms we get the following result:

\begin{theorem}[Krasnoselskii-Mann Iteration]\label{thm:krasnoselski-mann}
Let $T\colon \R^d\to \R^d$ and $(\tau_r)_{r\in \Nat}$ fulfill the following conditions: $Fix(T)\neq \empty$, $T$ is nonexpansive with respect to an arbitrary norm, $\sum_{r=1}^\infty \tau_r=\infty$ and $\limsup_{r\to\infty}\tau_r<1$. Then, for every $x^0\in \R^d$ the sequence $(x^r)_r$ generated by
\begin{equation}
x^{r+1}=\left((1-\tau_r)\mathrm{Id}+\tau_r T\right)x^r
\end{equation}
converges to an element of $Fix(T)$.
\end{theorem}
\begin{proof}
See \cite[Corollaries 10,11]{BRS92}.
\end{proof}

\begin{remark}[Convergence of the Picard iteration for averaged operators]\label{rem:picard-iteration}
Let $T\colon \R^d \to \R^d$ be averaged, i.e. $T=\alpha \mathrm{Id}+(1-\alpha) R$ for some nonexpansive operator $R\colon \R^d \to \R^d$. Then for $\tau\coloneqq 1-\alpha$, we can rewrite the Picard iteration $x^{r+1}=T x^r$ by
\begin{equation}
x^{r+1}=\left(\alpha \mathrm{Id}+(1-\alpha) R\right)x^r=\left((1-\tau)\mathrm{Id}+\tau T\right)x^r.
\end{equation}
Hence, the Picard iteration converges by Theorem \ref{thm:krasnoselski-mann} to an element of $Fix(R)=Fix(T)$.
\end{remark}

Now we can show convergence of Algorithm \ref{alg:disc_infty_harm}. Since \cite{ETT} only states stability of Algorithm \ref{alg:disc_infty_harm}, we did the proof ourself.

\begin{theorem}
Let $g\colon U\to \R$, $f^0\colon V\to \R$ with $f^0=g$ on $U$ and $0<\tau<1$. Let $f^n$ be generated by Algorithm \ref{alg:disc_infty_harm}. Then $f^n$ converges to the unique discrete $\infty$-harmonic extension $f^*\colon V\to \R$ of $g$ as $n\to \infty$.
\end{theorem}
\begin{proof}
For $A\coloneqq \{ f \colon V\to \R:f=g \text{ on }U\}$ we consider the operator $\Phi \colon A\to A$ defined by 
\begin{equation}
\Phi (f)(x)=
\begin{cases}
g(x), &$for $x\in U,\\
f(x)+\Delta_{\omega,\infty} f(x), &$for $x\in V\backslash U.
\end{cases}
\end{equation}
First we show, that $\Phi$ is nonexpansive. For $v\in V$ we use the notation $N(v)=\{v\}\cup\{u\in V:u\sim v\}$. For $x\in V$ we can rewrite $\Delta_{\omega,\infty} f(x)$ by
\begin{equation}
\Delta_{\omega,\infty} f(x)=\frac{1}{2} \left(\max_{v \in N(x)} \left(\sqrt{\omega(x,v)}(f(v)-f(x))\right) - \left(\max_{v \in N(x)} \sqrt{\omega(x,v)}(f(x)-f(v))\right)\right).
\end{equation}
For $f_1,f_2\in A$ define
\begin{equation}
y_i\in \argmax_{v \in N(x)} \left(\sqrt{\omega(x,v)}(f_i(v)-f_i(x))\right)
\end{equation}
and
\begin{equation}
z_i\in \argmax_{v \in N(x)} \left(\sqrt{\omega(x,v)}(f_i(x)-f_i(v))\right).
\end{equation}
Therefore we have
\begin{equation}
\begin{aligned}
\Phi(f_1)(x)-\Phi(f_2)(x)=&f_1(x)-f_2(x)+\Delta_{\omega,\infty} f_1(x)-\Delta_{\omega,\infty} f_2(x)\\
\leq& f_1(x)-f_2(x) + \frac{\sqrt{\omega(x,y_1)}}{2} (f_1(y_1)-f_1(x)-f_2(y_1)+f_2(x))-\\
&\frac{\sqrt{\omega(x,z_2)}}{2}(f_1(x)-f_1(z_2)-f_2(x)+f_2(z_2))\\
=& \left(1-\frac{\sqrt{\omega(x,y_1)}+\sqrt{\omega(x,z_2)}}{2}\right)(f_1(x)-f_2(x)) + \\
&\frac{\sqrt{\omega(x,y_1)}}{2} (f_1(y_1)-f_2(y_1))+\frac{\sqrt{\omega(x,z_2)}}{2} (f_1(z_2)-f_2(z_2))\\
\leq&\left(1-\frac{\sqrt{\omega(x,y_1)}+\sqrt{\omega(x,z_2)}}{2}\right)\Norm{f_1-f_2}{\infty} + \\
&\frac{\sqrt{\omega(x,y_1)}}{2} \Norm{f_1-f_2}{\infty}+\frac{\sqrt{\omega(x,z_2)}}{2} \Norm{f_1-f_2}{\infty}\\
=&\Norm{f_1-f_2}{\infty}.
\end{aligned}
\end{equation}
Analogously we get that 
\begin{equation}
\Phi(f_1)(x)-\Phi(f_2)(x)\geq -\Norm{f_1-f_2}{\infty}
\end{equation}
Hence we have that
\begin{equation}
\Norm{\Phi(f_1)-\Phi(f_2)}{\infty}\leq \Norm{f_1-f_2}{\infty}
\end{equation}
and that $\Phi$ is nonexpansive.
The $f^n$ generated by Algorithm \ref{alg:disc_infty_harm} fulfill by definition
\begin{equation}
f^{n+1}=T f^n\text{, where }T=(1-\tau)Id+\tau \Phi
\end{equation}
The operator $T$ is averaged for the constant $1-\tau\in(0,1)$ and therefore nonexpansive. Due to Theorem \ref{thm:dirichlet_unique} we have $Fix(T)\neq \emptyset$. Thus Remark \ref{rem:picard-iteration} yields that the sequence $(f^n)_n$ converges. Due to Theorem \ref{thm:alg_lim_disc_inft} the limit is discrete $\infty$-harmonic.
\end{proof}

\begin{example}
For $\tau=2$ the $u^r$ generated by Algorithm \ref{alg:disc_infty_harm} may not converge: Let $G=(V,E,\omega)$ be given by $V=\{x_1,...,x_5\}$, $E=\{(x_i,x_{i+1}:i\in\{1,...,4\}\}$ and $\omega(x,y)=1$ for $(x,y)\in E$. Further let $U=\{x_1,x_5\}$ and let $g\colon U \to \R$ be given by $g(x)=0$ for $x\in U$ . Then the $u^r$ generated by Algorithm \ref{alg:disc_infty_harm} with initial values $u^0\colon V\to \R$ given by $u^0(x_3)=1$ and $f^0(x)=0$ for $x\in V\backslash \{x_3\}$ fulfill $u^r=u^0$ if $r$ is even and $u^r(x)=1$ for $x\in \{x_2,x_4\}$ and $u^r(x)=0$ for $x\in \{x_1,x_3,x_5\}$ if $r$ is odd. Hence the $u^r$ diverges for $r\to \infty$.
\end{example}

For $1\leq \tau <2$ we can neither show convergence of Algorithm \ref{alg:disc_infty_harm} nor we know an counterexample for the convergence.

\subsection{Tight functions on graphs}\label{sec:discrete_tight}

\begin{example}\label{ex:disc_inf_harm_nicht_eind}
In general, discrete $\infty$-harmonic extensions of a function $g\colon U\to \R^m$ are not unique for $m>1$. The idea of a counterexample comes from \cite[Section 2.1]{sheffield2012vector}. We choose $V=\{x_1, ..., x_6\}$, $U=\{x_1,x_2,x_3\}$ and 
\begin{equation}
E=\{ (x_1,x_4), (x_2,x_5), (x_3,x_6), (x_4,x_5), (x_5,x_6), (x_4,x_6)\}
\end{equation}
and use the weighting function $\omega(x,y)=1$ for $(x,y)\in E$. Then the function $g\colon U \to \R^2$ defined by $x_1\mapsto (0,0)$, $x_2\mapsto (0,1)$ and $x_3\mapsto \left(\frac{\sqrt{3}}{2},\frac{1}{2}\right)$ has more than one discrete $\infty$-harmonic extension. For example, consider $u_1, u_2\colon V\to \R^2$ defined by $u_i(x)=g(x)$ on $U$ and $u_1(x_4)=\left(\frac{\sqrt{3}}{7},\frac{2}{7}\right)$, $u_1(x_5)=\left(\frac{\sqrt{3}}{14},\frac{9}{14}\right)$ and $u_1(x_6)=\left(\frac{2 \sqrt{3}}{7},\frac{4}{7}\right)$ and $u_2(x_4)=\left(\frac{1}{2+2 \sqrt{3}},\frac{\sqrt{3}}{2+2 \sqrt{3}}\right)$, $u_2(x_5)=\left(\frac{1}{2+2 \sqrt{3}},\frac{1}{4} + \frac{\sqrt{3}}{4}\right)$ and $u_2(x_6)=\left(\frac{1}{2},\frac{1}{2}\right)$. See Figure \ref{fig:disc_inf-harm_not_unique}. It is easy to verify that $u_1$ and $u_2$ are discrete $\infty$-harmonic.
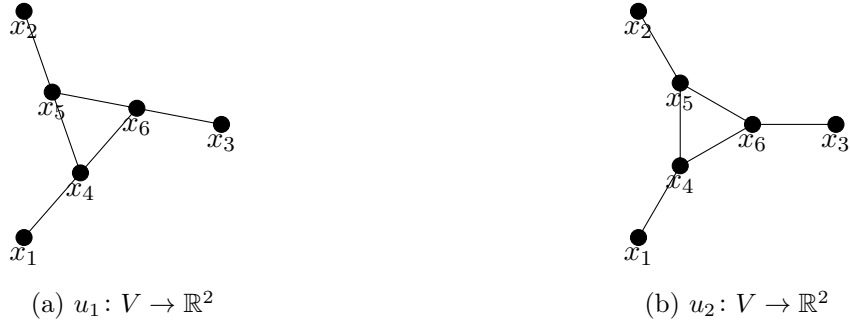
\begin{figure}
\centering
\begin{subfigure}[t]{0.45\textwidth}
\centering
\begin{tikzpicture}[scale=3]
\filldraw (0,0) circle (1pt) node[align=center, below] {$x_1$} -- (0.2474358,0.2857143) circle (1pt) node[align=center, below] {$x_4$} -- (0.1237179,9/14) circle (1pt) node[align=center, below] {$x_5$} -- (0,1) circle (1pt) node[align=center, below] {$x_2$};
\filldraw (0.2474358,0.2857143) -- (0.4948716,0.5714286) circle (1pt) node[align=center, below] {$x_6$} -- (0.1237179,9/14);
\filldraw (0.4948716,0.5714286) -- (0.8660254,0.5) circle (1pt) node[align=center, below] {$x_3$};
\end{tikzpicture}
\caption{$u_1\colon V \to \R^2$}
\end{subfigure}\hfill
\begin{subfigure}[t]{0.45\textwidth}
\centering
\begin{tikzpicture}[scale=3]
\filldraw (0,0) circle (1pt)  node[align=center, below] {$x_1$} -- (0.1830127,0.3169873) circle (1pt) node[align=center, below] {$x_4$} -- (0.1830127,0.6830127) circle (1pt) node[align=center, below] {$x_5$} -- (0,1) circle (1pt) node[align=center, below] {$x_2$};
\filldraw (0.1830127,0.3169873) -- (0.5,0.5) circle (1pt) node[align=center, below] {$x_6$} -- (0.1830127,0.6830127);
\filldraw (0.5,0.5) -- (0.8660254,0.5) circle (1pt) node[align=center, below] {$x_3$};
\end{tikzpicture}
\caption{$u_2\colon V \to \R^2$}
\label{fig:disc_inf-harm_not_unique_b}
\end{subfigure}\hfill
\caption{Two discrete $\infty$-harmonic extensions of $g\colon \{x_1,x_2,x_3\}\to \R^2$.}
\label{fig:disc_inf-harm_not_unique}
\end{figure}
\end{example}

Definition \ref{def:local_Lipschitz} yields a discrete formulation of a tight functions on graphs. This approach is considered in \cite{sheffield2012vector}.

\begin{definition}
Let $u,v\colon V\to \R^m$ with $u\mid_U=v\mid_U$. We say $u$ is \emph{tighter} than $v$ if $u$ and $v$ satisfy
\begin{equation}
\max \{S v(x): x\in V\backslash U \text{ and } S u(x) < S v(x)\} > \max \{S u(x): x\in V\backslash U \text{ and } S v(x) < S u(x)\}
\end{equation}
We say $u$ is \emph{tight} if there is no tighter $v\colon V\to \R^m$ with $u\mid_U=v\mid_U$.
\end{definition}

\begin{example}
For the choice of the graph $G=(V,E,\omega)$, the subset $U\subset V$ and $g\colon U\to \R^2$ as in Example \ref{ex:disc_inf_harm_nicht_eind} the tight extension is given by the function $u_2\colon V\to \R^2$. See Figure \ref{fig:disc_inf-harm_not_unique_b}. In particular, the tight extension of $g$ is discrete $\infty$-harmonic. Theorem \ref{thm:tight_implies_disc_infharm} shows this result in general.
\end{example}

\begin{definition}
Let $u\colon V\to \R^m$. We denote the values of $S u$ in nonincreasing order by $l(u)=(l_i(u))_{i=1}^{\abs{V\backslash U}}\in \R^{\abs{V\backslash U}}$. 
\end{definition}

\begin{lemma}\label{lem:lexmin_tight}
Let $g\colon U\to \R^m$ and let $u,v\colon V\to \R^m$ be tight extensions of $g$ such that $u$ is tighter than $v$. Then, $l(u)<l(v)$ in lexicographically order. If $l(u)$ is lexicographically minimal, then $u$ is tight.
\end{lemma}
The statement is used in \cite[Theorem 1.2]{sheffield2012vector}. For convenience we provide a proof.
\begin{proof}
Choose $x_0\in V\backslash U$ such that 
\begin{equation}
\begin{aligned}
S v(x_0)&=\max \{S v(x): x\in V\backslash U \text{ and } S u(x) < S v(x)\} \\
&> \max \{S u(x): x\in V\backslash U \text{ and } S v(x) < S u(x)\}
\end{aligned}
\end{equation}
Then we have for all $x\in V\backslash U$ with $S v(x)>S v(x_0)$ that $S v(x)= S u(x)$ and $S v(x_0)> S u(x_0)$. Further it holds that for all $x\in V\backslash U$ with $S v(x)<S v(x_0)$ that $S u(x)< S v(x_0)$. This yields by definition that $l(u)<l(v)$ in lexicographically order.\\

Now let $l(u)$ be lexicographically minimal. Assume that $u$ is not tight. Then there exists some $w\colon V\to \R^m$ with $w=g$ on $U$ such that $w$ is tighter than $u$. This yields by the first part of the proof that $l(w)<l(u)$ in lexicographically order. This contradicts the fact, that $l(u)$ is lexicographically minimal. Hence $u$ is tight.
\end{proof}

\begin{example}
Let $u,v\colon V\to \R^m$ with $l(u)<l(v)$ in lexicographically order. Then it does not follow that $u$ is tighter than $v$. We give a counterexample:\\
Let $G=(V,E,\omega)$ with $V=\{x_1,x_2,x_3\}$, $U=\{x_2\}$ and $E=\{(x_1,x_2),(x_2,x_3)\}$. Let $\omega(x,y)=1$ for $(x,y)\in E$. We define $u,v\colon V\to \R$ by $u(x_1)=0$, $u(x_2)=2$, $u(x_3)=5$, $v(x_1)=5$, $v(x_2)=2$ and $v(x_3)=1$. Then it holds
\begin{equation}
\max \{S v(x): x\in V\backslash U \text{ and } S u(x) < S v(x)\} = \max \{S u(x): x\in V\backslash U \text{ and } S v(x) < S u(x)\}.
\end{equation}
Hence neither $u$ is tighter than $v$ nor $v$ is tighter than $u$, although $l(u)< l(v)$ in lexicographically order.
\end{example}

\begin{theorem}\label{thm:tight_ex}
Let $g\colon U\to \R^m$. Then there exists an unique tight extension $u\colon V\to \R^m$ of $g$.\\
Further $u$ is tighter than every other extension $v\colon V\to \R^m$ of $g$. 
\end{theorem}

The proof of the existence comes from \cite[Theorem 1.2]{sheffield2012vector}. We added some details. Since we did not understand the proof of the uniqueness, we worked it out again using similar arguments.

\begin{proof}
\textbf{Existence:} We consider the set $\mathcal{M}=\{l(u): u\colon V\to \R^m \text{ with } u\mid_{U}=g\}\subset [0,\infty)^{\abs{V\backslash U}}$. $\mathcal{M}$ is closed due to the continuity of $l$. We show inductively that the sets 
\begin{equation}
\mathcal{M}_k =\{u\colon V\to \R^m \text{ with } u\mid_{U}=g\text{ and } (l_i(u))_{i=1}^k \text{ lexicographically minimal}\}
\end{equation}
are nonempty for $k=1,...,\abs{V\backslash U}$ and that $\{l(u):u\in M_k\}$ is closed.\\

$k=1$: Since $\mathcal{M}$ is nonempty and closed we have that also $\{l_1(u): u\colon V\to \R^m \text{ with } u\mid_{U}=g\}\subset [0,\infty)$ is nonempty and closed. Hence it has a minimum $m_1$ and the set of minimizers $\mathcal{M}_1$ is nonempty. Further $\{l(u): u\in M_1\}=\mathcal{M}\cap \{m_1\}\times \R^{\abs{V\backslash U}-1}$ is closed.

$k>1$: Since the set $\{l(u):u\in \mathcal{M}_{k-1}\}$ is nonempty and closed we have that also $\{l_k(u):u\in \mathcal{M}_{k-1}\}$ is nonempty and closed. Hence it has a minimum $m_k$ and the set of minimizers $\mathcal{M}_k$ is nonempty. Further $\{l(u): u\in \mathcal{M}_k\}=\mathcal{M}\cap \{m_1\}\times \dots \times \{m_k\} \times \R^{\abs{V\backslash U}-k}$ is closed.\\

In particular, the set $\mathcal{M}_{\abs{V\backslash U}}$ of lexicographically minimal extensions of $g$ is nonempty. By Lemma \ref{lem:lexmin_tight} these lexicographically minimal extensions of $g$ are tight.\\

\textbf{Uniqueness:} Let $u$ be tight and let $v\colon V\to \R^m$ be an arbitrary extension of $g$ with $u\neq v$. Assume $u$ is not tighter than $v$. Since $u$ is tight, we have that $v$ is also not tighter than $u$. Thus we have
\begin{equation}\label{eq:Kdef_uniquetight}
\begin{aligned}
K\coloneqq& \max \{ S u(x):x\in V\backslash U \text{ and }  S u(x)\neq S v(x)\}\\
=& \max \{ S v(x):x\in V\backslash U \text{ and } S u (x) \neq S v(x)\}.
\end{aligned}
\end{equation} 
Set $K\coloneqq0$ if $S u (x)=S v(x)$ for all $x\in V\backslash U$. Consider $w\coloneqq \frac{u+v}{2}$. For $x\sim y$ it holds
\begin{equation}\label{eq:tight_ex1}
\sqrt{\omega(x,y)}\abs{w(x)-w(y)}\leq \frac{1}{2} \sqrt{\omega(x,y)} \abs{u(x)-u(y)}+\frac{1}{2} \sqrt{\omega(x,y)} \abs{v(x)-v(y)}.
\end{equation}
If $S u(x)\geq K$ for some $x\in V\backslash U$, then we have either $S v(x)=S u(x) > K$ or $S v(x)\leq K = S u(x)$. Together with \eqref{eq:tight_ex1} we get
\begin{equation}\label{eq:wtight}
S w(x) \leq \frac{S u (x) + S v(x)}{2} \leq S u(x).
\end{equation}
If it holds $S w(x)<S u(x)$ for some $x\in V\backslash U$ with $S u(x)\geq K$ we that $S w (x)\leq S u(x)$ for all $x\in V\backslash U$ with strict inequality for one $x$. This yields that $w$ is tighter than $u$ and contradicts that $u$ is tight. Hence we have $S w(x)=S u(x)$ for all $x\in V\backslash U$ with $S u(x)\geq K$. If we insert this in \eqref{eq:wtight} we get
\begin{equation}
S w(x) \leq \frac{S w (x) + S v(x)}{2} \leq S w(x)
\end{equation}
and therefore we have $S u(x)=S w(x)= S v(x)$ for all $x\in V\backslash U$ with $S u(x)\geq K$. Hence $S u(x) \neq S v(x)$ implies $S u(x)<K$. This yields that
\begin{equation}
\max \{S u: S u\neq S v\}<K,
\end{equation}
what contradicts the definition of $K$ in \eqref{eq:Kdef_uniquetight} if $S u \not\equiv S v$. Hence $S u \equiv S v$ and $K=0$. This yields $S w(x)=S v(x)=S u(x)$ for all $x\in V\backslash U$. Thus also $l(w)$ is lexicographically minimal and $w$ is tight.\\

Define the set $W\coloneqq \{x\in V: u(x)=v(x)\}\supseteq U$. Because $u\neq v$, we have $W\subsetneq V$. Since $S u=S v=S w$ on $V\backslash U$ we also have that $u$, $v$ and $w$ are tight extensions of $\tilde{g}\colon W \to \R^m$ defined by $\tilde{g}(x)=u(x)=v(x)$ for all $x\in W$.\\

We show that there exists some $x\in V\backslash W$ and $y\in W$ such that $x\sim y$ and
\begin{equation}
\sqrt{\omega(x,y)}\abs{w(x)-w(y)}=S w(x).
\end{equation}
Assume that such $x$ and $y$ do not exist.Then there exists for all $x\in V\backslash W$ some $0<\epsilon_x<1$ such that $\sqrt{\omega(x,y)}\abs{(1-\epsilon_x)w(x)-w(y)}< S w(x)$ for all $y\in W$ with $y\sim x$. Set $\epsilon=\min_{x\in V\backslash W} \epsilon_x$. Then we have for all $x\in V\backslash W$ that 
\begin{equation}\label{eq:tight_ex2}
\abs{(1-\epsilon)w(x)-w(y)}< S w(x) \text{ for all } y\in W \text{ with } y\sim x
\end{equation}
and $\epsilon$ independent of $x$. Consider the extension $w_\epsilon \colon V \to \R^m$ of $g$ defined by
\begin{equation}
w_\epsilon(x)\coloneqq
\begin{cases}
w(x)=\tilde{g}(x)&$if $x\in W,\\
(1-\epsilon)w(x)&$if $x\in V\backslash W.
\end{cases}
\end{equation}
Now we have by \eqref{eq:tight_ex2} for $x\in V\backslash W$ and $y\sim x$ with $y\in W$ that 
\begin{equation}
\sqrt{\omega(x,y)}\abs{w_\epsilon(x)-w_\epsilon(y)}=\sqrt{\omega(x,y)}\abs{(1-\epsilon)w(x)-w(y)}<S w(x).
\end{equation}
For $y\sim x$ with $y\in V\backslash W$ it holds 
\begin{equation}
\begin{aligned}
\sqrt{\omega(x,y)}\abs{w_\epsilon(x)-w_\epsilon(y)}&=\sqrt{\omega(x,y)}(1-\epsilon)\abs{w(x)-w(y)}\\&<\sqrt{\omega(x,y)}\abs{w(x)-w(y)}\leq S w(x).
\end{aligned}
\end{equation} 
Hence $S w_\epsilon (x)<S w (x)$ for all $x\in V\backslash W$. Thus $w_\epsilon$ is tighter than $w$, what contradicts that $w$ is tight.\\

Therefore, there exists some $x\in V\backslash W$ and $y\in W$ such that $x\sim y$ and $\abs{w(x)-w(y)}=S w(x)$. Now it holds
\begin{equation}\label{eq:tight_uniqueness_ineq}
\begin{aligned}
S w(x)&=\sqrt{\omega(x,y)}\abs{w(x)-w(y)}\\
&\leq \frac{1}{2} \sqrt{\omega(x,y)}\abs{u(x)-u(y)}+\frac{1}{2} \sqrt{\omega(x,y)}\abs{v(x)-v(y)}\\
&\leq \frac{1}{2} S u(x)+\frac{1}{2} S v(x)=S w(x).
\end{aligned}
\end{equation}
Thus we must have equality in \eqref{eq:tight_uniqueness_ineq} and get
\begin{equation}
\sqrt{\omega(x,y)}\abs{u(x)-u(y)}=S u(x)=S v(x) =\sqrt{\omega(x,y)}\abs{v(x)-v(y)}.
\end{equation}
Plugging in the definition of $w$ we get with the equality of \eqref{eq:tight_uniqueness_ineq} that
\begin{equation}
\abs{u(x)-u(y)+v(x)-v(y)}=2\abs{w(x)-w(y)}=\abs{u(x)-u(y)}+\abs{v(x)-v(y)}.
\end{equation}
Therefore, $u(x)-u(y)$ and $v(x)-v(y)$ are linearly depended and have the same modulus. This yields $u(x)-u(y)=-(v(x)-v(y))$ or $u(x)-u(y)=v(x)-v(y)$.\\

If it holds $u(x)-u(y)=-(v(x)-v(y))$, then we have by definition $w(x)-w(y)=0$. Due to \eqref{eq:tight_uniqueness_ineq} we get that $Su(x)=Sv(x)=Sw(x)=0$ and therefore $\abs{u(x)-u(y)}=\abs{v(x)-v(y)}=0$. Thus we have $u(x)-u(y)=0=v(x)-v(y)$. Since it holds $y\in W$, we have $u(y)=v(y)$ and $u(x)=u(y)=v(y)=v(x)$. This contradicts that $x\not\in W$.\\

If it holds $u(x)-u(y)=v(x)-v(y)$, it holds $u(x)-v(x)=u(y)-v(y)=\tilde{g}(y)-\tilde{g}(y)=0$. This yields $u(x)=v(x)$ and contradicts that $x\not\in W$.\\

This yields $W=V$ and $u=v$ on $V$.
\end{proof}

\begin{remark}\label{rem:tight_lexmin}
In the proof of Theorem \ref{thm:tight_ex} is shown that there exists an extension $u\colon V\to \R^m$ of $g\colon U\to \R^m$ with lexicographically minimal $l(u)$. Further this $u$ is tight by Lemma \ref{lem:lexmin_tight}. Further, the tight extension is unique by Theorem \ref{thm:tight_ex}. This yields that an extension $u\colon V\to \R^m$ of $g\colon U\to \R^m$ is tight if and only if $l(u)$ is lexicographically minimal. 
\end{remark}

The proof of Theorem \ref{thm:tight_ex} is not constructive. To approximate the unique tight extension of $v\colon U\to \R^m$ we consider for $p \geq 1$ the energy functionals 
\begin{equation}
I_p\colon \{u\colon V\to \R^m: u=g \text{ on } U\}\to \R\text{ defined by }I_p(u)=\sum_{x\in V\backslash U} (S u(x))^p
\end{equation}
for $1<p<\infty$.

\begin{remark}
It is easy to show, that the functionals $I_p$ are continuous and coercive. Hence for $1\leq p<\infty$ there exists a minimizer $u_p$ of $I_p$.
\end{remark}

\begin{remark}
Let $g\colon U\to \R^m$. For $u\colon V\to \R^m$ with $u=g$ on $U$ it holds $I_p(u)=\Norm{(S u)\mid_{V\backslash U}}{p}^p$. The naive way to consider the minimizers of $I_p$ as $p\to \infty$ would be minimizing the functional
\begin{equation}
I_\infty(u)=\Norm{(Su)\mid_{V\backslash U}}{\infty}.
\end{equation}
But the minimizers of $I_\infty$ are in general not unique as shown in Example \ref{ex:I_inf_not_uni} below. But since $u\colon V\to \R^m$ is tight if and only if $l(u)$ is lexicographically minimal, we get that the unique tight extension of $g$ is one of the minimizers of $I_\infty$.
\end{remark}

\begin{example}\label{ex:I_inf_not_uni}
Minimizers of $I_\infty$ are not unique:\\
Let $G=(V,E,\omega)$ be given by $V=\{x_1,...,x_4\}$, $E=\{(x_1,x_2),(x_2,x_3),(x_2,x_4)\}$ and $\omega(x,y)=1$ for $(x,y)\in E$. See Figure \ref{fig:examplegraphs1}. Further, let $U=\{x_1,x_3\}$ and let $g\colon U \to \R$ be given by $g(x_1)=1$ and $g(x_3)=-1$. For $u\colon V\to \R$ with $u=g$ on $U$ it holds
\begin{equation}
I_\infty (u)\geq S u(x_2) \geq \frac{1}{2} (\abs{u(x_1)-u(x_2)}+\abs{u(x_2)-u(x_3)})\geq \frac{1}{2}\abs{u(x_1)-u(x_3)}=1.
\end{equation}
Now define $u,v\colon V\to\R$ by $u=v=g$ on $U$ and $u(x_2)=v(x_2)=0$, $u(x_4)=0$ and $v(x_4)=1$. Then we have $S u(x_2)=S v(x_2)=1$, $S u(x_4)=0$ and $S v(x_4)=1$. This yields $I_\infty(u)=I_\infty(v)=1$. Hence $u$ and $v$ are minimizers of $I_\infty$. But we have $I_p(u)=1<2=I_p(v)$ for all $1\leq p<\infty$. 
\end{example}

\begin{example}\label{ex:Ip_notstrconv}
As the concatenation of convex functionals, $I_p$ is also convex for $1\leq p<\infty$. But in general, it is not strictly convex. We give a counterexample:
Let $V=\{x_1,...,x_4\}$, $E=\{(x_1,x_2),(x_2,x_3),(x_3,x_4)\}$ and $\omega(x,y)=1$ for $(x,y)\in E$. See Figure \ref{fig:examplegraphs2}. Let $U=\{x_2\}$ and $g\colon U\to \R$ defined by $g(x_2)=0$. Consider the functions $u,v\colon V\to \R$ defined by $u(x_1)=v(x_1)=1$ $u(x_2)=v(x_2)=0$, $u(x_3)=1$, $u(x_4)=3$, $v(x_3)=-1$ and $v(x_4)=1$. Then we get by a short computation that $I_p(u)=I_p(v)=I_p\left(\frac{1}{2}(u+v)\right)=2^{p+1}+1$. Hence $I_p$ is not strictly convex for $1\leq p\leq \infty$.\\
The same counterexample shows that also the energy functional $\overline{I_p}$ defined by
\begin{equation}
\overline{I_p}(u)=\sum_{x\in V} (S u(x))^p
\end{equation}
is not strictly convex.
\end{example}

\begin{figure}
\centering
\begin{subfigure}[t]{0.45\textwidth}
\centering
\begin{tikzpicture}[scale=2]
\filldraw (1,0) circle (1pt) node[align=center, below] {$x_1$} -- (0,0) circle (1pt) node[align=center, below] {$x_2$} -- (-1,0) circle (1pt) node[align=center, below] {$x_3$};
\filldraw (0,0) -- (0,1) circle (1pt) node[align=center, below] {$x_4$};
\end{tikzpicture}
\caption{Graph from Example \ref{ex:I_inf_not_uni}}
\label{fig:examplegraphs1}
\end{subfigure}\hfill
\begin{subfigure}[t]{0.45\textwidth}
\centering
\begin{tikzpicture}[scale=1.5]
\filldraw (0,0) circle (1pt)  node[align=center, below] {$x_1$} -- (1,0) circle (1pt) node[align=center, below] {$x_2$} -- (2,0) circle (1pt) node[align=center, below] {$x_3$} -- (3,0) circle (1pt) node[align=center, below] {$x_4$};
\end{tikzpicture}
\caption{Graph from Example \ref{ex:Ip_notstrconv}}
\label{fig:examplegraphs2}
\end{subfigure}\hfill
\caption{Graphs from the Examples \ref{ex:I_inf_not_uni} and \ref{ex:Ip_notstrconv}.}
\label{fig:examplegraphs}
\end{figure}
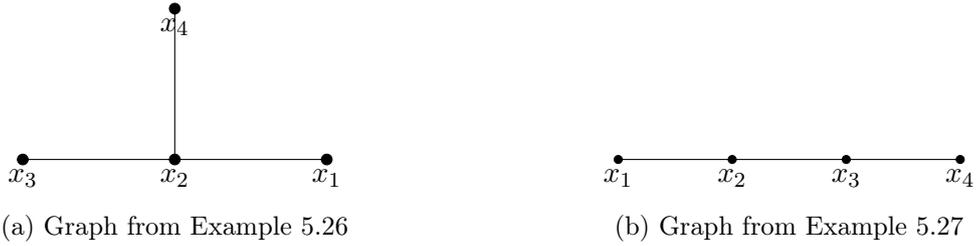

\begin{theorem}
Let $g\colon U\to \R^m$ and $1\leq p<\infty$. Then, the functional $I_p$ has a unique minimizer.
\end{theorem}
\begin{proof}
Assume that $u\neq v\colon V\to \R^m$ with $u=v=g$ on $U$ are minimizers of $I_p$. Set $w\coloneqq \frac{1}{2} (u+v)$. Obviously $S \cdot (x)$ is convex. If there exists some $x\in V\backslash U$ with $S u(x)\neq S v(x)$, then the strict convexity of $\abs{\cdot}^p$ yields
\begin{equation}
\begin{aligned}
I_p(w)=I_p\left(\frac{1}{2} u + \frac{1}{2}v\right)&=\sum_{x\in V\backslash U} \left(S \left(\frac{1}{2} u + \frac{1}{2} v\right)(x)\right)^p\\
&\leq \sum_{x\in V\backslash U} \left(\frac{1}{2} S u (x) + \frac{1}{2} S v (x)\right)^p\\
&< \sum_{x\in V\backslash U} \frac{1}{2} (S u (x))^p + \frac{1}{2} (S v (x))^p\\
&=\frac{1}{2} I_p (u) + \frac{1}{2} I_p(v).
\end{aligned}
\end{equation}
This contradicts the assumption, that $u$ and $v$ are minimizers of $I_p$. \\

Hence we have $S u(x)=S v(x)=S w(x)$ for all $x\in V\backslash U$. Define the set $W\coloneqq \{x\in V:u(x)=v(x)\}$. Since $u\neq v$, we have $W\subsetneq V$.\\

We show that either there exists some $x\in V\backslash W$ and $y\in W$ such that it holds
\begin{equation}\label{eq:ipmin1}
\sqrt{\omega(x,y)}\abs{w(x)-w(y)}=Sw(x)
\end{equation}
or there exists $x\in V\backslash W$ and $y\in W\backslash U$ such that it holds
\begin{equation}\label{eq:ipmin2}
\sqrt{\omega(x,y)}\abs{w(x)-w(y)}=Sw(y).
\end{equation}
Assume that such $x$ and $y$ do not exist. Then there exists for all $x\in V\backslash W$ some $0<\epsilon_x<1$ such that $\sqrt{\omega(x,y)}\abs{(1-\epsilon_x) w(x)-w(y)}<S w(x)$ for all $y\in W$ with $y\sim x$. Further, there exists for all $y\in W\backslash U$ some $0<\epsilon_y<1$ such that $\sqrt{\omega(x,y)}\abs{(1-\epsilon_y) w(x)-w(y)}<S w(y)$ for all $x\in V\backslash W$ with $x\sim y$. Set $\epsilon=\min \left\{ \min_{x\in V\backslash W} \epsilon_x, \min_{y\in W\backslash U} \epsilon_y\right\}$.\\

Consider the extension $w_\epsilon\colon V\to \R^m$ of $g$ defined by
\begin{equation}
w_\epsilon (x) \coloneqq 
\begin{cases}
w(x),&$if $x\in W,\\
(1-\epsilon)w(x),&$if $x\in V\backslash W.
\end{cases}
\end{equation}

Now we have for $x\in V\backslash W$ and $y\in W$ with $y\sim x$ that
\begin{equation}
\sqrt{\omega(x,y)}\abs{w_\epsilon(x)-w_\epsilon(y)}=\sqrt{\omega(x,y)}\abs{(1-\epsilon)w(x)-w(y)}<S w(x).
\end{equation}
Further, we have for $y\in W\backslash U$ and $x\in V\backslash W$ that
\begin{equation}
\sqrt{\omega(x,y)}\abs{w_\epsilon(x)-w_\epsilon(y)}=\sqrt{\omega(x,y)}\abs{(1-\epsilon)w(x)-w(y)}<S w(y).
\end{equation}
For $y\sim x$ with $y\in V\backslash W$ and $x\in V\backslash W$ it holds
\begin{equation}
\sqrt{\omega(x,y)}\abs{w_\epsilon(x)-w_\epsilon(y)}=\sqrt{\omega(x,y)}(1-\epsilon)\abs{w(x)-w(y)}\leq (1-\epsilon) S w(x)<S w(x).
\end{equation}
Hence $S w_\epsilon (x) < S w(x)=S u(x)$ on $V\backslash W$ and $S w_\epsilon (x)\leq S w (x) =S u(x)$ on $W\backslash U$. Thus it holds $I_p (w_\epsilon)<I_p(u)$. This contradicts the assumption that $u$ minimizes $I_p$.\\

Therefore, we find $x$ and $y$ with \eqref{eq:ipmin1} or \eqref{eq:ipmin2}. If $x$ and $y$ fulfill \eqref{eq:ipmin1}, then it holds
\begin{equation}\label{eq:ipmin3}
\begin{aligned}
Sw(x)&=\sqrt{\omega(x,y)}\abs{w(x)-w(y)}\\
&\leq \frac{1}{2}\left(\sqrt{\omega(x,y)}\abs{u(x)-u(y)}+\sqrt{\omega(x,y)}\abs{v(x)-v(y)}\right)\\
&\leq \frac{1}{2} (Su(x)+Sv(x)).
\end{aligned}
\end{equation}
Since $Su(x)=Sv(x)=Sw(x)$ we have equality in \eqref{eq:ipmin3}. Since $w=\frac{1}{2}(u+v)$ this yields that $u(x)-u(y)$ and $v(x)-v(y)$ are linearly dependent. Consequently, we have either $u(x)-v(x)=u(y)-v(y)=0$ or $u(x)-u(y)=-v(x)+v(y)$. In the case we have $w(x)-w(y)=0$ and therefore $Sw(x)=0$. This yields $Su(x)=Sv(x)=0$ and therefore $u(x)-u(y)=v(x)-v(y)=0$. Thus we have $u(x)=u(y)=v(y)=v(x)$. This contradicts that $x\not\in W$.\\

With an analogous proof we can show that if $x$ and $y$ fulfill \eqref{eq:ipmin2} we also have that $x\in W$. Hence we have $W=V$ and $u=v$. This finishes the proof. 
\end{proof}

\begin{theorem}\label{thm:Ip_conv_tight}
Let $g\colon U\to \R^m$ and let $u_p$ be extensions of $g$ minimizing $I_p$ for $p\in \Nat$. Then $(u_p)_p$ converges to the unique tight extension $u$ of $g$ in $\R^{\abs{V}}$ as $p\to \infty$
\end{theorem}
The proof comes from \cite[Theorem 1.3]{sheffield2012vector}. We added some details.
\begin{proof}
Since $G$ is connected the sequence $(u_p)_p$ is bounded. Therefore, it has a convergent subsequence $(u_{p_k})_k$ converging to some $\tilde{u}\colon V\to \R^m$ with $\tilde{u}=g$ on $U$. Let $u\colon V\to \R^m$ be the tight extension of $g$. Assume  that $\tilde{u}\neq u$. Then we have by Theorem \ref{thm:tight_ex} that $u$ is tighter than $\tilde{u}$. Define
\begin{equation}
\begin{aligned}
\tilde{K}\coloneqq& \max\{ S \tilde{u} (x) : x\in V\backslash U, S u (x) \neq S \tilde{u}(x) \}\\
>&\max\{ S u(x) : x\in V\backslash U, S u(x) \neq S \tilde{u}(x) \}\eqqcolon K
\end{aligned}
\end{equation}
and
\begin{equation}
Z\coloneqq U\cup \{ x\in V\backslash U: Su(x)>K\}
\end{equation}
In particular, we have $Su(x)=S\tilde{u}(x)$ on $Z\backslash U$. Further we define $G\mid_Z\coloneqq (Z,\{(x,y)\in E:x,y\in Z\},\omega\mid_{Z\times Z})$.\\

We show that $u\mid_Z$ is tight on $G\mid_Z$. For $x\in Z$ we have $S u(x)>K$. Thus there exists a $y\sim x$ with $\sqrt{\omega(x,y)}\abs{u(x)-u(y)}>K$. This yields that $S u(y)>K$ and $y\in Z$. Therefore, we have that $S(u\mid_Z)(x)=Su(x)$ for all $x\in Z$. Assume there is a tighter extension $w\colon Z\to \R^m$ of $g$ on $Z$. Then we define $w_\epsilon \colon V\to \R^m$ for $\epsilon>0$ by
\begin{equation}
w_\epsilon (x)=
\begin{cases}
(1-\epsilon)u(x) + \epsilon w(x) &$if $x\in Z,\\
u(x) &$if $ x\not\in Z.
\end{cases}
\end{equation}
Now we have for $x\in Z$ and $y \sim x$ with $y\not\in Z$ that for a sufficient small $\epsilon>0$ it holds
\begin{equation}\label{eq:wepsilonkleinerSv}
\begin{aligned}
\sqrt{\omega(x,y)}\abs{w_{\epsilon}(x)-w_{\epsilon}(y)}\leq &\sqrt{\omega(x,y)}\left((1-\epsilon)\abs{u(x)-u(y)} + \epsilon \abs{w(x)-u(y)}\right)\\
\leq &(1-\epsilon)K+\epsilon\sqrt{\omega(x,y)} \abs{w(x)-u(y)}<S u(x). 
\end{aligned}
\end{equation}
Because $w$ is tighter than $u\mid_Z$ on $G_Z$ we have that $S w_{\epsilon}(x)<S (u\mid_Z)(x)=S u(x)$ for sufficient small $\epsilon$ and there exists some $x_0\in Z$ such that 
\begin{equation}
\max_{x_0 \sim y\in Z} \sqrt{\omega (x_0,y)}\abs{w(x_0)-w(y)}<\max_{x_0 \sim y\in Z} \sqrt{\omega (x_0,y)}\abs{u(x_0)-u(y)}.
\end{equation}
Using \eqref{eq:wepsilonkleinerSv}, this yields $S w_{\epsilon}(x_0) < S u(x_0)$ for sufficient small $\epsilon$. Similar we get that $S w_\epsilon (x)< S u(x)$ for sufficient small $\epsilon$. For $x\not\in Z$ and $y\sim x$ with $y\in Z$ we have for sufficient small $\epsilon$ that
\begin{equation}
\begin{aligned}
\abs{w_\epsilon (x)- w_\epsilon(y)}\leq &(1-\epsilon) \abs{u(x)-u(y)} +\epsilon \abs{u(x)-w(y)}\\\leq &(1-\epsilon) K +\epsilon \abs{u(x)-w(y)}< Su(x_0).
\end{aligned}
\end{equation}
Altogether we get that $l(w_\epsilon)<l(u)$ for sufficient small $\epsilon$. This contradicts that $u$ is tight. Hence $u\mid_Z$ is tight on $G\mid_Z$. In particular, $l(u\mid_Z)$ is lexicographically minimal on $Z$\\

Since $S\tilde{u}=S u$ on $Z$ and because $S(\tilde{u}\mid_Z)(x)\leq S\tilde{u}(x)$ for all $x\in Z$, we have that $l(\tilde{u}\mid_Z)$ is lexicographically minimal. Hence $\tilde{u}\mid_Z$ is also tight on $G\mid_Z$. Due to the fact, that tight extensions are unique, it holds that $\tilde{u}=u$ on $Z$.\\

Define  
\begin{equation}
v_k(x)\coloneqq 
\begin{cases}
u_{p_k}(x) &$if $ x\in Z,\\
u(x) &$if $x\not\in Z.
\end{cases}
\end{equation}
Since $u=\tilde{u}$ on $Z$ and $u_{p_k}\to \tilde{u}$, it holds that $v_k\to u$ as $k\to \infty$ for all $x\in V$. \\

For $x\in Z\backslash U$ there exists $y\in V$ with $y\sim x$ such that
\begin{equation}\label{eq:Ip_conv_tight1}
\sqrt{\omega(x,y)} \abs{u_{p_k}(x)-u_{p_k}(y)}=\sqrt{\omega(x,y)} \abs{v_k(x)-v_k(y)}\to \sqrt{\omega(x,y)} \abs{u(x)-u(y)}>K.
\end{equation}
For $y\sim x$ with $x\in Z\backslash U$ and $y\not\in Z$ it holds
\begin{equation}
\sqrt{\omega(x,y)} \abs{v_k(x)-v_k(y)}\to \sqrt{\omega(x,y)} \abs{u(x)-u(y)} \leq K
\end{equation}
Hence we have $S v_k(x)\leq S u_{p_k}(x)$ for all $x\in Z\backslash U$.\\

Define $\delta=\tilde{K}-K$. Since $u_{p_k}\to \tilde{u}$ it holds for $x\in V\backslash U$ with $S\tilde{u}(x)=\tilde{K}$ that
\begin{equation}
S u_{p_k}(x)\geq \tilde{K}- \frac{\delta}{4}
\end{equation}
for sufficiently large $k$. Due to the definition of $Z$, it holds for $x\in V\backslash Z$ that $S u(x)\leq K$. Since $v_k\to u$ it holds for $x\in V\backslash Z$ that 
\begin{equation}
S v_k(x)<K+\frac{\delta}{2}=\tilde{K}-\frac{\delta}{2}
\end{equation}
for sufficiently large $k$. Using $S u_{p_k}(x)\geq S v_k(x)$ for $x\in Z\backslash U$ this yields for sufficiently large $k$ that
\begin{equation}\label{eq:tight_limit_Ip}
\begin{aligned}
I_{p_k} (u_{p_k}) - I_{p_k} (v_k)=& \sum_{x\in V\backslash Z} \left( (S u_{p_k})^{p_k} - (S v_k)^{p_k}\right)+ \sum_{x\in Z\backslash U} \left( (S u_{p_k})^{p_k} - (S v_k)^{p_k}\right)\\
\geq &  \sum_{x\in V\backslash Z} \left( (S u_{p_k})^{p_k} - (S v_k)^{p_k}\right)\\
= & \sum_{x\in V\backslash Z} (S u_{p_k})^{p_k} - \sum_{x\in V\backslash Z}(S v_k)^{p_k}\\
\geq &  \sum_{\substack{x\in V\backslash Z\\ S u(x)\neq S\tilde{u}(x)=\tilde{K}}} (S u_{p_k})^{p_k} - \sum_{x\in V\backslash Z}(S v_k)^{p_k}\\
\geq & \sum_{\substack{x\in V\backslash Z\\ S u(x)\neq S\tilde{u}(x)=\tilde{K}}} (S u_{p_k})^{p_k} - \abs{V\backslash Z}\left(\tilde{K}-\frac{\delta}{2}\right)^{p_k}.
\end{aligned}
\end{equation}

Due to the definition of $\tilde{K}$ there exists at least one $x\in V\backslash Z$ such that $S u(x)\neq S \tilde{u} (x)=\tilde{K}$. Thus we have by \eqref{eq:tight_limit_Ip} that for sufficiently large $k$ it holds
\begin{equation}\label{eq:tight_limit_Ip2}
\begin{aligned}
I_{p_k} (u_{p_k}) - I_{p_k} (v_k)\geq & \sum_{\substack{x\in V\backslash Z\\ S u(x)\neq S\tilde{u}(x)=\tilde{K}}} (S u_{p_k})^{p_k} - \abs{V\backslash Z}\left(\tilde{K}-\frac{\delta}{2}\right)^{p_k}\\
\geq & \left(\tilde{K}-\frac{\delta}{4}\right)^{p_k}- \abs{V\backslash Z}\left(\tilde{K}-\frac{\delta}{2}\right)^{p_k}\\
=& \left(\tilde{K}-\frac{\delta}{2}\right)^{p_k}\left(\left(\frac{\tilde{K}-\frac{\delta}{4}}{\tilde{K}-\frac{\delta}{2}}\right)^{p_k}-\abs{V\backslash Z}\right)
\end{aligned}
\end{equation}

Since $\frac{\tilde{K}-\frac{\delta}{4}}{\tilde{K}-\frac{\delta}{2}}>1$ we have $\left(\frac{\tilde{K}-\frac{\delta}{4}}{\tilde{K}-\frac{\delta}{2}}\right)^{p_k}-\abs{V\backslash Z}>0$ for sufficiently large $k$. Then \eqref{eq:tight_limit_Ip2} yields $I_{p_k}(u_{p_k})-I_{p_k}(v_k)>0$ for sufficiently large $k$. This is a contradiction to the assumption, that $u_{p_k}$ minimizes $I_{p_k}$ for all $k$.
\end{proof}

\begin{remark}\label{rem:logexp}
With an analogous proof we can show that the unique tight extension $u\colon V\to \R^m$ of $g\colon U\to \R^m$ is the limit of the minimizers $u_p\colon V\to \R^m$ of the functionals
\begin{equation}
\tilde{I_p} (u)=\frac{1}{p}\log\left(\sum_{x\in V \backslash U} e^{p S u (x)}\right)=\frac{1}{p}\log\exp \left(p (Su(x))_{x\in V\backslash U}\right)
\end{equation}
The asymptotic function $(\log\exp)^\infty$ of the $\log\exp$-function is given by the $\vecmax$-function. This yields for $x\in \R^{\abs{V\backslash U}}$ that
\begin{equation}
\lim\limits_{p\to\infty} \frac{1}{p} \log\exp (p x)=\vecmax (x).
\end{equation} 
In particular, we get that
\begin{equation}
\lim\limits_{p\to\infty} \tilde{I_p} (u)=\vecmax (S u(x))_{x\in V\backslash U}.
\end{equation}
\end{remark}

As mentioned in \cite{sheffield2012vector}, tightness is stronger than being discrete $\infty$-harmonic. Since there is no proof given, we did it ourself.

\begin{theorem}\label{thm:tight_implies_disc_infharm}
Let $g\colon U\to \R^m$ and let $v\colon V\to\R^m$ be its tight extension. Then $v$ is discrete $\infty$-harmonic.
\end{theorem}

\begin{proof}
Assume $v$ is not discrete $\infty$-harmonic. Then there exists some $x_0\in V\backslash U$ such that
\begin{equation}
v(x_0)\not\in \argmin_{a\in \R^m} \left\{ \max_{y\sim x_0} \left(\sqrt{\omega(x_0,y)}\abs{v(y)-a}\right)\right\}.
\end{equation}
We aim to construct a tighter function $v'\colon V\to \R^m$. Let 
\begin{equation}
a_0\in \argmin_{a\in \R^m} \left\{ \max_{y\sim x_0} \left(\sqrt{\omega(x_0,y)}\abs{v(y)-a}\right)\right\}.
\end{equation}
Define $v'$ by
\begin{equation}
v'(x)=
\begin{cases}
v(x)&$if $x\neq x_0\\
a_0&$if $x=x_0
\end{cases}.
\end{equation}
Then it holds for $x\in V\backslash \{ x_0\}$ with $x \not\sim x_0$ that $S v(x)=S v'(x)$. Further we have $S v(x_0)>S v'(x_0)$ by definition. 
Suppose we have $S v(x)< S v'(x)$ for $x_0\neq x\in V$ with $x \sim x_0$. Since it holds $\sqrt{\omega(x,y)} \abs{v(y)-v(x)}=\sqrt{\omega(x,y)} \abs{v'(y)-v'(x)}$ for $y\sim x$ with $y\neq x_0$ we have
\begin{equation}
x_0\in \argmax_{y\sim x}\left\{ \sqrt{\omega(x,y)}\abs{v'(y)-v'(x)}\right\}.
\end{equation}
Thus we can conclude $S v'(x)\leq S v'(x_0)$.\\
Now we get
\begin{equation}
\begin{aligned}
\max \left\{ S v(x): x\in V \text{ and } S v(x)> S v'(x) \right\}\geq S v(x_0) > S v'(x_0) \\
\geq \max \left\{ S v'(x): x\in V \text{ and } S v'(x)> S v(x) \right\}.
\end{aligned}
\end{equation}
This yields that $v'$ is tighter than $v$ and contradicts that $v$ is tight.
\end{proof}

\begin{remark}\label{rem:disc_inf_gleich_tight_reellwertig}
For $m=1$ Theorem \ref{thm:disc_infhar_gleich_dirichlet} and Theorem \ref{thm:dirichlet_unique} ensure that discrete $\infty$-harmonic extensions of a function $g\colon U\to \R$ are unique. Theorem \ref{thm:tight_implies_disc_infharm} implies that the unique tight extension of $g$ is this unique discrete $\infty$-harmonic extensions. Therefore, the definitions of tight and discrete $\infty$-harmonic coincide for $m=1$.\\
In particular, Algorithm \ref{alg:disc_infty_harm} converges to the tight extension.
\end{remark}

\subsection{A minimization algorithm for $I_s$}

Let $2\leq s<\infty$ $G=(V,E,\omega)$, $U\subset V$ and $g\colon U\to \R^m$ be defined as above. Since the minimizers of the functionals $I_s$ converge to the tight extension of $g$ we aim to minimize $I_s$. This section leads to a new minimization algorithm for $I_s$.

\begin{definition}
For $1\leq p <\infty$ and $x=(x_1,...,x_n)\in \R^{mn}$ with $x_i\in \R^m$ the \emph{$2$-$p$-Norm} is defined by
\begin{equation}
\Norm{x}{2,p}\coloneqq \left(\sum_{i=1}^n \abs{x_i}^p \right)^{\frac{1}{p}}.
\end{equation}
The \emph{$2$-$\infty$-Norm} is defined by
\begin{equation}
\Norm{x}{2,\infty}\coloneqq \max_{i=1,...,n} \abs{x_i}.
\end{equation}
\end{definition}

Let $V\backslash U=\{v_1,...,v_N\}$. Then we can rewrite an extension $u\colon V\to \R^m$ with $u=g$ on $U$ by $x\coloneqq (u(v_1),..., u(v_N))^T$. Using this representation of $u$ we rewrite the functional $I_p$. To do this we define for $i=1,...,N$ the neighbors $\{u_1,...,u_k\}\coloneqq \{v\in V:v\sim v_i\}$. Then we define the matrix $B_i\in \R^{k\times N}$, where the $j$-th row is defined by
\begin{equation}
\begin{cases}
\sqrt{\omega(v_i,u_j)}(e_i^T-e_l^T), &$if $v_l=u_j\in V\backslash U,\\
\sqrt{\omega(v_i,u_j)}e_i^T,&$if $u_j\in U
\end{cases}
\end{equation}
Then we define $A_i=Id_m \otimes B_i$, where $Id_m$ is the $m\times m$ identity matrix and $\otimes$ the tensor product. Further we define for $j=1,...,k$ the vector $c_j\in \R^m$ by
\begin{equation}
\begin{cases}
0, &$if $v_l=u_j\in V\backslash U,\\
\sqrt{\omega(v_i,u_j)}g(u_j),&$if $u_j\in U.
\end{cases}
\end{equation}
Then we define $b_i=(c_1,...,c_k)^T\in \R^{km}$.\\

Now we can rewrite $I_p$ for $x=(x_1,...,x_N)$ by
\begin{equation}
I_s(x)=\sum_{i=1}^N \Norm{A_ix_i+b_i}{2,\infty}^s.
\end{equation}
Instead of minimizing $I_s$ we minimize for $v=(v_1,...,v_N)$ the equivalent formulation
\begin{equation}
E_s(u,v)\coloneqq\sum_{i=1}^N \Norm{v_i}{2,\infty}^s \text{ subject to } A_i u - b_i - v_i=0\text{, } i=1,...,N.
\end{equation}
We rewrite the constraints by $A u-b-v=0$, where $A=\left(\begin{matrix}A_1\\ ... \\A_N\end{matrix}\right)$ and $b=(b_1,...,b_N)^T$.Then we can apply the alternative direction method of multipliers (ADMM) reads as the following:

\begin{algorithm}[H]
\caption{ADMM for $E_s$}
\label{alg:Es_ADMM}
\begin{algorithmic}
\State Given: $v^{(0)}$, $p^{(0)}$, $A_i$ and $b_i$ for $i=1,...,N$, $\gamma>0$.
\For {$j=0,1,2,...$}
	\State Step 1: $u^{(j+1)}\coloneqq\argmin_{u} \abs{A u - v^{(j)}-b+p^{(r)}}^2$.
	\State Step 2: $v^{(j+1)}\coloneqq\argmin_v \sum_{i=1}^N\Norm{v_i}{2,\infty}^s+\frac{\gamma}{2}\abs{A u^{(j+1)}-v-b+p^{(j)}}^2$.
	\State Step 3: $p^{(j+1)}\coloneqq A u^{j+1}-v^{(j+1)}-b+p^{(j)}$.
\EndFor
\end{algorithmic}
\end{algorithm}

\begin{remark} We consider Algorithm \ref{alg:Es_ADMM}.
\begin{enumerate}[(i)]
\item By \cite[Theorem 6.1]{BSS_prox_op} the $(u^{(j)}, v^{(j)})_j$ generated by Algorithm \ref{alg:Es_ADMM} converges to a minimizer of $E_s$. Therefore, $(u^{(j)})_j$ converges to a minimizer of $I_s$.
\item Since $\abs{A u^{(j+1)}-v-b+p^{(j)}}^2=\sum_{i=1}^N \abs{A_i u_i^{(j+1)}-v_i-b_i+p_i^{(j)}}^2$, we can split Step 2 and replace it by \begin{equation}
v_i^{(j+1)}=\argmin_{v_i} \Norm{v_i}{2,\infty}^s+\frac{\gamma}{2}\abs{A_i u_i^{(j+1)}-v_i-b_i+p_i^{(j)}}^2=\prox_{\frac{1}{\gamma}\psi}(A_i u^{(j+1)} - b_i +p^{(j)}),
\end{equation}
where $\psi(x)=\Norm{x}{2,\infty}^s$.
\item Step 1 and Step 3 can be computed efficiently. It remains to find an efficiently way to compute $\prox_{\frac{1}{\gamma}\psi}$.
\end{enumerate}
\end{remark}

\subsubsection*{Computation of $\prox_{\lambda \psi}$}

\begin{notation}
Let $x=(x_1,...,x_n)\in \R^{m n}$. We use the notations
\begin{equation}
\begin{aligned}
&I_0(x)\coloneqq \{i\in\{1,...,n\}:\abs{x_i}\neq 0\}\\
\text{and }&I_{max}\coloneqq \{i\in\{1,...,n\}:\abs{x_i}=\Norm{x}{2,\infty}\}.
\end{aligned}
\end{equation}
\end{notation}

\begin{lemma}
For $p\in [1,\infty]$he dual norm of $\Norm{\cdot}{2,p}$is given by $\Norm{\cdot}{2,q}$ with $\frac{1}{p}+\frac{1}{q}=1$.
\end{lemma}
\begin{proof}
The dual norm $\Norm{\cdot}{*}$ of $\Norm{\cdot}{2,p}$ is given by
\begin{equation}
\Norm{y}{*}=\sup_{\Norm{x}{2,p}\leq 1} \inner{y}{x}.
\end{equation}
Let $x=(x_1,...,x_n),y=(y_1,...,y_n)\in \R^{m n}$. Then we have by the Cauchy Schwarz inequality and Hölders inequality that
\begin{equation}\label{eq:csu_hoelder}
\inner{y}{x}=\sum_{i=1}^n \inner{y_i}{x_i}\leq \sum_{i=1}^n \abs{y_i} \abs{x_i} \leq \Norm{y}{2,p}\Norm{x}{2,q}.
\end{equation}
Hence we have $\Norm{y}{*}\leq \Norm{y}{2,q}$. For $x=\frac{y}{\Norm{y}{2,q}}$ we have equality in \eqref{eq:csu_hoelder}. Thus we have $\Norm{y}{*}\geq \Norm{y}{2,q}$. Since this holds for arbitrary $y\in \R^{m,n}$ we have $\Norm{\cdot}{*}=\Norm{\cdot}{2,q}$.
\end{proof}

\begin{lemma}\label{lem:subdiff_21}
Let $x=(x_1,...,x_n)\in \R^{m n}$. Then the subdifferential of the $2$-$1$-norm at $x$ is given by
\begin{equation}
\partial \Norm{\cdot}{2,1}(x)=
\begin{cases}
\left(\frac{x_i}{\abs{x_i}}\right)_{i=1}^n,&$if $x\neq 0,\\
B_{\Norm{\cdot}{2,\infty}}(0,1),&$if $x=0,
\end{cases}
\end{equation}
where we define
\begin{equation}
\frac{x_i}{\abs{x_i}}\coloneqq B_{\abs{\cdot}}(0,1) \quad \text{if }\abs{x_i}=0
\end{equation}
\end{lemma}
\begin{proof}
Due to Theorem \ref{thm:subdiff_norm} it is sufficient to show that for $x\neq 0$ holds
\begin{equation}\label{eq:subdiff_211}
\{p:\Norm{p}{2,\infty}=1, \inner{p}{x}=\Norm{x}{2,1}\}=\left(\frac{x_i}{\abs{x_i}}\right)_{i=1}^n.
\end{equation}
Since $x\neq 0$ we have that $y=(y_1,...,y_n)\in \left(\frac{x_i}{\abs{x_i}}\right)_{i=1}^n$ fulfills $\Norm{y}{2,\infty}=1$ and 
\begin{equation}
\inner{y}{x}=\sum_{i=1}^n\inner{y_i}{x_i}=\sum_{i=1}^n\frac{1}{\abs{x_i}}\inner{x_i}{x_i}=\sum_{i=1}^n \abs{x_i}=\Norm{x}{2,1}.
\end{equation}
This shows the inclusion ``$\supseteq$'' in \eqref{eq:subdiff_211}. Vice versa we have that for $q=(q_1,...,q_n)\in \{p:\Norm{p}{2,\infty}=1, \inner{p}{x}=\Norm{x}{2,1}\}$ it holds $q_i\in B_{\abs{\cdot}}(0,1)$ and
\begin{equation}
\Norm{x}{2,1}=\inner{q}{x}=\sum_{i=1}^n \inner{q_i}{x_i}\leq \sum_{i=1}^n \abs{q_i}\abs{x_i}\leq \sum_{i=1}^n \Norm{q_i}{2,\infty}\abs{x_i}=\Norm{x}{2,1}.
\end{equation}
Thus we have $\inner{q_i}{x_i}=\abs{q_i}\abs{x_i}$. This yields for $i\in\{1,...,n\}$ with $x_i \neq 0$ that $q_i=\alpha_i \frac{x_i}{\abs{x_i}}$ for some $\alpha_i\geq 0$. Further we have that for $i\in\{1,...,n\}$ with $x_i \neq 0$ that $\abs{q_i}=\Norm{q}{2,\infty}=1$ and therefore $\alpha_i=1$. This yields that $q\in \left(\frac{x_i}{\abs{x_i}}\right)_{i=1}^n$ and ``$\subseteq$'' in \eqref{eq:subdiff_211}.
\end{proof}

\begin{lemma}\label{lem:subdiff_2infty}
Let $x=(x_1,...,x_N)\in \R^{m n}$. Then the subdifferential of the $2$-$\infty$-norm at $x$ is given by
\begin{equation}
\partial\Norm{\cdot}{2,\infty}(x)=
\begin{cases}
\left\{\left(\delta_{i,I_{max}} \alpha_i \frac{x_i}{\abs{x_i}}\right)_{i=1}^n:\sum_{i\in I_{max}} \alpha_i=1, \alpha_i\geq0\right\},&$if $x\neq 0,\\
B_{\Norm{\cdot}{2,1}}(0,1),&$if $x=0,
\end{cases}
\end{equation}
where $\delta_{i,I_{max}}=1$ if $i\in I_{max}$ and $\delta_{i,I_{max}}=0$ if $i \not\in I_{max}$. Note that $\delta_{i,I_{max}}=0$ if $\abs{x_i}=0$.
\end{lemma}
\begin{proof}
Due to Theorem \ref{thm:subdiff_norm} it is sufficient to show that for $x\neq 0$ holds
\begin{equation}\label{eq:subdiff_2infty1}
\{p:\Norm{p}{2,1}=1, \inner{p}{x}=\Norm{x}{2,\infty}\}=\left\{\left(\delta_{i,I_{max}} \alpha_i \frac{x_i}{\abs{x_i}}\right)_{i=1}^n:\sum_{i\in I_{max}} \alpha_i=1, \alpha_i\geq0\right\}.
\end{equation}
Let $y=(y_1,...,y_n)\in\left\{\left(\delta_{i,I_{max}} \alpha_i \frac{x_i}{\abs{x_i}}\right)_{i=1}^n:\sum_{i\in I_{max}} \alpha_i=1, \alpha_i\geq0\right\}$. Then it holds 
\begin{equation}
\Norm{y}{2,1}=\sum_{i\in I_{max}} \alpha_i \abs{\frac{x_i}{\abs{x_i}}}=\sum_{i\in I_{max}} \alpha_i=1.
\end{equation}
Further it holds
\begin{equation}
\inner{y}{x}=\sum_{i=1}^n \inner{y_i}{x_i}=\sum_{i\in I_{max}} \alpha_i \inner{\frac{x_i}{\abs{x_i}}}{x_i}=\sum_{i\in I_{max}} \alpha_i \Norm{x}{2,\infty}.
\end{equation}
This yields ``$\supseteq$'' in \eqref{eq:subdiff_2infty1}. Vice versa we have that for $q=(q_1,...,q_n)\in\{p:\Norm{p}{2,1}=1, \inner{p}{x}=\Norm{x}{2,\infty}\}$ that
\begin{equation}
\Norm{x}{2,\infty}=\inner{q}{x}=\sum_{i=1}^n\inner{q_i}{x_i}\leq \sum_{i=1}^n\abs{q_i}\abs{x_i}\leq \sum_{i=1}^n\abs{q_i}\Norm{x}{2,\infty}=\Norm{q}{2,1}\Norm{x}{2,\infty}=\Norm{x}{2,\infty}.
\end{equation}
Therefore, we have $q_i=\alpha_i \frac{x_i}{\abs{x_i}}$ for some $\alpha_i\geq 0$. Since $\sum_{i=1}^n\abs{q_i}\abs{x_i}\leq \sum_{i=1}^n\abs{q_i}\Norm{x}{2,\infty}$ we have $\alpha_i=0$ for $i\not\in I_{max}$. Hence we have $q\in \left\{\left(\delta_{i,I_{max}} \alpha_i \frac{x_i}{\abs{x_i}}\right)_{i=1}^n:\sum_{i\in I_{max}} \alpha_i=1, \alpha_i\geq0\right\}$ and ``$\subseteq$'' in \eqref{eq:subdiff_2infty1}.
\end{proof}

Now we compute $\prox_{\lambda \psi}(x)$. Let $\lambda>0$. Due to Theorem \ref{thm:moreau} we have
\begin{equation}
\prox_{\lambda \psi}(x)=x-\prox_{\lambda \psi^* (\lambda^{-1} \cdot)}(x).
\end{equation}
Hence it is enough to compute $\prox_{\lambda \psi^* (\lambda^{-1} \cdot)}$. We have
\begin{equation}\label{eq:psi_stern}
\psi^*(p)=\sup_x\{\inner{p}{x}-\psi(x)\}=-\inf_x\{\Norm{x}{2,\infty}^s-\inner{p}{x}\}.
\end{equation}
It is easy to show that $x\mapsto \Norm{x}{2,\infty}^s-\inner{p}{x}$ is continuous and coercive. Hence the infimum in \eqref{eq:psi_stern} is a minimum. Now $x$ is a minimizer of $\Norm{x}{2,\infty}^s-\inner{p}{x}$ if and only if
\begin{equation}
0\in \partial \Norm{\cdot}{2,\infty}^s(x)-p.
\end{equation}
Due to Theorem \ref{thm:chain_subdiff} this is equivalent to 
\begin{equation}
p\in s\Norm{x}{2,\infty}^{s-1}\partial \Norm{\cdot}{2,\infty}(x).
\end{equation}
Hence Lemma \ref{lem:subdiff_2infty} yields that there exists $\sum_{i\in I_{max}} \alpha_i=1, \alpha_i\geq0$ such that
\begin{equation}\label{eq:psi_stern_rechnerei}
p= s\Norm{x}{2,\infty}^{s-1}\left(\delta_{i,I_max} \alpha_i \frac{x_i}{\abs{x_i}}\right).
\end{equation}
This yields
\begin{equation}
\Norm{p}{2,1}=s\Norm{x}{2,\infty}^{s-1} \quad \Leftrightarrow \quad \Norm{x}{2,\infty}=\left(\frac{\Norm{p}{2,1}}{s}\right)^{\frac{1}{s-1}}
\end{equation}
Therefore, we have for $i\in\{1,...,n\}$ with $p_i\neq 0$ that \eqref{eq:psi_stern_rechnerei} becomes
\begin{equation}
p_i=\Norm{p}{2,1}\alpha_i \frac{x_i}{\abs{x_i}}, \quad \abs{x_i}=\Norm{x}{2,\infty}, \quad \abs{p_i}=\Norm{p}{2,1} \alpha_i.
\end{equation}
Thus we get
\begin{equation}
x_i=\left(\frac{\Norm{p}{2,1}}{s}\right)^{\frac{1}{s-1}}\frac{p_i}{\abs{p_i}}.
\end{equation}
Now we can compute \eqref{eq:psi_stern}:
\begin{equation}
\begin{aligned}
\psi^*(p)&=\sum_{i\in I_0(p)} \inner{p_i}{\left(\frac{\Norm{p}{2,1}}{s}\right)^{\frac{1}{s-1}}\frac{p_i}{\abs{p_i}}}-\left(\frac{\Norm{p}{2,1}}{s}\right)^{\frac{s}{s-1}}\\
&= \left(\frac{\Norm{p}{2,1}}{s}\right)^{\frac{1}{s-1}} \Norm{p}{2,1} - \left(\frac{\Norm{p}{2,1}}{s}\right)^{\frac{s}{s-1}}\\
&=\left(\frac{1}{s}\right)^{\frac{1}{s-1}}\left(1-\frac{1}{s}\right)\Norm{p}{2,1}^{\frac{s}{s-1}}.
\end{aligned}
\end{equation}

Now we compute 
\begin{equation}
\begin{aligned}
\prox_{\lambda \psi^*(\lambda^{-1}\cdot)}(x)&=\argmin_y \left\{\frac{1}{2 \lambda} \abs{y-x}^2+\psi^*\left(\frac1\lambda y\right)\right\}\\
&=\lambda \argmin_y \left\{\frac{1}{2 \lambda} \abs{\lambda y-x}^2+\psi^*(y)\right\}
\end{aligned}
\end{equation}
For $y\neq 0$ we have that it is a minimizer of $\frac{1}{2 \lambda} \abs{\lambda y-x}^2+\psi^*(y)$ if and only if
\begin{equation}\label{eq:prox_psistern_rechnerei}
\begin{aligned}
&0 \in \frac{1}{\lambda} \lambda\left(\lambda y-x\right)+\partial \psi^* (y)\\
\Leftrightarrow \quad &0 \in \left(\lambda y-x\right)+\left(\frac{1}{s}\right)^{\frac{1}{s-1}}\Norm{y}{2,1}^{\frac{1}{s-1}}\partial \Norm{\cdot}{2,1}(y)\\
\Leftrightarrow \quad &x \in \lambda y+\left(\frac{\Norm{y}{2,1}}{s}\right)^{\frac{1}{s-1}} \left(\frac{y_i}{\abs{y_i}}\right)_{i=1}^n, \quad \frac{y_i}{\abs{y_i}}=B_{\abs{\cdot}}(0,1), \text{ if } i\not\in I_0(y).\\
\end{aligned}
\end{equation}
We define
\begin{equation}\label{eq:tau_prox_psistern}
\tau=\tau(y)\coloneqq \left(\frac{\Norm{y}{2,1}}{s}\right)^{\frac{1}{s-1}} \quad \Leftrightarrow \quad \Norm{y}{2,1}=s\tau^{s-1}.
\end{equation}
Thus \eqref{eq:prox_psistern_rechnerei} yields that
\begin{equation}
\begin{cases}
x_i=\lambda y_i + \tau \frac{y_i}{\abs{y_i}}\not\in \tau B_{\abs{\cdot}}(0,1), &$if $i \in I_0(y), \\
x_i\in \tau B_{\abs{\cdot}}(0,1), &$if $i\not\in I_0(y).
\end{cases}
\end{equation}
Therefore, we have $\frac{y_i}{\abs{y_i}}=\frac{y_i}{\abs{y_i}}$ if $y_i\neq0$. Thus $\prox_{\lambda \psi^*(\lambda^{-1}\cdot)}(x)=\lambda y$ fulfills
\begin{equation}\label{eq:prox_psistern_darstellung}
\lambda y_i=
\begin{cases}
x_i- \tau \frac{x_i}{\abs{x_i}}, &$if $\abs{x_i}>\tau \\
0, &$if $\abs{x_i}\leq \tau
\end{cases}
=x_i\left(1-\frac{\tau}{\abs{x_i}}\right)_+.
\end{equation}
Thus we get for $z=(z_1,...,z_n)=\prox_{\lambda \psi}(x)\in \R^{m n}$ that
\begin{equation}\label{eq:prox_psi_darstellung}
z_i=x_i-\lambda y_i=
\begin{cases}
\tau \frac{x_i}{\abs{x_i}}, &$if $\abs{x_i}>\tau \\
x_i, &$if $\abs{x_i}\leq \tau
\end{cases}.
\end{equation}

To find $\tau$ we assume without loss of generality that $x=(x_1,...,x_n)\in \R^{m n}$ fulfills
\begin{equation}
\abs{x_1}\geq...\geq\abs{x_n}.
\end{equation}
Then there exists some $K$ such that $y_1,...,y_K\neq0$and $y_{K+1}=,...,=y_n=0$. We write 
\begin{equation}
x_K\coloneqq (x_1,...,x_K)^T.
\end{equation}
Then we have due to \eqref{eq:prox_psistern_darstellung} that
\begin{equation}
\Norm{x_K}{2,1}=\Norm{\lambda y}{2,1}+K\tau.
\end{equation}
Now \eqref{eq:tau_prox_psistern} yields
\begin{equation}
0=s \lambda\tau^{s-1}+K\tau-\Norm{x_K}{2,1}.
\end{equation}

This leads to the following algorithm to compute $\prox_{\lambda \psi}(x)$.

\begin{algorithm}[H]
\caption{Computation of $\prox_{\lambda\psi}(x)$}
\label{alg:prox_psi}
\begin{algorithmic}
\State Given: $x=(x_1,...,x_n)\in \R^{m n}$, $2\leq s<\infty$ and $\lambda>0$.
\State Define $K=0$, $\tau=0$.
\While {$\abs{\{i\in\{1,...,n\}:\abs{x_i}>\tau\}}>K$}
	\State Step 1: Set $K=K+1$.
	\State Step 2: Find $\tau\geq0$ such that $0=s \lambda\tau^{s-1}+K\tau-\Norm{x_K}{2,1}$.
\EndWhile
\State Define $z=(z_1, ...,z_n)\in \R^{m n}$ by $z_i=
\begin{cases}
\tau \frac{x_i}{\abs{x_i}}, &$if $\abs{x_i}>\tau, \\
x_i, &$if $\abs{x_i}\leq \tau.
\end{cases}$
\State Return $z$.
\end{algorithmic}
\end{algorithm}

\begin{remark} \label{rem:prox_psi_precision}
We consider Algorithm \ref{alg:prox_psi}.
\begin{enumerate}[(i)]
\item The function $g(\tau)=s \lambda\tau^{s-1}+K\tau-\Norm{x_K}{2,1}$ is strictly monotone increasing with $g(0)\leq 0$. Hence Step 2 determines a unique $\tau\geq 0$. For example we can use a Newton scheme to determine $\tau$.
\item For big $s$ we have that the $\tau\geq0$ solving  $0=s \lambda\tau^{s-1}+K\tau-\Norm{x_K}{2,1}$ becomes smaller than $1$. This yields that for $s\to \infty$ the term $s\tau^{s-1}\lambda^{s-2}$ becomes close to $0$ and the solution of $\tau$ becomes up to a numerical error equal to $\frac{\Norm{x_K}{2,1}}{K}$. This means that for $K=1$ we have that $\tau=\frac{\Norm{x_K}{2,1}}{K}=\abs{x_1}$ is an approximative solution in Step 2. Since $\abs{x_1}\geq \abs{x_i}$ for $i=1,...,n$ the while loop stops and the algorithm returns $z=x$.\\
Because of these precision problems Algorithm \ref{alg:prox_psi} is not applicable for big $s$, although it is correct in theory. 
\end{enumerate}
\end{remark}

Now we can combine the Algorithms \ref{alg:Es_ADMM} and \ref{alg:prox_psi} to obtain a algorithm to minimize $I_s$.

\begin{algorithm}[H]
\caption{Minimization algorithm for $I_s$}
\label{alg:Is_min}
\begin{algorithmic}
\State Given: $v^{(0)}$, $p^{(0)}$, $A_i$ and $b_i$ for $i=1,...,N$, $\gamma>0$.
\For {$j=0,1,2,...$ until some suitable stopping criteria is fulfilled}
	\State Compute $u^{(j+1)}\coloneqq\argmin_{u} \abs{A u - v^{(j)}-b+p^{(r)}}^2$.
	\For {i=1,...,N}
		Compute $v_i^{(j+1)}=\prox_{\frac{1}{\gamma}\psi}(A_i u^{(j+1)}-b_i+p^{(j)})$ using Algorithm \ref{alg:prox_psi}.
	\EndFor
	\State Compute $p^{(j+1)}\coloneqq A u^{j+1}-v^{(j+1)}-b+p^{(j)}$.
\EndFor
\State Return $u^{(j)}$.
\end{algorithmic}
\end{algorithm}


\section{Numerical examples}\label{sec:numexamples}

We aim to use the previous methods for image inpainting. We have given a pixel grid $V=\{1,...,m\}\times \{1,...,n\}$ and a damaged image $g\colon U\to \R^m$ on a nonempty subset $U\subset V$. Our goal is to reconstruct the damaged image by extending $g$ to a function $f\colon V\to \R^m$. To preserve the natural structure of the image we generate a graph with vertex set $V$ and applicate the methods from Section \ref{sec:opt_Lip_ext_graphs}.\\

For more details on the implementation see \cite{Fachpraktikum}.\\

\begin{remark}[Approximation of the tight extension]
In the following we will assume that we can compute the tight extension of $g$ with respect to some graph $G=(V,E,\omega)$. To find such an approximation algorithm for $m>1$ is still an open question. Thus we can only compute the tight extension of $g$ in the case of gray-valued images using Algorithm \ref{alg:disc_infty_harm}. For $m>1$ we will use two alternatives instead of computing the tight extension:\\

In the real-valued case we can compute the discrete $\infty$-harmonic extension with Algorithm \ref{alg:disc_infty_harm}, which is tight by Remark \ref{rem:disc_inf_gleich_tight_reellwertig}. One intuitive approach is using this real-valued approximation algorithm componentwise. Although we will see in Example \ref{ex:comp_tight_neq_tight} that this componentwise tight extension is not the tight extension, we will use it later as an approximation of the tight extension.\\ 

The second approximation of the tight extension is minimizing the energy functional $I_s$ using Algorithm \ref{alg:Is_min} for $s$ as large as possible. Due to Remark \ref{rem:prox_psi_precision} the algorithm fails for too large $s$. Nevertheless this should give an good approximation due to Theorem \ref{thm:Ip_conv_tight}. Example \
\end{remark}

\begin{example}\label{ex:comp_tight_neq_tight}
In general, the componentwise tight extension is not tight. For an example we again use the graph $G=(V,E,\omega)$, the subset $U\subset V$ and $g\colon U\to \R^2$ from Example \ref{ex:disc_inf_harm_nicht_eind}. Recall that $V=\{v_1, ..., v_6\}$, $U=\{v_1,v_2,v_3\}$ and 
\begin{equation}
E=\{ (v_1,v_4), (v_2,v_5), (v_3,v_6), (v_4,v_5), (v_5,v_6), (v_4,v_6)\}
\end{equation}
with the weighting function $\omega(x,y)=1$ for $(x,y)\in E$. Then the function $g\colon U \to \R^2$ is defined by $v_1\mapsto (0,0)$, $v_2\mapsto (0,1)$ and $v_3\mapsto \left(\frac{\sqrt{3}}{2},\frac{1}{2}\right)$. The componentwise tight extension is given by $f_1\colon V\to \R^2$ defined by $f_1=g$ on $U$ and $f_1(v_4)=\left(\frac{\sqrt{3}}{6},\frac{1}{3}\right)$, $f_1(v_5)=\left(\frac{\sqrt{3}}{6},\frac{2}{3}\right)$ and $f_1(v_6)=\left(\frac{1}{\sqrt{3}},\frac{1}{2}\right)$. The tight extension is given by $f_2\colon V\to \R^2$ defined by $f_2=g$ on $U$ and $f_2(v_4)=\left(\frac{1}{2+2 \sqrt{3}},\frac{\sqrt{3}}{2+2 \sqrt{3}}\right)$, $f_2(v_5)=\left(\frac{1}{2+2 \sqrt{3}},\frac{1}{4} + \frac{\sqrt{3}}{4}\right)$ and $f_2(v_6)=\left(\frac{1}{2},\frac{1}{2}\right)$. See Figure \ref{fig:tight_comp_tight}.
\end{example}

\begin{figure}
\centering
\begin{subfigure}[t]{0.45\textwidth}
\centering
\begin{tikzpicture}[scale=3]
\filldraw (0,0) circle (1pt) node[align=center, below] {$v_1$} -- (0.2886751,0.3333333) circle (1pt) node[align=center, below] {$v_4$} -- (0.2886751,0.6666667) circle (1pt) node[align=center, below] {$v_5$} -- (0,1) circle (1pt) node[align=center, below] {$v_2$};
\filldraw (0.2886751,0.3333333) -- (0.5773503,0.5) circle (1pt) node[align=center, below] {$v_6$} -- (0.2886751,0.6666667);
\filldraw (0.5773503,0.5) -- (0.8660254,0.5) circle (1pt) node[align=center, below] {$v_3$};
\end{tikzpicture}
\caption{Componentwise tight extension}
\label{fig:tight_comp_tighta}
\end{subfigure}\hfill
\begin{subfigure}[t]{0.45\textwidth}
\centering
\begin{tikzpicture}[scale=3]
\filldraw (0,0) circle (1pt)  node[align=center, below] {$v_1$} -- (0.1830127,0.3169873) circle (1pt) node[align=center, below] {$v_4$} -- (0.1830127,0.6830127) circle (1pt) node[align=center, below] {$v_5$} -- (0,1) circle (1pt) node[align=center, below] {$v_2$};
\filldraw (0.1830127,0.3169873) -- (0.5,0.5) circle (1pt) node[align=center, below] {$v_6$} -- (0.1830127,0.6830127);
\filldraw (0.5,0.5) -- (0.8660254,0.5) circle (1pt) node[align=center, below] {$v_3$};
\end{tikzpicture}
\caption{Tight extension}
\label{fig:tight_comp_tightb}
\end{subfigure}\hfill
\caption{Tight and componentwise tight extension of $g\colon \{v_1,v_2,v_3\}\to \R^2$.}
\label{fig:tight_comp_tight}
\end{figure}
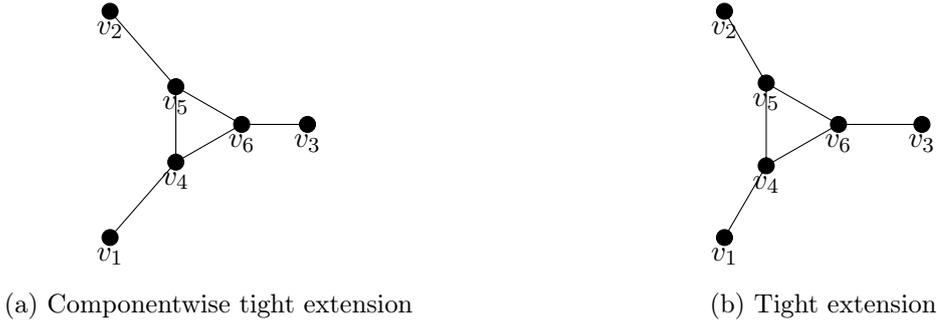

\begin{example}
We use $G$, $U$ and $g$ from Example \ref{ex:comp_tight_neq_tight}. Then the componentwise tight extension and tight extension are given by $f_1$ and $f_2$ from Example \ref{ex:comp_tight_neq_tight}.\\
Denote the output of Algorithm \ref{alg:Is_min} for $s=10$ by $f_3\colon V\to \R^2$. Then $\abs{f_3(v)-f_2(v)}<10^{-10}$ for all $v\in V$. This means that in this example Algorithm \ref{alg:Is_min} already for $s=10$ gives a very good approximation of the tight extension.
\end{example}

\subsection{Graph generation}

There are many methods to generate a graph on $V$. We consider the two methods implemented in \cite{Fachpraktikum}.\\

The \emph{Grid graphs} connect each pixel with its adjacent pixels by edges for local image inpainting. We obtain the 4-adjacency grid graph and the 8-adjacency grid graph with the weighting function $w (u,v)=1/dist(u,v)$, where $dist(u,v)$ is the Euclidean distance of the pixels $u$ and $v$.\\

The second method to generate the graph is the \emph{k-nearest neighborhood graph} with respect to patch similarity as weighting function for non-local image inpainting. Here we connect each vertex with its k-nearest neighbors with respect to the distance measure of patch similarity $s\colon V \times V \to \R\cup \{\infty\}$, which we will discuss below. This leads to a directed graph. Thus we symmetrize this directed graph by adding an edge between two vertices $u$ and $v$ if $u$ is one of the k-nearest neighbors of $v$ or vice versa.\\
As weighting function we use $w(u,v)=\exp(-s(u,v)/\sigma)$ for some $\sigma \in \R_{\geq 0}$, which scales the similarity.\\

\subsubsection*{Patch similarity}

Let $U$, $V$ and $g$ be defined as above and let $\Patch_k=\{-k,...,k\}\times \{-k,...,k\}$. To define a distance measure between two pixels, we consider for all pixels $u=(i,j)$ patches $u+\Patch_p =\{i-p,...,i+p\}\times \{j-p,...,j+p\}\cap V$ which contain all pixels in a $2p+1\times2p+1$ window around $u$ for some patch-size $p\in \Z_{>0}$. \\
Now we define for $u=(i,j)$ and $v=(k,l)\in V$ the similarity measure 
\begin{equation}
s_0(u,v)=
\begin{cases}
\frac{1}{\abs{A}}\sum_{(a,b)\in A}\Norm{g(i+a,j+b)-g(k+a,l+b)}{2}, $if $\abs{A} > \frac{(2p+1)^2}{4}\\
+\infty$, otherwise$
\end{cases}
\end{equation}
where $A=\{(a,b)\in \Patch_p : (i+a,j+b),(k+a,l+b)\in U\}$ i.e. $s_0(u,v)< +\infty$ if we can compare at least one fourth of the patches.\\
To reduce the computing effort we define for some radius $r\in \Z$
\begin{equation}
s(u,v)=\begin{cases} s_0(u,v)$, if $v\in u+\Patch_r \\ +\infty$, otherwise$\end{cases}
\end{equation}
instead of working with $s_0$.

\begin{remark}\label{rem:patch_similarity_probleme}
If for $u\in V\backslash U$ it holds that $u+P_p\subset V\backslash U$ then we have that $s(u,v)=+\infty$ for all $u\neq v\in V$. Thus $u$ is an isolated vertex in our graph and we cannot interpolate the value of $u$ with the methods from the previous sections. We will deal with this problem in the next section. 
\end{remark}

\subsection{An inpainting algorithm}

If we use grid graphs we can directly use the previous methods and approximate the tight extension $f\colon V\to \R^m$ of $g$. Since the tight extension is smooth along the edges the algorithm blurs the image into the unknown area $V\backslash U$. See Figure \ref{fig:grid_rechteck}. For more details on the examples see Section \ref{sec:examples}.\\

Due to Remark \ref{rem:patch_similarity_probleme} we cannot extend $g$ directly on $V$ by computing the tight extension if we use the k-nearest neighborhood graph with respect to patch similarity as weighting function. Therefore, we define
\begin{equation}
U'\coloneqq \{ v\in V:v\sim u \text{ for some } u\in U\}.
\end{equation}
Then we compute the tight extension $\tilde{g}\colon U'\to \R^m$ of $g$. Then we generate the k-nearest neighborhood graph with respect to $\tilde{g}$. This leads to the following algorithm.

\begin{algorithm}[H]
\caption{Image inpainting with k-nearest neighborhood graph}
\label{alg:inp_patch_sim}
\begin{algorithmic}
\State Given: $U\subset V$, $g:U \to \R^m$.
\State Define $U'\coloneqq U$ and $\tilde{g}\coloneqq g$.
\While {$U'\neq V$ or other stopping criteria}
	\State Generate the k-nearest neighborhood graph $G=(V,E,\omega)$ wrt. $\tilde{g}\colon U'\to \R^m$.
	\State Define $\overline{U}\coloneqq \{ v\in V:v\sim u \text{ for some } u\in U'\}$.
	\State Compute the tight extension $\overline{g}\colon \overline{U}\to \R^m$ of $g\colon U\to \R^m$.
\EndWhile
\end{algorithmic}
\end{algorithm}

\begin{remark}
We cannot show that Algorithm \ref{alg:inp_patch_sim} terminates or converges. 
But if the algorithm terminates and converges the result is the tight extension of $g$ with respect to the patch similarity graph of the reconstructed image. The results of our examples are convincing. See Figure \ref{fig:nonlocal_rechteck}. For more details on the examples see Section \ref{sec:examples}. 
\end{remark}

\subsection{Examples}\label{sec:examples}

In Figure \ref{fig:grid_rechteck} we start with the constructed image in Figure \ref{fig:rechteck_orig} with pixels $V=\{1,...,80\}\times \{1,...,80\}$ and destroy the image at the pixels $V\backslash U=\{21,...,60\}\times \{21,...,60\}$ to get the damaged image in Figure \ref{fig:rechteck_dama}. We construct the 4-adjacency grid graph and use Algorithm \ref{alg:disc_infty_harm} with $\tau=0.4$ to compute the componentwise tight extension in Figure \ref{fig:rechteck_grid_comptight}. For Figure \ref{fig:rechteck_grid_Ismin} we used Algorithm \ref{alg:Is_min} to minimize the energy functional $I_s$ with $s=40$.\\

\begin{figure}
\centering
\begin{subfigure}[t]{0.5\textwidth}
\centering
\includegraphics[width=0.95\textwidth]{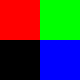}
\caption{original image}
\label{fig:rechteck_orig}
\end{subfigure}\hfill
\begin{subfigure}[t]{0.5\textwidth}
\centering
\includegraphics[width=0.95\textwidth]{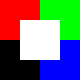}
\caption{Damaged image with $V\backslash U$ colored white}
\label{fig:rechteck_dama}
\end{subfigure}\hfill
\begin{subfigure}[t]{0.5\textwidth}
\centering
\includegraphics[width=0.95\textwidth]{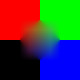}
\caption{Inpainted with componentwise tight extensions}
\label{fig:rechteck_grid_comptight}
\end{subfigure}\hfill
\begin{subfigure}[t]{0.5\textwidth}
\centering
\includegraphics[width=0.95\textwidth]{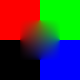}
\caption{Inpainted with minimization of $I_{22}$}
\label{fig:rechteck_grid_Ismin}
\end{subfigure}\hfill
\caption{Example for the approximation of tight extensions using grid graphs.}
\label{fig:grid_rechteck}
\end{figure}

In Figure \ref{fig:nonlocal_rechteck} we again reconstruct the damaged image in Figure \ref{fig:rechteck_dama}. Here we use Algorithm \ref{alg:inp_patch_sim}, where we create the graph with patch size $p=7$ and radius $p=15$. As scaling parameter we use $\sigma=0.1$. In Figure \ref{fig:rechteck_nonlocal_comptight} we use the componentwise tight extension and in Figure \ref{fig:rechteck_nonlocal_Ismin} we use the minimizer of the energy functional $I_s$ with $s=20$ as approximation of the tight extension.\\

\begin{figure}
\centering
\begin{subfigure}[t]{0.5\textwidth}
\centering
\includegraphics[width=0.95\textwidth]{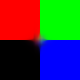}
\caption{damaged image with $V\backslash U$ colored white}
\label{fig:rechteck_nonlocal_comptight}
\end{subfigure}\hfill
\begin{subfigure}[t]{0.5\textwidth}
\centering
\includegraphics[width=0.95\textwidth]{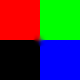}
\caption{Inpainted with minimization of $I_{20}$}
\label{fig:rechteck_nonlocal_Ismin}
\end{subfigure}\hfill
\caption{Example for usage of Algorithm \ref{alg:inp_patch_sim}.}
\label{fig:nonlocal_rechteck}
\end{figure}

In Figure \ref{fig:nonlocal_verschiedene} we use Algorithm \ref{alg:inp_patch_sim} with the same parameters as in Figure \ref{fig:nonlocal_rechteck} but with some other images.\\

\begin{figure}
\centering
\begin{subfigure}[t]{0.33\textwidth}
\centering
\includegraphics[width=0.95\textwidth]{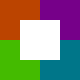}
\caption{Damaged image with $V\backslash U$ colored white}
\label{fig:rechteck2_withmask}
\end{subfigure}\hfill
\begin{subfigure}[t]{0.33\textwidth}
\centering
\includegraphics[width=0.95\textwidth]{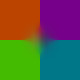}
\caption{Inpainted with componentwise tight extensions}
\label{fig:rechteck2_nonlocal_comptight}
\end{subfigure}\hfill
\begin{subfigure}[t]{0.33\textwidth}
\centering
\includegraphics[width=0.95\textwidth]{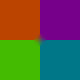}
\caption{Inpainted with minimization of $I_{20}$}
\label{fig:rechteck2_nonlocal_Ismin}
\end{subfigure}\hfill
\begin{subfigure}[t]{0.33\textwidth}
\centering
\includegraphics[width=0.95\textwidth]{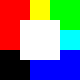}
\caption{Damaged image with $V\backslash U$ colored white}
\label{fig:test_withmask}
\end{subfigure}\hfill
\begin{subfigure}[t]{0.33\textwidth}
\centering
\includegraphics[width=0.95\textwidth]{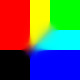}
\caption{Inpainted with componentwise tight extensions}
\label{fig:test_nonlocal_comptight}
\end{subfigure}\hfill
\begin{subfigure}[t]{0.33\textwidth}
\centering
\includegraphics[width=0.95\textwidth]{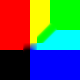}
\caption{Inpainted with minimization of $I_{20}$}
\label{fig:test_nonlocal_Ismin}
\end{subfigure}\hfill
\begin{subfigure}[t]{0.33\textwidth}
\centering
\includegraphics[width=0.95\textwidth]{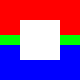}
\caption{Damaged image with $V\backslash U$ colored white}
\label{fig:test_withmask}
\end{subfigure}\hfill
\begin{subfigure}[t]{0.33\textwidth}
\centering
\includegraphics[width=0.95\textwidth]{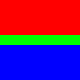}
\caption{Inpainted with componentwise tight extensions}
\label{fig:test_nonlocal_comptight}
\end{subfigure}\hfill
\begin{subfigure}[t]{0.33\textwidth}
\centering
\includegraphics[width=0.95\textwidth]{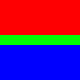}
\caption{Inpainted with minimization of $I_{20}$}
\label{fig:test_nonlocal_Ismin}
\end{subfigure}\hfil
\caption{Example for usage of Algorithm \ref{alg:inp_patch_sim}.}
\label{fig:nonlocal_verschiedene}
\end{figure}

We use the same methods for some natural image. In Figure \ref{fig:grid_ett} and Figure \ref{fig:nonlocal_ett} we use a $391\times 520$ image from \cite[Figure 8]{ETT}. In Figure \ref{fig:nonlocal_ett} we generate the graph with the parameters $\sigma=0.2$ and use $15\times 15$ patches in a $31\times 31$ neighborhood.\\

\begin{figure}
\centering
\begin{subfigure}[t]{0.5\textwidth}
\centering
\includegraphics[width=0.95\textwidth]{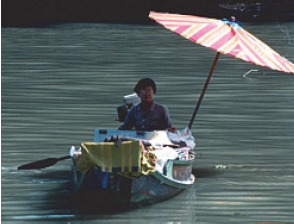}
\caption{original image}
\label{fig:ett_orig}
\end{subfigure}\hfill
\begin{subfigure}[t]{0.5\textwidth}
\centering
\includegraphics[width=0.95\textwidth]{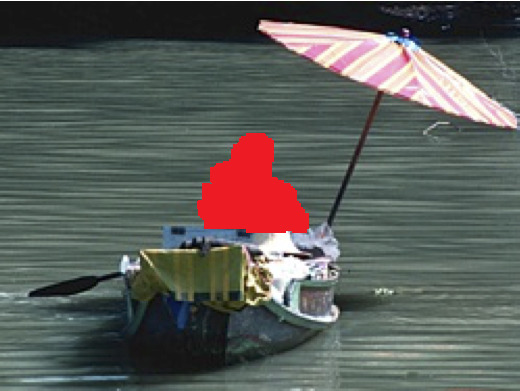}
\caption{Damaged image with $V\backslash U$ colored red}
\label{fig:ett_dama}
\end{subfigure}\hfill
\begin{subfigure}[t]{0.5\textwidth}
\centering
\includegraphics[width=0.95\textwidth]{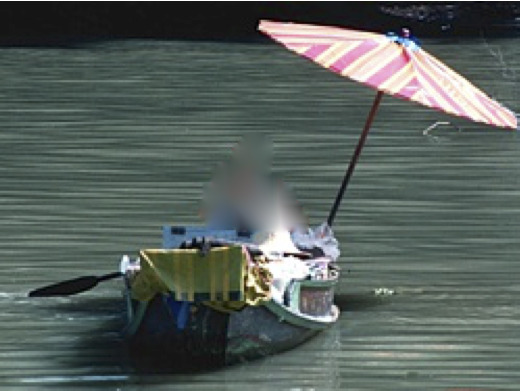}
\caption{Inpainted with componentwise tight extensions}
\label{fig:ett_grid_comptight}
\end{subfigure}\hfill
\begin{subfigure}[t]{0.5\textwidth}
\centering
\includegraphics[width=0.95\textwidth]{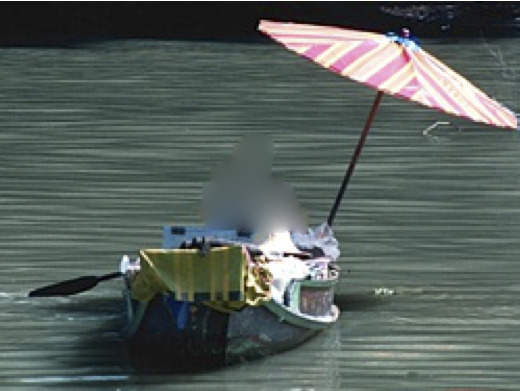}
\caption{Inpainted with minimization of $I_{20}$}
\label{fig:ett_grid_Ismin}
\end{subfigure}\hfill
\caption{Example for the approximation of tight extensions using grid graphs.}
\label{fig:grid_ett}
\end{figure}

\begin{figure}
\centering
\begin{subfigure}[t]{0.5\textwidth}
\centering
\includegraphics[width=0.95\textwidth]{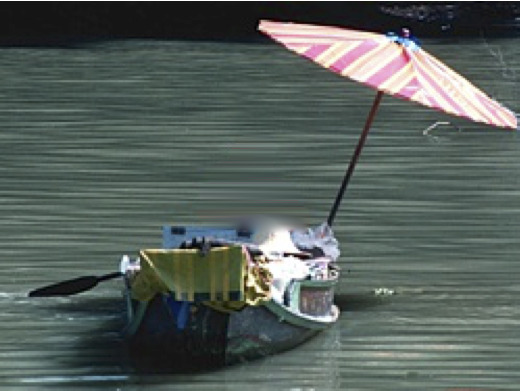}
\caption{Inpainted with componentwise tight extensions}
\label{fig:ett_nonlocal_comptight}
\end{subfigure}\hfill
\begin{subfigure}[t]{0.5\textwidth}
\centering
\includegraphics[width=0.95\textwidth]{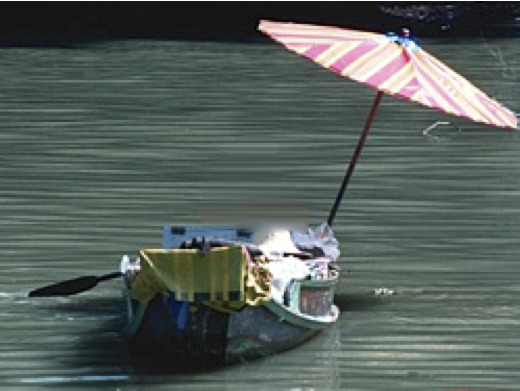}
\caption{Inpainted with minimization of $I_{20}$}
\label{fig:ett_nonlocal_Ismin}
\end{subfigure}\hfill
\caption{Example for usage of Algorithm \ref{alg:inp_patch_sim}.}
\label{fig:nonlocal_ett}
\end{figure}

In Figure \ref{fig:nonlocal_natural} we use Algorithm \ref{alg:inp_patch_sim} for some other natural images with minimization of the energy functional $I_{20}$ as approximation for the tight extension. We observe that inpainting of straight structures works very well, but other structures cannot be reconstructed. This is caused by the patch similarity as similarity measure. A patch around a pixel and the rotated patch have a small patch similarity measure although they are similar on a natural way.\\

\begin{figure}
\centering
\begin{subfigure}[t]{0.33\textwidth}
\centering
\includegraphics[width=0.95\textwidth]{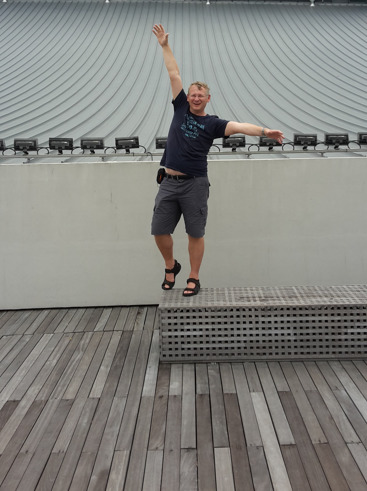}
\end{subfigure}\hfill
\begin{subfigure}[t]{0.33\textwidth}
\centering
\includegraphics[width=0.95\textwidth]{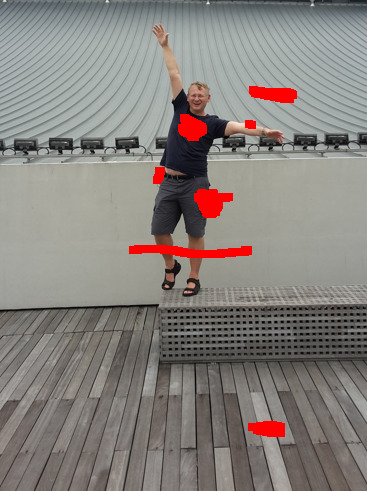}
\end{subfigure}\hfill
\begin{subfigure}[t]{0.33\textwidth}
\centering
\includegraphics[width=0.95\textwidth]{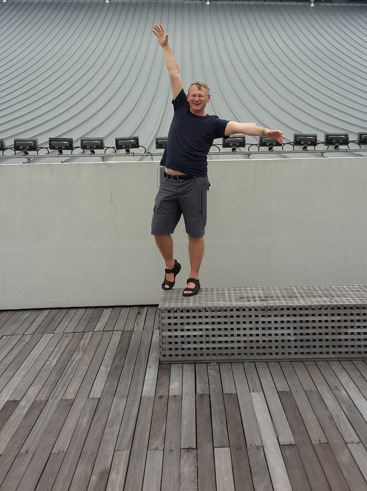}
\end{subfigure}\hfill
\vskip+2ex
\begin{subfigure}[t]{0.33\textwidth}
\centering
\includegraphics[width=0.95\textwidth]{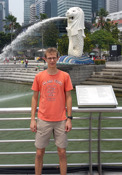}
\end{subfigure}\hfill
\begin{subfigure}[t]{0.33\textwidth}
\centering
\includegraphics[width=0.95\textwidth]{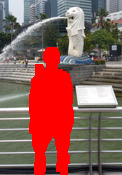}
\end{subfigure}\hfill
\begin{subfigure}[t]{0.33\textwidth}
\centering
\includegraphics[width=0.95\textwidth]{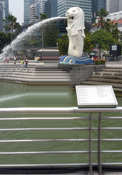}
\end{subfigure}
\vskip+2ex
\begin{subfigure}[t]{0.33\textwidth}
\centering
\includegraphics[width=0.95\textwidth]{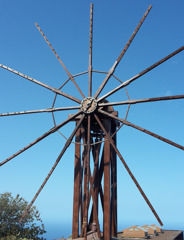}
\caption{Original image}
\end{subfigure}\hfill
\begin{subfigure}[t]{0.33\textwidth}
\centering
\includegraphics[width=0.95\textwidth]{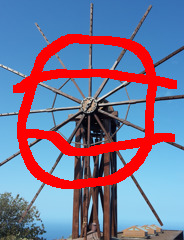}
\caption{Damaged image with $V\backslash U$ colored red}
\end{subfigure}\hfill
\begin{subfigure}[t]{0.33\textwidth}
\centering
\includegraphics[width=0.95\textwidth]{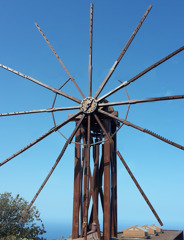}
\caption{Inpainted with minimization of $I_{20}$}
\end{subfigure}
\caption{Some natural image examples for usage of Algorithm \ref{alg:inp_patch_sim}.}
\label{fig:nonlocal_natural}
\end{figure}

In Figure \ref{fig:grid_rand} we compute the tight extension of a $243\times 182$ pixel image, where a pixel is known with probability $0.3$. We use a $8$-Neighborhood grid graph and compute the two approximations of the tight extension, which are mentioned above.

\begin{figure}
\centering
\begin{subfigure}[t]{0.5\textwidth}
\centering
\includegraphics[width=0.95\textwidth]{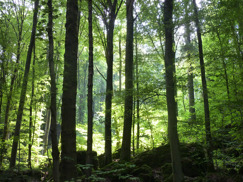}
\caption{Original image}
\label{fig:echt_orig}
\end{subfigure}\hfill
\begin{subfigure}[t]{0.5\textwidth}
\centering
\includegraphics[width=0.95\textwidth]{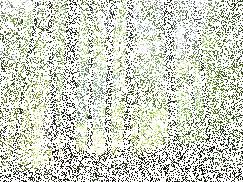}
\caption{Damaged image with $V\backslash U$ painted white}
\label{fig:echt_with}
\end{subfigure}\hfill
\begin{subfigure}[t]{0.5\textwidth}
\centering
\includegraphics[width=0.95\textwidth]{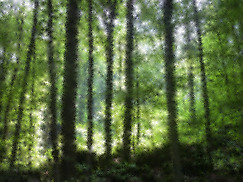}
\caption{Inpainted with componentwise tight extensions}
\label{fig:echt_Es}
\end{subfigure}\hfill
\begin{subfigure}[t]{0.5\textwidth}
\centering
\includegraphics[width=0.95\textwidth]{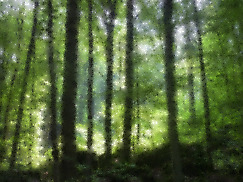}
\caption{Inpainted with minimization of $I_{22}$}
\label{fig:echt_with}
\end{subfigure}\hfill
\caption{Inpainting of a randomly damaged image using grid graphs.}
\label{fig:grid_rand}
\end{figure}

\begin{remark}[YUV color system]
In most cases we use the RGB color system. The components of the RGB values in natural images are in general highly correlated. Therefore we use an orthogonal transform to the so called YUV color system. In the YUV color system the components are in many cases more decoupled. For a vector $\left(\begin{array}{c}R\\G\\B\end{array}\right)\in [0,1]^3$ the transform from the RGB color system to the YUV color system is given by
\begin{equation}\label{yuvrgb}
\left(\begin{array}{c}
Y \\
U \\
V
\end{array} \right)
=A
\left(\begin{array}{c}
R \\
G \\
B
\end{array} \right),
\end{equation} 
where
\begin{equation}
A=
\left( \begin{array}{ccc}
0.299 & 0.587 & 0.114 \\
-0.14713 & -0.28886 & 0.436 \\
0.615 & -0.51498 & -0.10001
\end{array} \right)
\end{equation}
The transform from YUV to RGB is given by the inverse matrix.
\begin{equation}\label{yuvrgb}
\left(\begin{array}{c}
R \\
G \\
B
\end{array} \right)
=
A^{-1}
\left(\begin{array}{c}
Y \\
U \\
V
\end{array} \right) .
\end{equation} 
Since the matrix $A$ is not orthogonal and the generation of the $k$-nearest neighborhood graph depends on the scaling of the image, we additionally scale the transform by a factor $s>0$.
\begin{equation}
\left(\begin{array}{c}
Y \\
U \\
V
\end{array} \right)
=s A
\left(\begin{array}{c}
R \\
G \\
B
\end{array} \right),
\end{equation}
\end{remark}
 
In Figure \ref{fig:rechteck2_RGBYUV} we apply Algorithm \ref{alg:inp_patch_sim} on the image from Figure \ref{fig:rechteck2_withmask}. We use the same parameters as in Figure \ref{fig:nonlocal_rechteck}. We observe that in the YUV color system some edges are more smoothed and other edges are less smoothed than in the RGB color system. This is caused by the fact that the transformation of the RGB to the YUV color system is no isometry. Thus two colors with high distance in the RGB color system can have a small distance in the YUV color system and vice versa.

\begin{figure}
\centering
\begin{subfigure}[t]{0.5\textwidth}
\centering
\includegraphics[width=0.95\textwidth]{rechteck2_komponentenweise_nonlocal}
\caption{Inpainted with componentwise tight extensions in the RGB color system.}
\label{fig:rechteck2_komp_nichtYUV}
\end{subfigure}\hfill
\begin{subfigure}[t]{0.5\textwidth}
\centering
\includegraphics[width=0.95\textwidth]{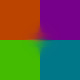}
\caption{Inpainted with componentwise tight extensions in the YUV color system.}
\label{fig:rechteck2_komp_YUV}
\end{subfigure}\hfill
\begin{subfigure}[t]{0.5\textwidth}
\centering
\includegraphics[width=0.95\textwidth]{rechteck2_Es_ADMM_nonlocal}
\caption{Inpainted with minimization of $I_{20}$ in the RGB color system.}
\label{fig:rechteck2_Es_nichtYUV}
\end{subfigure}\hfill
\begin{subfigure}[t]{0.5\textwidth}
\centering
\includegraphics[width=0.95\textwidth]{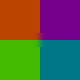}
\caption{Inpainted with minimization of $I_{20}$ in the YUV color system.}
\label{fig:rechteck2_Es_YUV}
\end{subfigure}\hfill
\caption{Comparison between the the usage of the RGB and YUV color system in Algorithm \ref{alg:inp_patch_sim}.}
\label{fig:rechteck2_RGBYUV}
\end{figure}


\section{Future work}

For the scalar case in Section \ref{sec:plap_inflap}, it is shown that there is a unique absolute minimal extension of Lipschitz continuous boundary values. In the vector valued case, we have given in Section \ref{sec:vec-valLip} the formulation of a tight function. In order to find a formulation for an ``optimal'' extension of Lipschitz continuous boundary values, it would be interesting to investigate existence and uniqueness of tight extensions $f\colon \R^d\supset \overline{\Omega}\to \R^m$ of Lipschitz continuous boundary values $g\colon \partial\Omega\to \R^m$.\\

In the discrete case in Section \ref{sec:opt_Lip_ext_graphs}, we have showed, that the formulations of tightness and discrete $\infty$-harmonic functions coincide in the scalar case. We formulated an approximation algorithm to find the unique discrete $\infty$-harmonic extension of a function $g\colon V\supset U \to \R$. For vector-valued functions $g\colon U \to \R^m$ we showed existence and uniqueness of a tight extension $f\colon V\to \R^m$ as the limit of the minimizers some energy functionals $I_p$ as $p\to \infty$. We gave an approximation algorithm to minimize these functionals $I_p$. To compute these minimizers we are stuck in numerical precision problems for big $p$. Can one resolve these precision problems? Maybe we find an algorithm to minimize the energy functionals from Remark \ref{rem:logexp} instead of the $I_p$.\\

Under the assumption, that we can compute the tight extensions on graphs, we formulated in Section \ref{sec:numexamples} an inpainting algorithm. It would be interesting to investigate, under which conditions Algorithm \ref{alg:inp_patch_sim} convergences.

\newpage

\bibliographystyle{abbrv}
\bibliography{database}

\end{document}